\documentclass[12pt]{amsart}
\usepackage[english]{babel}
\usepackage[pdftex]{graphicx}
\usepackage{wrapfig}
\usepackage{pdfsync}
\usepackage{epsfig}
\usepackage{epstopdf}
\usepackage{latexsym,amssymb,amsfonts,amsmath, amscd, euscript, MnSymbol}
\usepackage{enumerate}
\usepackage{verbatim}
\usepackage{float}

\newcommand{\N}{\mathbb{N}}

\newcommand{\Z}{\mathbb{Z}}

\newcommand{\concat}{%
  \mathord{
    \mathchoice
    {\raisebox{1ex}{\scalebox{.7}{$\frown$}}}
    {\raisebox{1ex}{\scalebox{.7}{$\frown$}}}
    {\raisebox{.7ex}{\scalebox{.5}{$\frown$}}}
    {\raisebox{.7ex}{\scalebox{.5}{$\frown$}}}
  }
}

\def\Power #1 { \powerset(#1) }
\def\Bidom #1 { {\mathfrak P} (#1) }

\newtheorem{definition}{{\bf Definition}}[section]
\newtheorem{theorem}[definition]{{\bf Theorem}}

\newtheorem{conjecture}[definition]{{\bf Conjecture}}
\newtheorem{corollary}[definition]{{\bf Corollary}}
\newtheorem{proposition}[definition]{\noindent {\bf Proposition}}
\newtheorem{lemma}[definition]{\noindent {\bf Lemma}}

\newtheorem{claim}[definition]{\noindent {\bf Claim}}
\newtheorem{question}[definition]{\noindent {\bf Question}}
\newtheorem{notation}[definition]{\noindent {\bf Notation}}

\newtheorem{remark}[definition]{\noindent {\bf Remark}}\newtheorem{problem}[definition]{\noindent {\bf Problem}}

\newtheorem{subclaim}[definition]{\noindent {\bf Subclaim}}

\def\proofref #1 {{\noindent  {\bf Proof} (#1).}\ }

\def\endproof{\hfill {\kern 6pt\penalty 500
\raise -0pt\hbox{\vrule \vbox to5pt {\hrule width 5pt
\vfill\hrule}\vrule}}}
\def\centerpicture #1 by #2 (#3){\leavevmode
        \vbox to #2{
        \hrule width #1 height 0pt depth 0pt
        \vfill
        \special{pictfile #3}}}
\title[Invariant subsets of scattered trees]{Invariant Subsets of Scattered Trees and the Tree Alternative Property of Bonato and Tardif}

\author[C.~Laflamme]{Claude Laflamme*} 
\address{Mathematics \& Statistics Department, University of Calgary, Calgary, Alberta, Canada T2N 1N4}
\email{laflamme@ucalgary.ca} 
\thanks{*Supported by NSERC of Canada Grant \# 10007490} 
\author[M.~Pouzet]{Maurice Pouzet**} \address{Univ. Lyon, Universit\'e Claude-Bernard Lyon1, CNRS UMR 5208, Institut Camille Jordan,  43 bd. 11 Novembre 1918, 69622 Villeurbanne Cedex, France and Mathematics \& Statistics Department, University of Calgary, Calgary, Alberta, Canada T2N 1N4}
\email{pouzet@univ-lyon1.fr }
 \thanks{**Research started while the author visited the Mathematics and Statistics Department of the University of 
Calgary in June 2012; the support provided is gratefully acknowledged.}
\author[N.~Sauer]{Norbert Sauer***}\thanks{***The third author was supported by NSERC of
Canada Grant \# 10007490}  \address{Mathematics \& Statistics Department, University of Calgary, Calgary, Alberta, Canada T2N 1N4}
\email{nsauer@ucalgary.ca }  

\date{Calgary May-June 2012 -- Version June 13, 2016}

\begin{document}


\begin{abstract}  
A tree is \emph{scattered} if it does not contain a subdivision of the complete
binary tree as a subtree. We show that every scattered tree contains a
vertex, an edge, or a set of at most two ends preserved by every embedding
of T. This extends results of Halin, Polat and Sabidussi.

Calling two trees equimorphic if each embeds in the other, we then
prove that either every tree that is equimorphic to a scattered tree $T$ is
isomorphic to $T$, or there are infinitely many pairwise non-isomorphic trees
which are equimorphic to T. This proves the tree alternative conjecture
of Bonato and Tardif for scattered trees, and a conjecture of Tyomkyn for
locally finite scattered trees.
\end{abstract} 
\maketitle


\section{Introduction} 

In this paper we deal with trees, alias connected graphs with no
cycle. We follow Diestel \cite{diestel} for most of the standard graph
theory notions.  A {\em ray} is a one-way infinite path.

We first generalizes theorems dealing with automorphisms of trees obtained
by Polat and Sabidussi \cite{polat-sabidussi}. They proved that every rayless tree has either a
vertex or an edge preserved by every automorphism, and that a scattered tree
having a set of at least three ends of maximal order contains a rayless tree
preserved by every automorphism. (Scattered trees are defined in the abstract
and at the end of Section 2. Ends are equivalence classes of rays and defined
in Section \ref{sec:ends}, just before Subsection \ref{subsection:order}. The order of an end, which we usually
call rank, is defined and discussed in Subsection \ref{space of ends}.)

Our first result pertains to scattered trees for which there is no end
of maximal order. We prove that if T is such a tree then it contains a
rayless tree preserved by every embedding (injective endomorphism) of
T. Together with a result of Halin \cite{halin}, who showed that every
rayless tree contains a vertex or an edge preserved by every
embedding, we obtain the first main contribution of this paper:

\begin{theorem}\label{main1} 
If a tree $T$ is scattered, then either there is one vertex, one edge,
or a set of at most two ends preserved by every embedding of $T$.
\end{theorem}

This is proved in Section \ref{section-polat-sabidussi} as a
consequence of Lemmas \ref{lem:raylesslimit}
and \ref{lem:Polat_Sabidussi_Halin}. Of course the fact that every end
in a scattered tree has an order (or rank), which has been proven by
Jung in \cite{jung}, is needed. However, for our proofs we need a more
structural development, with a finer differentiation of the types of
preserved elements. In this context we will also reprove Jung's
theorem (cf. Section 6.1).

The above results pertain to scattered trees with no end of maximal order or
at least 3 ends of maximal order. If there are only two ends of maximal order,
they must be preserved. However, the situation of two ends preserved by every
embedding is quite easy to characterize:

\begin{proposition} 
A tree $T$ has a set of two distinct ends preserved by every embedding
if and only if it contains a two-way infinite path preserved by every
embedding.
\end{proposition}

The case of an end preserved by every embedding is more
subtle. Indeed, this does not imply that the end contains a ray
preserved by every embedding.

We provide examples in Subsection \ref{subsection:examples} of
scattered trees with such an end but which do not contain an infinite
path nor a vertex nor an edge preserved by every embedding.

In order to present our next results, we introduce \footnote{We will
use an alternative presentation in Section \ref{section: Bonato-tardif-level function} based on the notion of
level function (see Subsection \ref{subsection:order}).} the following
notions. 

Let $T$ be a tree and $e$ be an end. We say that $e$ is preserved
\emph{forward}, resp. \emph{backward}, by an embedding $f$ of $T$ if
there is some ray $C\in e$ such that $f[C]\subseteq C$,
resp. $C\subseteq f[C]$. We say that $e$ is \emph{almost rigid} if it
is preserved backward and forward by every embedding, that is every
embedding fixes pointwise a cofinite subset of every ray belonging to
$e$.

We recall that two trees $T$ and $T'$ are \emph{equimorphic}, or {\em
twins}, if $T$ is isomorphic to an induced subtree of $T'$ and $T'$ is
isomorphic to an induced subtree of $T$.  As we will see (cf. Lemma
\ref {lem:sumoftrees}), if $C:= \{x_0, \dots x_n, \dots \}$ is a ray,
then $T$ decomposes as a sum of rooted trees indexed by $C$, that is
$T= \bigoplus _{x_i \in C} T_{x_i}$ where $T_{x_i}$ is the tree,
rooted at $x_i$, whose vertex set is the connected component of $x_i$
in $T\setminus \{x_{i-1}, x_{i+1}\}$ if $i\geq 1$ and in $T\setminus
\{ x_{i+1}\}$ if $i=0$. If the number of pairwise non-equimorphic
rooted trees $T_{x_i}$ is finite, we say that $C$ is
\emph{regular}. We say that an end is \emph{regular} if it contains
some regular ray (in which case all other rays that it contains are
regular).

\begin{theorem} \label{main1-2} 
Let $T$ be a tree. If $T$ is scattered and contains exactly one end
$e$ of maximal order, then $e$ is preserved forward by every
embedding. \\ 
If $T$ contains a regular end $e$ preserved forward by
every embedding, then $e$ contains some ray preserved by every
embedding provided that $T$ is scattered or $e$ is not almost rigid.
\end{theorem}

Theorem \ref{main1-2} will be proved in Subsection \ref {proof of
main1-2}; it follows from Corollary \ref{lem:part of thm1.3},
Proposition \ref{prop:one end almost rigid} and Proposition \ref
{Fact:finitdist}.

\begin{corollary} 
Let $T$ be a scattered tree. If $T$ contains a regular almost rigid
end $e$, then $e$ contains a ray fixed pointwise by every embedding. \\
In particular, the set of embeddings of $T$ has a common fixed point.
\end{corollary} 

Examples of trees containing an end preserved backward by every
embedding and no end preserved forward by every embedding were
obtained independently by Hamann \cite{hamann1} (see Example 5 in
Subsection \ref{subsection:examples}) and Lehner \cite{lehner} (see
Example 6 in Subsection \ref{subsection:examples}). They are
reproduced with their permission.

\begin{definition}\label{defin:stable}
A tree $T$ is called {\em stable} if either: 
\begin{enumerate}
\item There is a vertex or an edge or a two-way infinite path or a one-way infinite path  preserved by every embedding or an almost rigid end;
\item Or there exists a non-regular  end which is preserved forward by every embedding. 
\end{enumerate}
\end{definition}

Summarizing Theorems \ref{main1} and \ref{main1-2} we have. 
\begin{theorem}\label{main} 
Every scattered tree is stable. 
\end{theorem}

We will apply these results to conjectures of Bonato-Tardif and
Tyomkyn.

\begin{definition}\label{defin:treealternative}
The \emph{tree alternative property} holds for a tree $T$ if
either every tree equimorphic to $T$ is isomorphic to $T$,  or else there are
infinitely many pairwise non-isomorphic trees which are equimorphic to
$T$. 
\end{definition}

\begin{conjecture}[Bonato-Tardif \cite{Bo-Ta}] 
The tree alternative property holds for every tree.
\end{conjecture}

Bonato and Tardif \cite {Bo-Ta} proved that their conjecture holds for
rayless trees, and their result was extended to rayless graphs by Bonato
et al, see~\cite{Bo-al}.

Now let $twin(T)$ denote the set of twins of $T$ up to isomorphism. Note
that if $T$ is a tree for which every embedding is an automorphism,
that is surjective, then $\vert twin(T)\vert =1$; in particular this
is the case for any locally finite rooted tree. On the other hand a star $R$ with
infinitely many vertices is an example of a tree having embeddings
which are not automorphisms and with $\vert twin(R)\vert =1$.  Another
such example is the tree consisting of two disjoint stars with
infinitely many vertices whose roots are adjacent.

As this will be developed in Subsection \ref{subsection:order}, if
$e$ is an end in a tree $T$, there is an orientation of the edges of
$T$ which is the (oriented) covering relation of an ordering $\leq _e$
on the set $V(T)$ of vertices of $T$. Indeed, for every vertex $x\in
V(T)$ there is unique ray starting at $x$ and belonging to $e$;
denoting this ray by $e(x)$ we may set $x\leq_ey$ if $y\in e(x)$. If
$x$ is a vertex, let $T(\to x)$ be the induced rooted tree rooted at
$x$ with $V(T(\to x))=\{y\in V(T): y\leq_{e} x\}$.

We are now ready to state the tree alternative result for stable trees. 

\begin{theorem}\label{main2}
\begin{enumerate}[{ (i)}]
\item The tree alternative conjecture holds for stable trees.  In particular: 

\item If $T$ is a stable tree which does not contain a vertex 
or an edge preserved by every embedding or an almost rigid end and
which has a non-surjective embedding, then $\vert twin(T)\vert=\infty$
unless $T$ is the one-way infinite path.

\item If $T$ has a vertex or an edge or an almost rigid end preserved 
by every embedding and $T$ is locally finite,  then every embedding of
$T$ is an automorphism of $T$.

\item If $T$ has an almost rigid end,  then $\vert twin(T)\vert= 1$ 
if and only if $\vert twin(T(\to x))\vert =1$ for every vertex $x$. 
Otherwise $\vert twin(T)\vert= \infty$.

\item  $T$ has a non-regular and not almost rigid end preserved  
forward by every embedding, then $\vert twin(T)\vert\geq 2^{\aleph_0}$. 
 \end{enumerate}
\end{theorem}

Theorem \ref{main2} is proved in Sections $9$ with the help of results
of Section $8$. Note that Theorems~\ref{main2} and \ref{main} together
with Definition \ref{defin:stable} imply:

\begin{corollary}\label{main3}
Let $T$ be a scattered tree with $ \vert twin(T)\vert < 2^{\aleph_0}
$. Then there exists a vertex or an edge or a two-way infinite path or
a one-way infinite path or an almost rigid end preserved by every
embedding of $T$.
\end{corollary}

Tyomkyn \cite{tyomkyn} proved that the tree alternative property holds
for rooted trees and conjectured that $twin(T)$ is infinite at
the exception of the one-way infinite path, for every locally finite
tree $T$ which has a non-surjective embedding.  

We obtain from Theorem \ref{main} and from Theorem \ref{main2} we now
obtain the second main contribution of the paper:

\begin{theorem}\label{thm:main2} 
The tree alternative conjecture holds for scattered trees, and
Tyomkyn's conjecture holds for locally finite scattered trees.
\end{theorem}

Theorem \ref{main2} clearly holds for finite trees.  That Theorem
\ref{main} holds for finite trees can either be seen directly by
successively removing endpoints until an edge or a single vertex
remains, or by just applying the Polat, Sabidussi, Halin results
mentioned above.  Hence, in the process of proving Theorems \ref{main}
and \ref{main2}, we will mostly assume that our trees $T$ are
infinite.

The results of this paper have been presented in part at the workshop
on Homogeneous Structures, Banff, Nov. 8-13, 2015 and to the seminar
on Discrete Mathematics, Hamburg, Feb. 19. 2016. The paper benefited
of remarks from the audiences and we are pleased to thank them; in particular we
thank M.~Hamann and F.~Lehner for their examples of trees.

We further thank M.~Hamann for the information he provided on his recent
result generalizing Theorem \ref{main1}.  This generalization
\cite{hamann3} has two parts, and the first one reads as follows:

\begin{theorem} Let $G$ be a  monoid of  embeddings of a tree $T$. Then either: 
\begin{enumerate}
\item There is a vertex, an edge or a set of at most two ends preserved by each member of $G$;
\item Or $G$ contains a submonoid freely generated by two embeddings.
\end{enumerate}
\end{theorem}

Note that if $G$ is a group of automorphisms, the conclusion is
similar: in the second case $G$ contains a subgroup freely generated
by two automorphisms (see Theorem 2.8 and Theorem 3.1 of
\cite{hamann2}). For anterior versions of this result, see Tits
\cite{tits}, Pays and Valette \cite{pays-valette} and Woess
\cite{woess}.

In his second part, M.~Hamann proves that if the first case of the
alternative above does not hold, then $T$ contains a subdivision of
the binary tree. For doing so, he shows that there are infinitely many
embeddings with different directions (the \emph{direction} of an
embedding $f$ being an end preserved forward by $f$ with a positive
period (see Section \ref{section:ends} for the definition of
period)). Then, he observes that there are two embeddings $f$ and $g$
such that none fixes the direction of the other. From that, he builds
a subdivision of the binary tree.


We wish to warmly thank the referees for their truly valuable and appreciated contributions.

\section{Basic definitions}\label{section:definitions}

Let $T$ be a tree. For $x,y\in V(T)$ we write $x\sim y$ if $x$ is
adjacent to $y$.  An {\em embedding of $T$ into a tree $T'$} is an
injection $f$ of $V(T)$ to $V(T')$ for which $x\sim y$ if and only if
$f(x)\sim f(y)$ for all $x,y \in V(T)$. An {\em embedding of $T$} is
an embedding of $T$ into $T$.  We write $H\subseteq T$ if $H$ is a
subtree of $T$. That is if $H$ is connected and $V(H)\subseteq
V(T)$. If $H\subseteq T$ and $f$ is an embedding of $T$ then $f[H]$ is
the subtree of $T$ induced by the set $\{f(h): h\in V(H)\}$ of
vertices.  Hence an embedding $f$ is an isomorphism of $H$ to $f[H]$.
An embedding $f$ {\em preserves a subtree $H$} of $T$ if
$f[H]\subseteq H$. Note that if $H$ is finite or a two-way infinite
path and $f$ preserves $H$ then $f[H]=H$, in which case $f$ restricted
to $H$ is an automorphism of $H$.  The embedding $f$ {\em fixes} a
vertex $x$ if it preserves it, that is if $f(x)=x$. Two trees $T$ and
$R$ are \emph{isomorphic}, resp. {\em equimorphic} or {\em twins}, and
we set $T \simeq R$, resp. $T\equiv R$, if there is an isomorphism of
$T$ onto $R$, resp. an embedding of $T$ to $R$ and an embedding of $R$
to $T$.  A \emph{rooted tree} is a pair $(T, r)$ consisting of a tree
$T$ and a vertex $r\in V(T)$, the \emph{root}. If $(T,x)$ is a rooted
tree with root $x$ and $(R,y)$ is a rooted tree with root $y$, then an
embedding of $(T,x)$ to $(R,y)$ is an embedding of $T$ to $R$ which
maps $x$ to $y$. The definitions above extend to rooted trees.

Let $P$ be an ordered set (poset), that is a set equipped with an
order relation, denoted by $\leq$. We say that $P$ is an \emph{ordered
forest} if for each element $x\in P$, the set $\downarrow x:= \{y\in
P: y\leq x\}$ is totally ordered. If furthermore, two arbitrary
elements have a common lower bound, then $P$ is an \emph{ordered
tree}.

Rooted trees can be viewed as particular types of ordered
trees. Indeed, let $(T, r)$ be a rooted tree; if we order $T$ by
setting $x\leq y$ if $x$ is on the shortest path joining $r$ to $y$ we
get an ordered tree with least element $r$, such that $\downarrow x$
is finite for every $x\in T$. Conversely suppose that $P$ is an
ordered tree, with a least element $0$, such that $\downarrow x$ is
finite for every $x\in P$. Then $a)$ $P$ is a
\emph{meet-semilattice}, that is every pair of elements $x$ and $y$
has a \emph{meet} (a largest lower bound) that we denote by $x\wedge
y$ and $b)$ the unoriented graph of the covering relation of $P$ is a
tree (we recall that an element $y\in P$ \emph{covers} $x$, and we
note $x\lessdot y$, if $x<y$ and there is no element $z$ such that
$x<z<y$).  

In the sequel, we will rather consider the dual of the above order on
a rooted tree, we will denote it $\leq_r$ and we will denote by $P(T,
r)$ the resulting poset.

Let $2^{<\mathbb N}$ be the set of finite sequences with entries $0$
and $1$. For $s\in 2^{<\mathbb N}$ we denote by $s\concat 0$ and
$s\concat 1$ the sequences obtained by adding $0$ and $1$ on the right
of $s$; we also denote by $\Box$ the empty sequence. There are an
order and a graph structure on $2^{<\mathbb N}$.  Ordering
$2^{<\mathbb N}$ via the initial subsequence ordering, the least
element being the empty sequence, we obtain an ordered tree, the
\emph{binary ordered tree}. Considering the (undirected) covering
relation we get a tree, the \emph{ binary tree} also known as
\emph{dyadic tree}, which we denote by $\mathrm{T}_2$.  A tree $T$ is
\emph{scattered} (See Figure 1) if no subdivision of the binary tree $\mathrm{T}_2$
is embeddable into $T$.

\begin{figure}[H]\label{fig:scattered}
\centering \includegraphics[width=8cm]{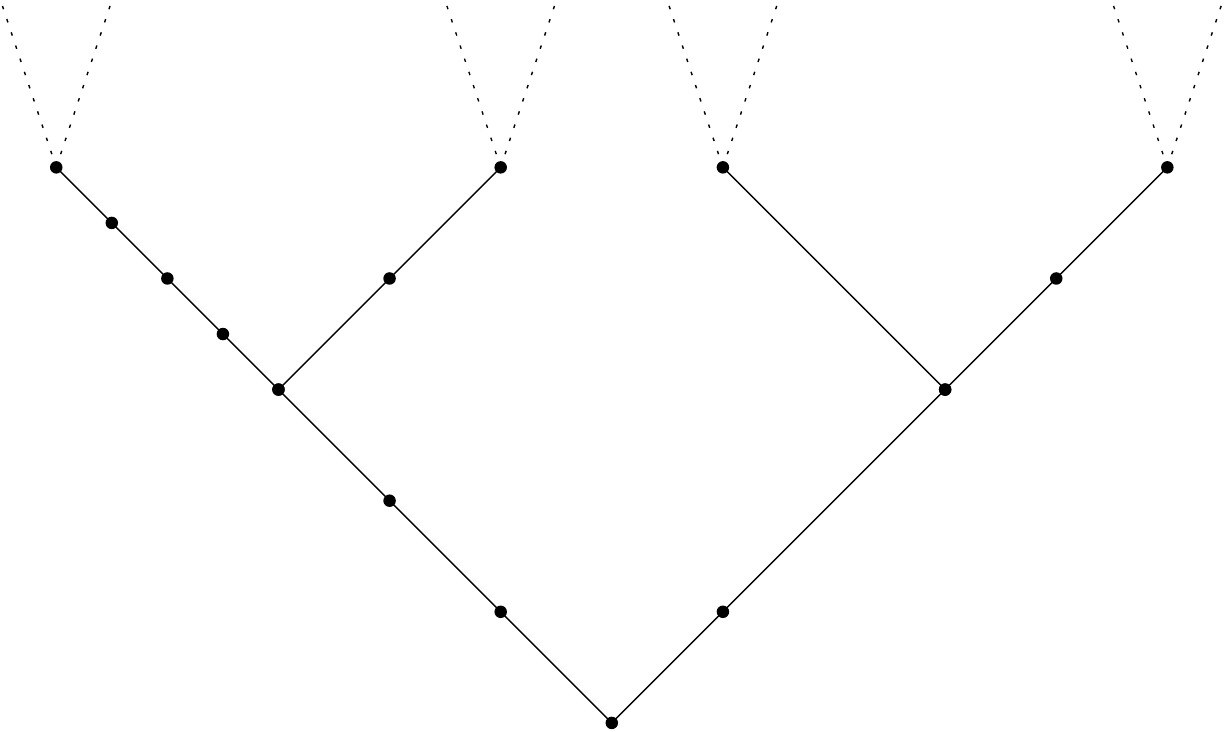}
\caption{A tree is \emph{scattered} if it does not embed a subdivision of the binary tree.}
\end{figure}

\section{Automorphisms and embeddings}

We recall that an automorphism $f$ of a tree $T$ is called a \emph{rotation}
if it fixes some vertex $x$; it is called  an \emph{inversion} if it reverses
an edge,  and it is called a \emph{translation} if it leaves invariant a
two-way infinite path. We recall the following basic result due to
Tits: 

\begin{theorem} \cite{tits} \label{theo:automorphism} 
Every automorphism of a tree is either a rotation, an inversion or a translation. 
\end{theorem}

The short and beautiful proof is worth  recalling. 

\begin{proof}
For two elements $x, y\in V(T)$, let $d_T(x,y)$ be the length of the
shortest path joining $x$ to $y$. Let $f$ be an automorphism of
$T$. Let $x\in V(T)$ with $m:= d_T(x, f(x))$ minimum.  If $f(x)=x$,
then $f$ is a rotation. If $f(x)\not = x$ and $f(f(x))=x$, then $m=1$
and $f$ is an inversion. If $f(f(x))\not =x$, then $C:= \bigcup_{n\in
\Z} f^{(n)}[C[x, f(x)]]$, where $C[x, f(x)]$ is the the shortest path
joining $x$ and $f(x)$, is a two-way infinite path left invariant by $f$.
\end{proof}

The case of embeddings is similar.  For this Halin  proved the following result:

\begin{theorem}\cite {halin3}\label{th:embeddingsoftrees} 
Let $f$ be an embedding of a tree $T$ into itself. Then either there is:
\begin{enumerate}
\item A fixed point; or 
\item An edge reversed by $f$; or 
\item  A ray $C$ with  $f[C] \subset C$.
\end{enumerate}

Furthermore, each case excludes the others. 
\end{theorem}

Case $(3)$ of the theorem above can be refined  into two parts as follows: 

\begin{proposition}\label{prop:embeddingsoftrees} Let $f$ be an embedding of  a tree $T$ into itself. Then either there is:
\begin{enumerate}
\item A fixed point; or 
\item An edge reversed by $f$; or 
\item A two-way infinite path preserved by $f$ on which $f$ fixes no vertex and reverses no edge; or 
\item A ray preserved by $f$ with a vertex not in the range of $f$. 
\end{enumerate}
Furthermore, each case excludes the others and the edge and path  in Cases (2) and  Case (3) are unique, whereas in Case (4), the ray  has a unique maximal extension to a ray preserved by $f$. 
\end{proposition}

 \section{Decomposition  of trees} \label{sec:decomposition}

The basic idea for the finite case has been  to study the effect of reducing the tree by removing the endpoints and study the effect of this operation on the embeddings, of course in this case the automorphisms, of the tree. If the tree $T$ is infinite, we will have to decompose the tree into larger pieces and then determine properties of the tree in relation to properties of the parts and the way those parts are put together to reconstitute  the tree.  Of course some notation, which will then be used throughout the paper, has to be introduced to describe those actions. In order to reduce the length of chains of symbols we will often, if we think that no confusion is possible, use the symbol for a tree to also stand for the set of its vertices and the set of vertices of an induced subtree of a tree to stand for the subtree. So for example if $C$ is a path in a tree, then we might write $x\in C$ and $P\subseteq C$ to indicate that $x$ is a vertex of $C$ and $P$ is a subpath of $C$ and so on.

Let $T$ be a tree and $X$ be a subset of $T$. We denote by $T\setminus X$ the subgraph of $T$ induced by $V(T)\setminus X$. If $K$ is a subtree of $T$ and $x$ a vertex of $K$, we denote by $(T,K)(x)$  the connected component of $x$ in  $T\setminus (K\setminus \{x\})$. Note that $(T,\{x\})(x)=T$ and that since $T$ is a tree,  $(T,K)(x)$ is the set of vertices on the paths of $T$ meeting $K$ at $x$. Note that $(T,K)(x)=\{x\}$ whenever the neighbourhood of $x$ is included in $K$. We denote by $ ((T,K)(x), x)$ the tree  $(T,K)(x)$ rooted at $x$.  The tree $T$ then consists of the tree $K$ with the rooted trees $ ((T,K)(x), x)$ attached at $x$ for every $x\in V(K)$. In general the trees $ ((T,K)(x), x)$ may be given  as a set of rooted trees of the form $(T_x,r_x)$. That is with a function associating the rooted trees to the vertices of $K$ and then we identify the root $r_x$ with $x$ and let the trees $(T_x,r_x)$ disjointly stick out of $K$.  

Here is the formal definition: Let $T$ be a tree, $X$ be a subset of
$V(T)$ and $(T_x, r_x)_{x\in X}$ be a family of rooted trees. The
\emph{sum} of this family is the tree denoted by $\bigoplus_{x\in X,
T} (T_x, r_x)$, or also simply $\bigoplus_{x\in X, T} (T_x, x)$, which
is obtained by first taking isomorphic copies $T'_x$ of those trees
which are pairwise disjoint and which contain $x$ instead of $r_x$ as
a root. With the understanding that if the trees $T_x$ are already
pairwise disjoint and contain $x$ as a root, then we will not take
different copies of them.  The set of vertices of the tree
$\bigoplus_{x\in X, T} (T_x, x)$ is the set $V(T)\cup \bigcup_{x\in
X}V(T'_x)$ and the edge set is $E(T)\cup \bigcup_{x\in X}E(T_x)$. If
$X=V(T)$ we simply denote this tree by $\bigoplus_{x\in T} (T_x,
r_x)$. In some special cases we will simplify even further. For
example if $\{(T_n,r_n): n\in \mathbb{N}\}$ is a set of rooted trees,
then $\bigoplus_{\mathbb{N}}T_n:=\bigoplus_{n\in \mathbb{N}}T_n$ with
the understanding that $\mathbb{N}$ also denotes the ray indexed
naturally by $\mathbb{N}$. As to be expected we then have:


\begin{lemma}\label{lem:sumoftrees} 
If $T$ is a tree and $K$ a subtree, then $T=\bigoplus_{x\in  K} ((T,K)(x), x)$.  
\end{lemma}
\begin{proof}
First, the vertex-sets of the trees ${(T, K)}(x)$ and $(T,K)(y)$ for
$x\not=y$ are disjoint and if there is an edge adjacent to a vertex in
${(T, K)}(x)$ and also adjacent to a vertex in $(T,K)(y)$, then it is
the edge $(x,y)$. Indeed, by construction, ${(T, K)}(x)\cap K=\{x\}$
for each $x\in K$. If $( T, K)(x)$ and $(T, K)(y)$ have a non-empty
intersection or there is an edge not equal to $(x,y)$ from $(T,K)(x)$
to $(T,K)(y)$, then the path from $x$ to $y$ would complete a cycle of
$T$.
\end{proof}

Let $T$ be a tree and $\emptyset \not=X\subseteq T$. For every vertex $x\in V(T)$, the \emph{distance from $x$ to $X$}, denoted by $d_T(x,X)$ is the least integer $n$ such that there is some path of length $n$ from $x$ to some vertex $y\in X$. Hence $d_T(x,X)=0$ iff $x\in X$. If $X$ induces a subtree, say $K$, this vertex $y$ is unique;  we denote it by $p_{(T, K)}(x)$. 

\begin{lemma}\label{lem:retraction}
If $K$ is a subtree of $T$, then the map $p_{(T, K)}$ from $T$ to $K$
is a retraction of reflexive graphs and $p^{-1}_{(T, K)}(y)= (T,K)(y)$
for every $y\in V(K)$.  
\end{lemma} 

\begin{proof} 
Clearly, $p_{(T, K)}(x)=x$ if and only if $x\in V(K)$. Hence  $p_{(T, K)}\circ p_{(T, K)}=p_{(T, K)}$ and $K$ is the range of $p_{(T, K)}$. This amounts to say that $p_{(T, K)}$ is a set-retraction of $V(T)$ onto $V(K)$. To conclude that it is a reflexive-graph-retraction, we need to prove that it transforms an edge into an edge or identifies its end vertices. For that, let  $x, x'\in V(T)$,  $y:=p_{(T, K)}(x)$ and $y':=p_{(T, K)}(x')$. Then, as it is easy to see, $d_T(x,x')= d_T(x,y)+d_T(y,y')+d_T(y',x')$. Hence, if $u:=\{x,x'\} \in E(T)$ and $\{y,y'\}$ is not an edge, then $y=y'$ as required. The fact that the vertices on the shortest path from a vertex $x$ to $K$ belong to $y:=p_{(T, K)}(x)$  ensures that $K_y:=((T,K)(y), y)$ is connected. It follows that $T$ is the sum $\bigoplus_{y\in  K}K_y$.   
\end{proof}

\begin{lemma} \label{lem:sumoverpath}Let  $T$ be a tree, $C$ and $C'$ be two infinite  paths of $T$  and $f$ be an  embedding of $C$ into $C'$. Then  $((T,C)(x), x)$ is embeddable into $((T,C')(f(x)), {f(x)})$ for every $x\in V(C)$ if and only if  $f$ has an extension $g$ to an embedding of  $T$ into itself such that $C=g^{-1}(C')$. 
\end{lemma}
\begin{proof} For the direct implication, choose for each $x\in V(C)$ an embedding $f_x$ of $((T,C)(x), x)$ into $((T,C')(f(x)), {f(x)})$. Observe that $g:= \bigcup_{x\in V(C)} f_x$ is an embedding of  $\bigoplus_{x\in  C} ((T,C)(x), x)$  into $\bigoplus_{y\in  C'} ((T,C')(y), y)$ and $C=g^{-1}(C')$. According to Lemma \ref{lem:sumoftrees} $\bigoplus_{x\in  C} ((T,C)(x), x)=T$, hence the implication holds. For the converse, note that if $g$ is an extension of $f$, then $g$ induces an embedding of $((T,C)(x), x)$  into $((T,C')(f(x)), {f(x)})$ for every vertex $x\in V(C)$ of degree $2$. Thus the converse holds if $C$ is a two-way infinite path.  If $C$ is  a one-way infinite path, i.e. a ray,   let $x_0$ be its  vertex of degree $1$. Due to the condition on the map $g$, it  induces an embedding of $((T,C)(x_0), {x_0})$  into $((T,C')(f(x_0)), {f(x_0)})$. 
\end{proof}  

\begin{corollary} \label{cor:sumoverpath} 
Let $f$ be an embedding of a tree $T$ which preserves a
two-way-infinite path $D$.  Then $f\big((T,D)(x),x)\big)\subseteq
((T,D)(f(x)),f(x))$ for all $x\in V(D)$.\\

If $f$ is an embedding which preserves a ray $C$ starting at $r_0$
with the predecessor of $f(r_0)$ in $C$ not in $V(f[T])$, then
$$f\big((T,C)(x),x)\big)\subseteq ((T,C)(f(x)),f(x))$$ 
for all $x\in V(C)$.
\end{corollary}

\section{Ends}\label{section:ends}\label{sec:ends}

Let $T$ be a tree. Two rays $C$ and $C'$ are called \emph{equivalent}
if the set $V(C)\cap V(C')$ of vertices induces a ray, equivalently if
$C\setminus F=C'\setminus F$ for some finite subset $F$ of
$V(T)$. This relation is an equivalence relation, with classes called
\emph{ends}; the equivalence class of a ray $C$ is denoted by
$end(C)$. The set of ends is denoted by $\Omega(T)$.

\begin{figure}[H]\label{fig:endtree}
\centering \includegraphics[width=6cm]{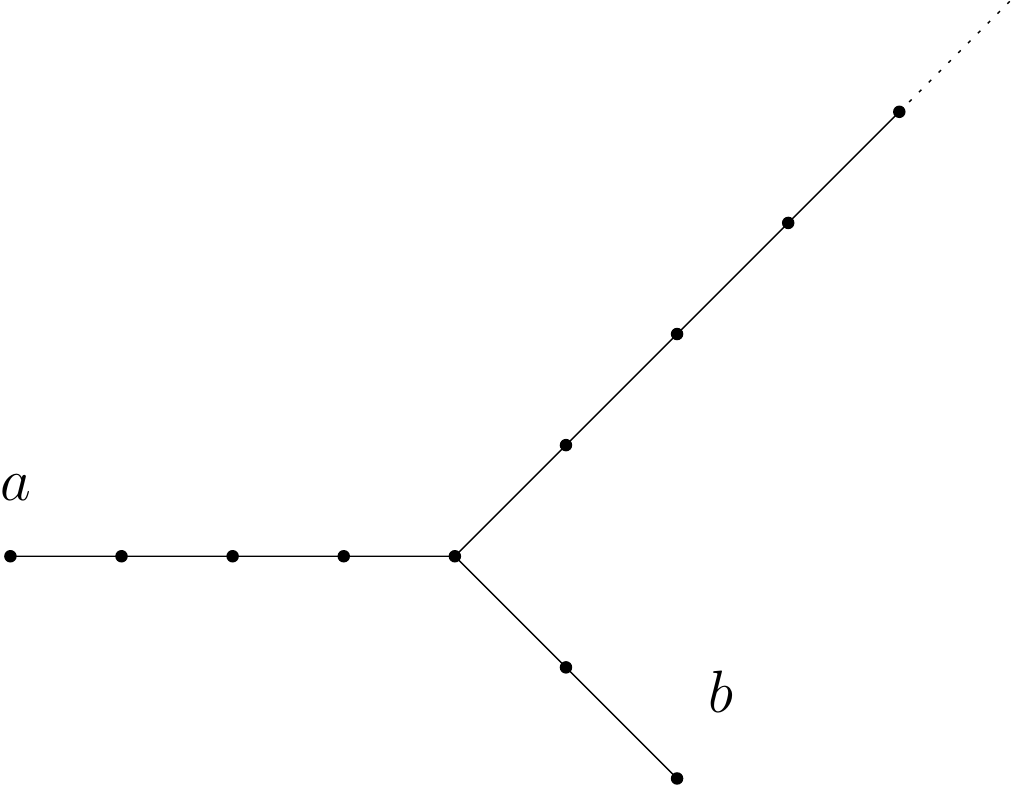}
\caption{An \emph{end} of a tree is an equivalence class of rays.}
\end{figure}

We note that since $T$ is tree, then for each end $e\in \Omega(T)$ and
$x\in V(T)$ there is a unique ray originating at $x$ and belonging to
$e$. We denote it by $e(x)$. For $n\in \N$ we denote by $x\boxplus n$
(or $x\boxplus_e n$ if there is a risk of confusion) the vertex of
$e(x)$ at distance $n$ of $x$. For $y\in e(x)$ we denote by
$T<e(x),y>$ the tree rooted at $y$ and whose vertex set is the
connected component of $y$ in $V(T) \setminus \{y^{-}, y^{+}\}$ where
$y^-$ and $y^+$ are the predecessor and successor of $y$ in $e(x)$. If
$y= x$, this rooted tree will be also denoted by $T(\rightarrow x)$,
its vertex set is the connected component of $x$ in $V(T) \setminus
\{x^{+}\}$. We note that $T= \bigoplus_{y\in e(x)}T<e(x), y>$.
 
To each end $e$ is attached an order $\leq_e$ and to $\Omega(T)$ a
topology. We discuss their properties below.

\subsection{Orders on a tree}\label{subsection:order}
Let $T$ be a tree and $K$ be a subtree. Let $x,y \in V(T)$; we set
$x\leq_K y$ if $y$ is on the shortest path joining $x$ to some vertex
of $K$.

This relation is an order on $V(T)$, $V(K)$ is the set of maximal
elements of $(T, \leq _K)$. For every edge $u:=\{x,y\}$ disjoint from
$K$ either $x<_Ky$ or $y<_Kx$. Indeed, since $K$ is connected, no path
containing $x$ and $y$ has its end vertices in $K$, hence either the
shortest path from $x$ to $K$ contains $y$ or the shortest path from
$y$ to $K$ contains $x$.  For each vertex $x\in V(K)$,
$(T,K)(x)=\{y\in V(T): y\leq_Kx\}$.  As a poset, $V(T)$ is the dual of
a forest. If $K$ reduces to a vertex $x\in V(T)$, we denote by
$P(T,x)$ the resulting poset; its dual satisfies properties a) and b)
given in Section \ref{section:definitions}.

Let $e\in \Omega(T)$. Let $x,y\in V(T)$. We set $x\leq_e y$ if $y\in
e(x)$.  This relation is an order. The unoriented covering relation of
this order is the adjacency relation on $T$. As a poset, $V(T)$
equipped with this order is the dual of an ordered tree; this dual is
a join-semilattice. We denote by $x\vee_ey $ the join of two vertices
$x,y\in V(T)$.

\begin{lemma} 
Let $f$ be an embedding of  a tree $T$ into itself. Then an end $e$, as an
equivalence class, is sent by $f$ into an equivalence class $e'$ of
$T$ and furthermore:

$f$ preserves the covering relation, (that is $x\lessdot_ey$ iff
$f(x)\lessdot_{e'} f(y)$ for all $x,y\in V(T)$). \\
In particular $f$ is
a join-semilattice embedding from $(T, \leq_e)$ into $(T, \leq_{e'})$.
\end{lemma}

\begin{proof}
Since $f(e)\subseteq e'$ we have $f(e(x))= e'(f(x))$ for every $x\in
V(T)$. Let $x,y\in V(T)$.  We have respectively $x\lessdot_ey$ iff
$x\sim y$ and $y\in e(x)$ and $f(x)\lessdot_{e'}f(y)$ iff $f(x)\sim
f(y)$ and $f(y)\in e'(f(x))$. Since $f(e)\subseteq e'$, $x\lessdot_ey$
implies trivially $f(x)\lessdot_{e'}f(y)$. Conversely, suppose that
$f(x)\lessdot_{e'}f(y)$.  Let $C\in e$. By deleting some elements of
$C$ we may suppose that $x,y\not \in C$. Let $z:=p_{(T, C)}(x)$, $C':=
f(C)$ and $z':= f(z)$. Since embeddings are isometric, we have
$z'=p_{(T, C')}(x)$. Since $f(x)\leq_{e'} f(y)$ and $C' \in e'$,
$f(y)$ is on the shortest path joining $f(x)$ to $C'$. This path
ending at $z'$ it is the image of the shortest path joining $x$ to
$z$, hence $y$ belongs to this path, proving that $x\leq_ey$; since
$f(x)\sim f(y)$, it follows that $x\sim y$ hence $x\lessdot_e y$. From
this, we have $f(x\vee_{e} y)=f(x)\vee_{e'}f(y)$ for all $x, y\in
V(T)$. Indeed, let $x,y\in V(T)$. Let $z:= f(x)\vee_{e'}f(y)$. Since
$f$ is order preserving, $z\leq_{e'} f(x\vee_{e} y)$, hence $z$ is
both on the shortest path joining $f(x)$ and $f(x\vee_{e} y)$ and the
shortest path joining $f(x)$ and $f(x\vee_{e} y)$. Since $f$ is a
one-to one isometry, $z= f(x\vee_{e} y)$. \end{proof}
 
\begin{corollary} \label{cor:joinembedding} 
An embedding $f$ of $T$ preserves $e$, that is $f[e]\subseteq e$, iff
$f$ is a join-semilattice embedding of $(T, \leq_e)$.
\end{corollary}

\subsection{Valuation,  level function and period}\label{subsection:valuation}


Let $T$ be a tree and $e$ be an end of $T$.  To each embedding $f$ preserving $e$ we attach an integer $d\in \Z$, the \emph{period} of $f$. In order to do so we introduce the notion of valuation. 

A $e$-\emph{valuation} is any map $v: V(T)\rightarrow \Z$ such that 
\begin{equation} 
x\lessdot_e y \Rightarrow v(y)= v(x)+1\end{equation}

 for every $x,y\in V(T)$. 
 
 \begin{lemma}\label{lem:valuation1} 
Let $e$ be an end of $T$. For
 every $x\in V(T)$, $n\in \Z$ there is a unique $e$-valuation $v$ such
 that $v(x)=n$. In particular two $e$-valuations differ by some
 constant.  
\end{lemma} 

\begin{proof} 
We order $V(T)$ by $\leq_e$. Let $y\in V(T)$. Set $v(y):=
n+d_T(x,z)-d_T(y,z)$ where $z:= x\vee_e y$.  Let $y'\lessdot_e
y$. Observe that $y'\vee x= y\vee x=z$. If follows that
$v(y)=v(y')+1$.  Hence $v$ is an $e$-valuation.  If $v'$ is any
valuation, we must have $v'(z)-v'(x)= d_T(x, z)$ and $v'(z)-v'(y)=
d_T(y, z)$.  Hence, $v'(y)= v'(x)+ d_T(x,z)-d_T(y,z)$. Thus $v'=v$ iff
$v'(x)=v(x)$. If $v$ is a $e$-valuation and $k\in \Z$, then $v+k$ is a
$e$-valuation.  Hence, if $v'$ is an other $e$-valuation, set $k:=
v(x)-v'(x)$. Then $v'+k$ is an $e$-valuation which coincides with $v$
on $x$. Hence $v=v'+k$.  
\end{proof}
 
From this lemma, it follows that if we choose a vertex $o\in V(T)$,
there is a unique $e$-valuation $v$ such that $v(o)= 0$. Such a vertex
is called the {\em origin} of the $e$-valuation.

%

\begin{lemma}\label{lem:valuation2} 
Let $e$ be an end of $T$, $v$ be an $e$-valuation on $T$ and $f$ be an
embedding. Then $f$ preserves $e$ iff there is some $d\in \Z$ so that
$v(f(x))=v(x)+d$ for all vertices $x$ of $T$. Moreover, $d$ is
non-negative iff $f$ preserves $e$ forward; $d$ is negative iff $f$
preserves $e$ backward and not forward.
\end{lemma}

\begin{proof} 
Set $v'(x):= v(f(x))$. If $f$ preserves $e$, then $v'$ is an
$e$-valuation. Hence from Lemma \ref {lem:valuation1}, $v'= v+d$ for
some $d\in \Z$. Conversely, suppose that $v(f(x))=v(x)+d$ for all
vertices $x$ of $T$.  Let $x,y\in V(T)$. We have $v(f(y))-v(f(x))=
v(y)-v(x)$. In particular, $v(f(y))-v(f(x))=1$ iff $v(y)-v(x)=1$, that
is $f(x)\lessdot_e f(y)$ iff $x\lessdot_e y$. Since $f$ preserves the
covering relation $\lessdot_e$, it preserves $e$ (Corollary
\ref{cor:joinembedding}). Suppose that $f$ preserves $e$ forward, then
there is some $C\in e$ such that $f(C)\subseteq C$. This implies
$x\leq_e f(x)$ for $x\in C$, hence $d\geq 0$. Similarly, if $f$
preserves $e$ backward there is some $C\in e$ such that $C\subseteq
f(C)$, hence $d\leq 0$. If $C\not =f(C)$, then $d\not =0$.  We prove
that if $f$ preserves $e$, then it preserves $e$ forward or backward
according to the sign of $d$.

\begin{claim} \label{claim:join}
Let $x\in V(T)$, $z:= x\vee_e f(x)$ and $k:= v(z)-v(x)-d$. Then $k\geq 0$ and $f(x\boxplus  k)= z$. 
\end{claim}
\noindent {\bf Proof of Claim \ref{claim:join}.}
Since $v(f(x))= v(x) +d$ we have $k= v(z)-v(f(x))$;  since $f(x)\leq_e z$ we have $k\geq 0$, hence $z':= x\boxplus k$ is well defined.  We have $v(z')-v(x)=v(f(z'))-v(f(x))= v(z)-v(f(x)=k$. Since $f(x)\leq_e f(z')$ and $f(x)\leq_e z$,  $f(z')$ and $f(z)$ are comparable. Since 
$v(f(z'))-v(f(x))= v(z)-v(f(x)=k$, $f(z')=z$. 
\hfill $\Box$

Pick any $x\in V(T)$. Let $z:= x\vee_ef(x)$ and $z':= x\boxplus  k$ . According to Claim \ref{claim:join}, $f(z')= z$. If $d\geq 0$, then  $z'\leq_e z$ and hence $f[e(z')]\subseteq e(z')$, proving that $e$ is preserved forward.   If $d=0$, then $z'=z$, and $e(z')$   is fixed pointwise by $f$. If  $d<0$, then $z<_e z'$ then $e(z')\subseteq f[e(z')]$ and $e$ is preserved downward. \end{proof}

The number $d$ is the \emph{period} of $f$ w.r.t. $v$. According to Lemma \ref {lem:valuation1} if $v'$ is an other $e$-valuation, then $f$ has the same period.  

An embedding $f$ \emph{ preserves a valuation} $v: V(T)\rightarrow \Z$  if $v(f(x))= v(x)$ for all $x\in V(T)$. In this case,  $f$ preserves every other valuation attached to the same end.  A map $v: V(T)\rightarrow \Z$ is a \emph{level function} if this is the  valuation associated with an end preserved forward by every embedding.  
\begin{lemma}\label{lem:valuation3} Let  $v$ be a valuation associated with an end $e$. Then $v$ is a level function iff every embedding of $T$ has a non-negative period. 
In particular,  $v$ is a level function provided that $e$ is almost rigid.\end{lemma}

\begin{proof}
The first assertion follows from  Lemma \ref{lem:valuation2}. The second assertion  is obvious.  
\end{proof}

\begin{lemma} \label{lem:Sf}
Let $T$ be a tree, $e$ be an end, $f$ be an embedding with
non-negative period $d$ and $S_f:= \{x\in V(T) : f(x)\in e(x)\}$. Then
\begin{equation}\label{eq:extensive1}
x\vee_e f(x)\in S_f
\end{equation}

and

\begin{equation}\label{eq:extensive2}
f(x\vee_e f(x))=(x\vee_e f(x))\boxplus d  
\end{equation}

\noindent for every $x\in V(T)$.

If $d=0$, then $S_f$ is the set of fixed points of $f$, whereas if
$d>0$, then $S_f$ is either a ray belonging to $e$ or a two-way
infinite path preserved by $f$.
\end{lemma}

\begin{proof}
Equations (\ref {eq:extensive1}) and (\ref {eq:extensive2}) are
equivalent.  The second implies trivially the first; since
$v(f(x\vee_e f(x))=v(x\vee_e f(x))+ d$ the second equation follows
from the first. To prove Equation (\ref {eq:extensive1}) apply Claim
\ref{claim:join}: there is some $z'\leq _ez$ such that $f(z')= z:=
x\vee_e f(x)$. This implies $z'\in S_f$ hence $z\in S_f$.

Since $S_f= \{x\in V(T): x\leq_{e}f(x)\}$ and $f$ preserves $\leq_e$,
it follows that $f$ preserves $S_f$ that is $f(S_f)\subseteq
S_f$. Since $f$ preserves $\leq_{e}$, $f$ preserves some ray, say $C$,
belonging to $e$, hence $C\subseteq S_f$ and thus $S_f$ is non-empty.
Suppose $d >0$.  We claim that $S_f$ is totally ordered. From this, it
follows that this is a path (a ray or a two-way infinite path). If
this is not the case, $S_f$ contains at least two incomparable
elements, say $x, y$.  Let $z:= x\vee_e y$. Then let $m:= d_T(x, y)=
d_T(x, z)+d_T(y, z)$. Since $f$ is an embedding $f(x)$ and $f(y)$ are
incomparable. This implies that $f(x)\vee_e f(y)=z$.  Since $f(x)=
x\boxplus d$ and $f(y)=y\boxplus d$, we have $d_T(f(x), f(y))= m-2d$
contradicting the fact that $f$ is an isometry.
\end{proof}

\subsection{Rays preserved by every embedding}

We present two results on regular ends, namely Proposition \ref{prop:one end almost
rigid} and Proposition \ref{Fact:finitdist}. The proofs are based on
the same idea, but the proof of the first one is much simpler.

\begin{proposition}\label{prop:one end almost rigid}
Le $T$ be a scattered tree and $e$ be an almost rigid regular
end. Then $e$ contains a ray preserved by every embedding of $T$.
\end{proposition}

\begin{proof} 
Our aim is to find $x\in V(T)$ such that $f(x)=x$ for every embedding
$f$. Indeed, since $e$ is almost rigid, all embeddings have period
zero (Lemma \ref{lem:valuation2}) hence $f(y)=y$ for every $y\in e(x)$
and every embedding $f$. In fact, we prove that under the weaker
assumption that $e$ is preserved forward by every embedding, then
there is some $x$ such that $f(x)\geq x$ for every embedding, from
which follows that the ray $e(x)$ is preserved by every embedding.

Pick $u\in V(T)$. Let $n\in \N$, set 
$x_n:=u\boxplus n$ in $e(u)$ and  $T_n:= T<e(u),x_n>$. 
%
%
%

\begin{claim}\label{claim:regular1}There is   $k\in \N$ such that for each $n\geq k$ the interval $[0, k[$ contains  two  integers $k'_n< k''_n$ such that $T_{k'_n}, T_{k''_n}$ and $T_{n}$ are equimorphic.\end{claim}

\noindent{\bf Proof of Claim \ref{claim:regular1}.}
The existence of $k$ is immediate: say that two integers $n$ and $m$ are equivalent if $T_n$ and $T_m$ are equimorphic (as rooted trees). The fact that  $e$ is regular means that  this  equivalence relation has only finitely many blocks. Pick two elements  in each equivalence class whenever possible and otherwise one element;  since the number of these  elements is finite,  some integer dominates the elements chosen. This integer has the required property.\hfill $\Box$

Let $k$ be given by   Claim \ref{claim:regular1}. We claim that $x_k\leq_e f(x_k)$  for every embedding $f$ of $T$.   Indeed, suppose not. Let $z:= x_k \vee_e f(x_k)$ in the join-semilattice $(T, \leq_e)$. Since $x_k \leq_e z$, $z\in e(u)$ hence there is some $n\in \N$ with $n\geq k$ such that  $z= x_n$. Let $k'_n, k''_n<k$ such that  $T_{k'_n}, T_{k''_n}$ and $T_{n}$ are equimorphic as rooted trees. Since $x_n= x_k \vee_e f(x_k)$, the rays $e(x_k)$ and $e(f(x_k))$ meet at $x_n$ (and not before), hence $f$ embeds  $T{(\rightarrow x_k)}$ into $T(\rightarrow f(x_k))$. In particular,  the  trees $T_{k'_n}$ and $T_{k''_n}$ embed into $T_{n}$. Set $r:= x_n$, $(R,r):= T_n$, $y:=f(u\boxplus k'_n)$, $x=f(u\boxplus k''_n)$, $C_y:= f[\{u\boxplus m:  k'_n\leq m\leq n\}$.  According to  Claim  \ref{claim:scatbin},  $T_{n}$ is non-scattered. 

\begin{claim} \label{claim:scatbin}  Let $(R,r)$ be a rooted  tree ordered by $u\leq_r v$ if $v$ is on the unique path $C_u$ joining $u$ to $r$. Let  $x, y\in V(R)$ such that  $y<_rx <_rr$, let  $R(C_y, x)$ be the tree rooted at $x$ whose vertex set is the  connected component of $x$ in $R\setminus \{x^-, x^+\}$ (where $x^-$ and $x^+$ are the neighbours of $x$ in $C_y$) and $R(\rightarrow y)$ be the tree rooted at $y$ whose vertex set is the connected component of $y$ in $R\setminus \{y+\}$.  If there exist  embeddings of the rooted tree  $(R,r)$  into the rooted trees $R(\to y)$ and $R(C_y, x)$,  then $R$ is not  scattered. 
\end{claim}
\noindent{\bf Proof of Claim \ref{claim:scatbin}.}  
Let $f$ be an embedding of $(R,r)$ into $R(C_y,x)$ and $g$ be an embedding of $(R, r)$ into $R(\to y)$. Then $f$ maps the path from $y$ to $r$ to a path from $z:=f(y)$  to $x$ and $g$ maps the path from $y$ to $r$ to the path from to $z'=g(y)$ to $y$.  Implying that there is an oriented path from  $z'=g(y)$ to $x$.  It follows that every copy $(R',r)$ of $(R,r)$ contains three distinct vertices $x, z$ and $z'$ with an oriented path from $z$ to $x$ excluding $z'$ and an oriented path from $z$ to $x$ excluding $z'$  and embeddings of $(R,r)$ into $R'(\to z)$ and into $R'(\to z')$. Implying that $(R,r)$ contains a subdivision of the binary tree. 
\hfill $\Box$ 

This concludes the proof of Proposition \ref{prop:one end almost rigid}.
\end{proof}

When we have embeddings with positive period, we do not need that $T$ is scattered, but   the proof is  more complex. 

Let $T$ be a tree and $e$ be an end. Let $Emb(T)$ be the set of
embeddings of $T$, $Emb_{e}(T)$ be the subset of those preserving $e$,
$Emb^+_{e}(T)$ be the subset of those with positive period and
$\mathbf d$ be the greatest common divisor of the periods of members
of $Emb^+_{e}(T)$. Let $u\in V(T)$ and $p\in \N$; we say that $p$ is a
\emph{period} of $e(u)$ if $T<e(u), y>\equiv T<e(u), y\boxplus p>$ for
every $y$ such that $u<_e y$.  If $e(u)$ has a positive period, then
there is one which divides all the others, we will call it the
\emph{period} of $e(u)$.

%
If  $f\in Emb_{e}^+(T)$ and $d$ is  the period of $f$ we set  $\check {S}_f:= \{x\in S_f:  d\; \text{is a period of}\; e(x)\}$. Furthermore, we set $\check S:= \bigcup _{f\in Emb_{e}^+(T)} \check {S}_f$. 

As we will see below, under the existence of embeddings with positive period, the regularity of an end amounts to the fact that  it contains some periodic ray.


%
%

\begin{lemma}\label{lem:period}
Let $T$ be a tree and  $e$  be a  regular end.  Then,  
\begin{enumerate}
\item for every embedding $f\in Emb_{e}^+(T)$, the set $\check {S}_f$ contains some $u$ such that $d$ is a period of $e(u)$;  
\item all   rays $e(u)$ for $u\in \check S$ have the same period; in particular this period divides ${\mathbf {d}}$; 
\item If every embedding of $T$ has a non negative period, then $\check S$ contains no two-way infinite path; in particular each $\check S_f$ is a ray. 
\end{enumerate}
\end{lemma}

\begin{proof} 
Item (1). Let  $f\in Emb_{e}^+(T)$ and  $u\in S_f$. Set 
$x_n:=u\boxplus n$ and  $T_n:= T<e(u),x_n>$ for $n\in \N$.  Say that two non-negative integers $n,m$ are equivalent if $T_n\equiv T_m$. Since $e$ is a regular end, $e(x)$ is a regular ray and this means that the equivalence relation above has only finitely many classes. Let $d$ be the period of $f$. Then $f(x_n)=x_n\boxplus d=x_{n+d}$ for every non-negative integer $n$. Since $f$ is an embedding of $T$,  $T_n$ is embeddable into $T_{n+d}$ for every $n>0$. Hence, the rooted trees $T_{n+ d.k}$, for $k\geq 0$, form an increasing sequence w.r.t embeddability. Since the number of equivalence classes is finite, there is some $k_n$ such that all $n+ d.k$, for $k\geq k_n$ are equivalent (pick $k_n$ such that $T_{n+ d.k_n} $ is  maximal with respect to equimorphy). This means that on the set $\{m\in \N: m\geq n+d.k_n\}$, the congruence class of $n$ modulo $d$ is included into the  equivalence class of $n+d.k_n$. By considering an upper bound of  $k_1, \dots k_d$, we get an integer $\ell$ such that on the set $\{m\in \N : m\geq \ell\}$
each congruence class is included into some equivalence class of our relation above. This means  $d$ is a period of $e(x_{\ell})$ hence $x_{\ell}\in \check  S_f$ and thus $e(x_{\ell})\subseteq \check  S_f$.

Item (2) Indeed, let $u, v\in \check S$,   $k$ and  $l$ be   periods  of $e(u)$ and $e(v)$ respectively.  Let   $d$ be the greatest common divisor of $k$ and $l$. The rays     $e(u)$ and $e(v)$ intersect on the ray $e(x\vee_e y)$, hence $k$ and $l$,   and thus $d$,  is a  period of that ray. But then $d$ is a period of $e(u)$ and $e(v)$. Taking for $k$ and $l$  the periods of $e(u)$ and $e(v)$, we get $k=l$,  proving our assertion. This common period must divide  the period of each $f\in Emb_{e}^+(T)$, hence it  divides ${\bf d}$, proving  that Item (2) holds. 

Item (3) Suppose that $\check S$ contains a two-way infinite path, say $D$. Let $y\in D$; pick $u\in D$ with $u<_e y$. Since $D\subseteq \check S$ there is some $f \in Emb_{e}^+(T)$ such that $u\in \check S_f$. According to Item (2) we have  $T<D, y>=T<e(u), y>\equiv T<e(u), y\boxplus {\bf {d}}>= T<D, y\boxplus {\bf {d}}>$. Since $T=\bigoplus_{y\in D} T<D,y>$, every translation of $D$ with period $\bf{d}$ extend to an embedding of $T$.  In particular,  the translation $t_{-\bf{d}}$ on $D$ defined by $t_{-{\bf d}}(x)\boxplus {\bf{d}}=x$ for $x\in D$ extends to an embedding of $T$, hence there are embeddings with  negative period which preserve $e$. \end{proof}

%
%
%
%

 We will denote by $\check E$ the set of end points of $\check S$. To avoid confusion, we will use $\check {S}(T)$ and $\check {E}(T)$ when needed.

\begin{proposition}\label{Fact:finitdist} 
Le $T$ be a tree admitting a regular and not almost rigid end $e$
preserved forward by every embedding.  Then the intersection $C_o$ of
all rays $e(u)$ for $u\in \check{E}(T)$ is a ray preserved by every
embedding and the distance from the origin $o$ of this ray to each $u\in
\check{E}(T)$ is at most $\mathbf{d}$, where $\mathbf{d}$ is the
largest common divisor of the positive periods of the embeddings of
$T$.
\end{proposition} 

\begin{proof}
We prove that $\check {E}$ has a join, say $o$, and that the distance from $o$ to every $u\in \check {E}$ is at most ${\bf d}$.  This is a consequence of the following claim.

\begin{claim}
Let $u\in \check {E}$ and   $x:= u\boxplus{\bf{d}}$. Then  $x\leq_ef (x)$ for every  $f\in Emb_{e}^+(T)$. 
\end{claim}

From this claim,  we have $v\leq x$ for every $v\in \check {E}$, from which it follows that $ \check {E}$  has a join, and this join $o$  is majorized by $x$. Indeed, let $v\in \check {E}$, let $g \in Emb_{e}^+(T)$ such that $v$ is the origin of $\check {S}_g$. From the claim, we have  $x\leq_e g(x)$ hence $x\in S_g$. The ray  $e(x)$ is periodic and its period divides ${\bf d}$, hence it divides the period of $g$ thus  $x\in \check S_g$ hence $u\leq_ex$. Since  $u\leq o\leq x= u+\bf{d}$ the distance from $u$ to $o$ is at most ${\bf d}$.

We prove the claim by contradiction. We suppose $x\not \leq_ef (x)$
for some $f\in Emb_{e}^+(T)$.  We build a tree $T'$ equimorphic to $T$
which is of the form $T':= \oplus_{n\in \Z} T'_n$ and on which act all
translations of $\Z$ of period $k.{\mathbf{d}}$ for $k\in \Z$. From
this, we obtain that there are embeddings of $T$ with negative period,
contradicting our hypothesis.

Let $C:= e(u)$.  According to (2) of Lemma \ref{lem:period}, $C$ is periodic and its period divides ${\bf {d}}$. As in Lemma \ref{lem:period} set  
$x_n:=u\boxplus n$ and  $T_n:= T<e(u),x_n>$ for $n\in\N$.
For  $n\in \Z$, let $T'_n$ be an isomorphic copy rooted at $n$ of the rooted tree $T_{m}$ where $m:=1+ (n\; \mod {\bf d})$ and  $n\;  \mod {\bf d}$ is the residue of $n$ modulo ${\bf d}$. On the tree $T':= \oplus_{n\in \Z} T'_n$ all translations of $\Z$ of  period $k.{\mathbf{d}}$ for $k\in \Z$ extend to embeddings of $T'$. 

\begin{claim}\label{claim:embeddable1}
$T$ is embeddable into $T'$.
\end{claim}

\noindent{\bf Proof of Claim \ref{claim:embeddable1}.} 
The ray $e(u)$ is preserved by $f$. It extends to a maximal path $D$ preserved by $f$. This path is either a ray or a two-way infinite path. We may embed it on the path on $\Z$ by the map $\phi$ sending $u$ on $-1$ and $u\boxplus 1$ on $0$. We claim that $T< D, x> $ embeds into $T'_{\varphi (x)}$ for each $x\in D$. From this, it follows that $\varphi$ extends to an embedding of  $T$ into $T'$ as claimed.  If $x:= u\boxplus (n+1)$, with $n\geq 0$, then $\varphi(x)= n$, $T'_n\equiv T_{m}$ where $n:=1+ (n\; \mod {\bf d})$. Since ${\bf d}$ is a period of $e(u)$, then  $T_{m}\equiv T_{n+1}= T<D, x>$, hence $T<D, x>\equiv T_{\varphi (x)}$. Suppose $x\leq_e$. Then some iterate of $f$ say $g$ send $x$ onto some $g(x)$ with $u<g(x)$ and $T<D, x>$ into $T<D, g(x)>$. But then, $T<D, g(x)>\equiv T_{\varphi(g(x))}$. Since $d$ is a multiple of ${\bf d}$, $\varphi (x)$ and $\varphi (g(x))$ are congruent modulo ${\bf d}$, hence $T_{\varphi (x)}\equiv T_ \varphi {(g(x))}$. This yields 
$T<D, x>\leq T_{\varphi (x)}$. 
\hfill $\Box$
\begin{claim}\label{claim:embeddable2}
$T'$ is embeddable into $T$.
\end{claim}
\noindent{\bf Proof of Claim \ref{claim:embeddable2}.} 

\begin{subclaim} \label{claim:treevertices}There are $z, z', z''\in e(u)\setminus \{u\}$ such that 
$z''<_e z'<_e z$, $z''\boxplus {\bf {d}}= z'$, $z'\boxplus d=z$,  $z''\vee_e f(z'')=z$, $f(z')=z$, $T<e(u), z>\equiv T<e(u), z'>\equiv T<e(u), z''>$. 
\end{subclaim}
\noindent{\bf Proof of Subclaim \ref{claim:treevertices}.}
Set $z:= x\vee f(x)$.  Set $z': = x\boxplus k$ where  $k:=v(z)-v(f(x))$ and $v$ is any $e$-valuation. According to Claim \ref {claim:join}, we have $f(z')= z$. Since $z'\boxplus d=z$ and  $\mathbf d$ divides $d$, we have $T<e(u), z>\equiv T<e(u), z'>$. Since $z:= x\vee f(x)$, we have $x<_e z'$, hence $z'=u\boxplus k'$ with $k'>{\mathbf{d}}$. Set $k'':= k'-{\mathbf {d}}$ and $z'':= u\boxplus k''$. Hence $z'= z''\boxplus {\mathbf {d}}$. It follows that $T<e(u), z'>\equiv T<e(u), z''>$. By construction, we have $z''<_e z'<_e z$; since  $u<_ez''<_ez'$ we have $f(u)<_ef(z'')<_e f(z')=z$. From $z=x\vee_e f(x)$ we have  $z=u\vee_ef(u)$,  hence  $z=z''\vee_e f(z'')$. This proves our subclaim. \hfill $\Box$

With this subclaim the proof of Claim \ref{claim:embeddable2} goes as follows. 

Let $h$ be an embedding of $T<e(u), z>$ into $f[T<e(u), z''>]$.  Let $H_0$ be the path joining $z$ and $f(z'')$, set $H_{n+1}:= h[H_n]$ and $H:=\bigcup_{n\in \N} H_n$. Then $H$ is an infinite path ending at $z$ and $D:= H\cup e(z)$ a two-way infinite path. We define an embedding $\psi$ of the path on $\Z$ onto the path $D$ by setting $s:={\bf {d}}+k-1$,  and for $m\geq 0$, $\psi (s+m)=z\boxplus m$ and $\psi (s-m)= z_m$ where $z_m$ is the unique member of $H$ such that $z_m\boxplus m= z$. We check that $T'_n$ is embeddable into $T<D, \psi (n)>$ for every $n\in \Z$.  Let $m\geq 0$. Then $\varphi\circ \psi (s+m)=s+m$. Indeed, $\psi(s+m)= z\boxplus m=u\boxplus(s+m+1)$ and $\varphi (u\boxplus (s+m+1))=s+m$. For $m=0$, we have  $T'_s\equiv T_{s+1}= T<e(u), z>\equiv T<e(u), z'>$. The map $f$ embeds $T<e(u), z'>$ into $T<D, z>= T<D, \psi (s)>$. Hence $T'_s$ embeds into $T<D, \psi (s)>$. For $m>0$ we have $T'_{s+m}\equiv T_{s+m+1}= T<D, \psi (s+m)>$. Let $m':= m\; \mod{\bf d}$. We have $T'_{s-m}\equiv T'_{s-m'}\equiv T_{s+1-m'}\equiv T_{s+1-m'-d}\leq T_{s+1} <H, z_{m'}>\leq T_{s+1}<H, z_{m}>$ and 
$ T_{s+1}<H, z_{m}> \leq T<D, \psi (s-m)>$.  \end{proof}

The origin of $C_o$ is an isomorphism invariant of $T$, we will call it the  {\em origin} of $T$.

\subsection{The space of ends}\label{space of ends} 
Let $r$ be a vertex of $T$ and let $\Omega_r(T)$ be the set of rays
starting at $r$. The map $\phi_r$ which associate with each end $e$
the unique ray $e(r)$ belonging to $e$ and starting at $r$ is a
bijective map of $\Omega(T)$ onto $\Omega_r(T)$. Using this,
$\Omega(T)$ can be topologized as a subset of $V(T)^\N$, equipped with
the product topology, with $\N$ equipped with the discrete topology, as
well as a subset of the Cantor space $\powerset (V(T))$. The major
features of $\Omega(T)$ with respect to the problem we consider are
the following: \\ a) The embeddings of $T$ acts on $\Omega (T)$; more
specifically, if $f$ is an embedding of $T$, then $f$ defined by
$f(e):=end(f[C])$ for some $C\in e$ is a continuous embedding of
$\Omega(T)$.

b)  $\Omega(T)$ is \emph{topologically scattered} (that is every subset contains an isolated point)  if and only $T$ is scattered in the sense that no subdivision of the binary tree  is embeddable in $T$ (See Theorem \ref{thm-scattered} and \cite {jung}). 

c) If $T$ is locally finite, then $\Omega(T)$ is compact, in fact a Stone space (totally disconnected compact space) and  $\Omega(T)$ is countable if and only if $T$ is scattered (see \cite{diestel}). 

In somewhat more concrete terms, see also \cite{diestel}:  Let $(T,r)$ be the tree $T$  rooted at $T$. For every $s\in V(T)$ let $\Omega_{r,s}(T)$ be the set of rays in $\Omega_r(T)$ which contain $s$. The sets $\Omega_{r,s}(T)$ form a basis of a topology.  This topology on $\Omega_r(T)$ is easily seen and of course well known to satisfy Item c). Let $r'\in V(T)$. Let $P$ be the set of vertices on the path between $r$ and $r'$. For $C\in \Omega_r(T)$ let $C''$ be the ray induced by the set $V(C)\setminus P$ of vertices. Let $C'$ be the unique extension of $C''$ to a ray in $\Omega_{r'}(T)$ starting a $r'$. The map which associates with $C\in \Omega_r(T)$ the ray $C'\in \Omega_{r'}(T)$ is a homeomorphism. Both $C$ and $C'$ are elements of the same end of $T$. It follows that the set of ends of $T$ inherits via $\phi^{-1}_r$ a topology homeomorphic to $\Omega_r(T)$  which indeed is independent of the particular vertex $r\in V(T)$.  
%
 
Let $f$ be an embedding mapping $r$ to $r'$. Then $f$ is an isomorphism of $T$ to  $f[T]$. Let $C$ be a ray in $\Omega_r(T)$ contained in the end $e$. Then $f[C]$ is a  ray  in $\Omega_{r'}(f[T])$ and  also a ray in $\Omega_{r'}(T)$.   The end of $f[T]$ containing $f[C]$ is a subset of the end, denoted $f(e)$,  of $T$ containing $f[C]$. Hence,  $f$ is a homeomorphism of $\Omega_r(T)$ to $\Omega_{r'}(f[T])$. From this,  Item a) follows.  

This applies to the Cantor-Bendixson derivatives of $\Omega_r(T)$. Let
us recall that if $X$ is a topological space, we may denote by
$\text{Isol}( X)$ the set of isolated points of $X$ and define for
each ordinal $\alpha$, the \emph{$\alpha$-th-derivative}
$X^{(\alpha)}$ setting $X^{(0)}=X$, $X^{(\alpha)}=
X^{(\alpha-1)}\setminus \text{Isol}(X^{(\alpha-1)})$ if $\alpha$ is a
successor ordinal and $X^{(\alpha)}= \bigcap_{\beta<\alpha}
X^{(\beta)}$ if $\alpha$ is a limit ordinal. The least $\alpha$ such
that $X^{(\alpha)}=X^{(\alpha+1)}$ is the \emph{Cantor-Bendixson rank}
of $X$ which we denote by $rank(X)$.  The \emph {rank} or \emph{order}
of an element $x\in X$, denoted by $rank(x, X)$ is the least ordinal
$\alpha$ such that $x\not\in X^{\alpha+1}$.  If $X$ is non-empty,
$rank(X)$ is a successor ordinal iff there is a last non-empty
derivative, that we denote by $X^{(last)}$.  As it is well-known, $X$
is scattered iff $X^{(rank(X))}= \emptyset$; if furthermore, $X$ is
compact, then $rank(X)$ is a successor ordinal and $X^{(last)}$, is
finite.  If $X= \Omega_r(T)$ we denote by $\Omega_r^{(\alpha)}(T)$ the
$\alpha$-th derivative of $\Omega(T)$ and if $rank(\Omega_r(T)$ is a
successor ordinal we denote by $\Omega^{(last)}_r(T)$ the last
non-empty derivative. For an example, if $T$ is rayless, then
$\Omega(T)=\emptyset$ hence then $rank(\Omega(T))= 0$. If $T$ is an
infinite one way path or consists of a set of one-way infinite path
originating from some vertex $r$ but otherwise disjoint, then
$\Omega_r^{(1)}(T)$ is empty and hence $rank(\Omega_r(T))$ is equal to
one. Implying that $\Omega^{(1)}(T)$ is empty and the rank of
$\Omega(T)$ is one.

Let $f$ be an embedding of $T$. Let $r\in V(T)$, and  $r':= f(r)$. Since $f$ is a homeomorphism of $\Omega_r(T)$ onto  $\Omega_{r'}(f[T])$, a  ray $C$ is isolated in $\Omega_r(T)$ iff every  $f[C]$ is isolated in $\Omega_{r'}(f[T])$.  Hence,  the $\alpha$-derivative   of $\Omega_r(T)$ is mapped by $f$ onto  the $\alpha$-derivative of $\Omega_{r'}(f[T])$  and hence the Cantor-Bendixson rank of $\Omega_r(T)$ is equal to the Cantor-Bendixson rank of $\Omega_{r'}(f[T])$. Since $f$ is a continuous maps  of $\Omega_r(T)$ into   $\Omega_{r'}(T)$ it maps each $\Omega_r^{(\alpha)}(T)$ into $\Omega_{r'}^{(\alpha)}(T)$. Hence, with the identification of $\Omega(T)$ with $\Omega(T)$ and $\Omega_{r'}(T)$, $f$ preserves each $\Omega^{(\alpha)}(T)$.     

We give an illustration of the notion of rank  in  Subsection \ref{subsection:preservation}, particularly with a proof of the first part of Theorem \ref{main1-2}. To do so we need to compute the rank for some scattered trees. 

\subsection{Operations on trees and computation of the rank}
Let us define the following operations on rooted trees that we call respectively \emph{successor}, \emph{sup} and \emph{sum}.
If $T$ is a rooted tree, $1+T$ is the rooted tree obtaining by adding a new vertex, say $a$, joined to the root of $T$, and by choosing $a$ as a root of this new tree. If $(T_i)_{i\in I}$ is a family of rooted trees, then their supremum $\bigvee_{i\in I} T_i$ is the tree obtained by identifying all the roots of the $T_i$'s to a single one. If $(T_n)_{n\in \N}$ is a family of rooted trees, each $T_n$ rooted at $n$,   then $\bigoplus_{n\in \N} T_n$ is the tree rooted at $0$ of the sum of the $T_n$'s over the infinite path on $\N$. If for some number $n_0$ the trees $T_n$ with $n>n_0$ are all equal to the empty tree $\Box$, then $\bigoplus_{n\in \N} T_n$ is the tree rooted at $0$ of the sum of the $T_n$'s over the finite path on $\{0,1,\dots, n_0\}$. 

We relate the the rank with the operations sup and sum. For this we will need the following notion:

A sequence $(T_n)_{n\in \N}$  of trees  has property $(\ast)$ if: 
\[
\hskip -3pt(\ast)\hskip 15pt  \inf\{\sup\{rank(\Omega(T_n)): m\leq n\in \mathbb{N}\}: m\in \N\}<\sup\{rank(\Omega(T_n)): n\in \N\}.   
\]
Note that if the sequence of trees $T_n$ has property $(\ast)$, then $\max\{rank(\Omega(T_n)): n\in \N\}$ exists and there is a largest number $n_0\in \mathbb{N}$ with $rank(\Omega(T_{n_0}))=\max\{rank(\Omega(T_n)): n\in \N\}$ and with $rank(\Omega(T_{n_0}))> \sup\{rank(\Omega(T_n)): n_0< n\in \mathbb{N}\}$.   Note that $(\ast)$ does not hold if and only if the following property (not $(\ast))$ holds:
\begin{align}\label{align:notast}
&\text{for every ordinal $\alpha<\sup\{rank(\Omega(T_n)): n\in \N\}$ and every $n\in \N$}\tag{not $(\ast)$}\\
&\text{there is an $n<m\in\N$ with} \notag\\
&\alpha<rank(\Omega(T_n))\leq \sup\{rank(\Omega(T_n)): n\in \N\}. \notag  
\end{align}

\begin{lemma}\label{lem:values of the rank}
Let $(T_i)_{i\in I}$ be a family of pairwise disjoint rooted trees. If each $\Omega (T_i)$ is topologically scattered, then $\Omega(\bigvee_{i\in I} T_i)$ is scattered and:
\begin{align}\label{eq:rank-sup}
rank(\Omega(\bigvee_{i\in I} T_i))= sup\{rank(\Omega(T_i)) : i\in I\}.
\end{align}
Let $(T_n)_{n\in \N}$ be  a family of pairwise disjoint rooted trees, each $T_n$ rooted at $n$ If   each $\Omega (T_n)$ is topologically scattered, then  $\Omega(\bigoplus_{n\in \N} T_n)$ is topologically scattered and:
\begin{equation}\label{eq:rank-sum1}
rank(\Omega(\bigoplus_{n\in \N} T_n))=
\begin{cases}
\max\{rank(\Omega(T_n)): n\in \N\},     &\text{ if  $(\ast)$ holds;}\\ 
\sup\{rank(\Omega(T_n)): n\in \N\}+1,   & \text{if  $(\ast)$ does not hold}.
\end{cases}
\end{equation} 
\end{lemma}

\begin{proof} 
Let $(T_i)_{i\in I}$ be a family of pairwise disjoint rooted trees and
let $T_i'$ be the tree arising from $T_i$ by identifying its root with
the other trees to obtain the root, say r, of $\Omega(\bigvee_{i\in
I}T_i)$. Then $T_i'$ is isomorphic to $T_i$. Let $S$ be a non-empty
subset of $\Omega(\bigvee_{i\in I}T_i)$. Then there is an $i\in I$ for
which $S\cap \Omega_r(T_i')$ is not empty and hence, if
$\Omega_r(T_i')$ is scattered contains an isolated chain $C$ of
$\Omega_r(T_i')$. This chain $C$ is isolated in $\Omega(\bigvee_{i\in
I}T_i)$. 

Equation (\ref{eq:rank-sup}): Let $T:= \bigvee_{i\in I}
T_i$. Observe that $\Omega_r(T)$ is the union of the sets
$\Omega_r(T_i)'$ of chains and that these sets are pairwise disjoint.
Each $\Omega_r(T_i')$ is a clopen subset of $\Omega(T)$. The result
follows.

Let $T:= \bigoplus_{n\in \N} T_n$ and $\gamma:= \sup \{rank(\Omega_n(T_n)): n\in \N\}$. For $n\in \N$ denote by $T'_n$ the tree obtained by adding to  $T_n$  the path from 0 to $n$ rooted at 0. Let $S$ be a non-empty subset of $\Omega(T)$.  Clearly, $\Omega_0(T)$ is the union of the $\Omega_{0}(T_n')$  plus the path on $\N$. Each $\Omega_{0}(T_n')$ is a clopen set homeomorphic to $\Omega_{n}(T_n)$. Hence $\gamma:=\sup\{rank(\Omega(T_n)): n\in \N\}=\sup\{rank(\Omega(T_n')): n\in \N\}$ and

\begin{equation} \label{eq:rank}
\gamma  \leq rank(\Omega_0(T))\leq \gamma+1
\end{equation}
and there is an $n\in \N$ with $S\cap \Omega_0(T_n')\not=\emptyset$ or $S=\{\N\}$. If $\Omega(T_n)$ is scattered there exists a chain $C$ isolated in $\Omega_0(T_0')$  which then is also isolated in $T$.    

\vskip 5pt
\noindent
Property $(\ast)$ holds: Then $rank(\Omega(T_{n_0}))> \sup\{rank(\Omega(T_n)): n_0< n\in \mathbb{N}\}:=\delta$ and   $\gamma= rank(\Omega(T_{n_0}))$ and every path in  $\Omega_0^{(\delta)}(T)$ is a path in $\Omega_0^{(\delta)}(T'_{n_0})$,  which is not empty because $\delta<rank(\Omega(T_{n_0}))$.  Thus $rank(\Omega_{0}(T))=  rank(\Omega_0(T_{n_0}))=\max\{rank(\Omega(T_n)): n\in \N\}$.

\vskip 5pt
\noindent
Property $(\ast)$ does not hold: Then $\gamma$ is not attained or it
is attained infinitely many times. In either case the path on $\N$
belongs to $\Omega_{0}^{(\gamma)}(T)$. With inequality~(\ref{eq:rank})
this yields $rank(\Omega_{0}(T))=\gamma+1$, thus the result.
\end{proof}

\subsection{A set of ends preserved } \label{subsection:preservation}

\begin{lemma}\label{lem:dichotomic}
If a tree $T$ is infinite, locally finite and  scattered,  then there is a non-empty finite subset $\mathcal C$ of $\Omega(T)$ which is preserved by every embedding of $T$. 
\end{lemma}

\begin{proof} The proof will follow from the following claims.

The space $\Omega(T)$ is non-empty by K\H{o}nig's lemma
\cite{konig}. Since $T$ is scattered, $\Omega(T)$ is scattered; since
it is compact its rank is a successor ordinal and $\Omega(T)^{(last)}$
is finite. Set $\mathcal C:= \Omega^{(last)}(T)$.  \end{proof}

If $T$ is not locally finite,  the rank can be a limit ordinal and even if it is a successor ordinal the set $\Omega^{(\infty)}(T)$ is not necessarily finite. If $T$ is scattered but not necessarily locally  finite, Theorem \ref{main1} gives an extension of Lemma \ref{lem:dichotomic} whose proof is presented in Section \ref{prooftheorem1.1}.

\begin{proposition} \label{lem:infinite orbit}
Let $f$ be an embedding of a tree $T$. Suppose that there is a two-way
infinite path $D$ preserved by $f$ on which $f$ fixes no vertex and
reverses no edge, or if not, a one-way infinite path $C$ preserved by
$f$ with a vertex not in the range of $f$.

In the first case,  set $\mathcal J:= \{end(D^-), end(D^+)\}$ where $D^-$ and  $D^+$ are two paths whose union is $D$ and such that   $f[D^{+}]\subseteq D^{+}$   and set $\mathcal J:=\{end(C)\}$ in the second case.  If $T$ is scattered, then $rank(\Omega(T))$ is a successor ordinal and in the first case $end(D^{+}) \in \Omega^{(last)}(T)\subseteq \mathcal J$, whereas $\Omega^{(last)}(T)= \mathcal J$ in the second case.  
\end{proposition}
\begin{proof}
\begin{claim}\label {claim:orbit}
 Every end $e\not \in \mathcal J$  has an infinite orbit under $f$. In particular the members of $\mathcal J$ are the only ends preserved by  $f$. 
 \end{claim}
\noindent {\bf Proof of Claim \ref{claim:orbit}.}
In the first case take $E=D$  and in the second case take $E$ equal to a maximal one way path  (in fact a maximum one) preserved by $f$. For $x\in E$, set $T_x:=((T,E)(x), x)$. We have $T=\bigoplus_{x\in E} T_x$ and the map $f$ induces an embedding of 
$T_x$ into $T_{f(x)}$ for every $x\in V(E)$. Apply Corollary  \ref{cor:sumoverpath} and let $e \not \in \mathcal J$ and $R\in e$. Then $R$ has all but finitely many of its vertices in one of the rooted trees $T_x$. Then $f[R]$ has all but finitely many of its vertices in $T_{f(x)}$. Because $x\not=f(x)$ the rooted trees $T_x$ and $T_{f(x)}$ have no vertex in common, implying that the end $f(e)$ containing $f[R]$ is different from $e$. Similarly $f^{(n)}(e)$ is different from $f^{(m)}(e)$ for $n\not=m$.\hfill $\Box$

\begin{claim}\label {claim:rank}
Let $\alpha:= \sup\{rank (\Omega_{x} (T_x)): x\in E\}$. Then, in the
first case $e(D^{+}) \in \Omega^{(\alpha)}(T)\subseteq \mathcal J$,
whereas $\Omega^{(\alpha)}(T)= \mathcal J$ in the second case. In particular, $rank
(\Omega(T)) = \alpha+1.$
\end{claim}

\noindent {\bf Proof of Claim \ref{claim:rank}.}
Let $\alpha:=  \sup\{rank (\Omega_{x} (T_x)): x\in E\}$. Since $f$ embeds each 
$T_x$  into   $T_{f(x)}$ for all $x\in E$,  $rank (\Omega_{x} (T_{x}))\leq rank (\Omega_{f(x)} (T_{f(x)}))$.  Let $x, y \in E$. Set $x\leq_E y$ if the ray originating at $x$ and belonging to $end(D^{+})$ contains $y$.  Since  $x\leq f(x)$, we have   $\sup \{rank (\Omega_{y} (T_{y})): x\leq_E x\}=\alpha$. Let $T_{\geq x}:=\bigoplus_{y\in E_x} T_y$, where $E_x:= \{y\in E: x\leq y\}$. 

We claim that $rank (\Omega_x(T_{\geq x}))=\sup \{rank
(\Omega_{y}(T_{y})):y\geq x\}+1=\alpha+1$, a fact which follows from
Lemma \ref{lem:values of the rank}. Indeed, label the vertices of
$E_x$ as $x_0, \dots x_n,\dots$ in an increasing order, set $T'_n:=
T_{x}$ and observe that $(T'_{n})_{n\in \N}$ does not have property
$(*)$. According to $(2)$ of Lemma \ref{lem:values of the rank}, $rank
(\Omega(T_{\geq x}))=\sup \{rank (\Omega(T_{y})):
x\leq_Ey\}+1=\alpha+1$; proving our claim.  

Let $D_{\geq x}:= D\cap
\{y: x\leq_E y\}$. Since $D_{\geq x}$ meets no $T_{y}$ into infinitely
often, it follows from our claim that $\Omega^{(\alpha)}(T_{\geq
x})=\{end(D_{\geq x})\}$.  If $E$ is a ray, then for the least element
$x$ of $E$ we have $T= T_{\geq x}$, hence $\Omega(T)^{\alpha})=\mathcal
J$ as claimed.  If $E$ is a double ray, then, since each $\Omega
(T_{\geq x})$ is closed into $\Omega (T)$, $rank( \Omega (T_{\geq
x}))\leq rank (\Omega (T))$, we infer $\alpha+1\leq rank (\Omega
(T))$. Furthermore, $\Omega (T)= \bigcup_{x\in E} \Omega (T_{\geq x})
\cup \{end(D^{-}\}$, hence we have $ end (D^+)\in \Omega^{(\alpha)}
(T)\subseteq \mathcal J$. In particular, $rank(\Omega^{(\alpha)} (T))=
\alpha+1$. This proves Claim \ref{claim:rank}. \hfill $\Box$

With this Claim, the proof of the Proposition \ref{lem:infinite orbit} is complete
\end{proof}

\begin{remark}  
In the first case of Proposition \ref{lem:infinite orbit},  we have $rank(end (C^-),\Omega(T)) \leq rank(end(C^+),\Omega(T))$. 
\end{remark}

This inequality can be proved directly.  Let $e:=end(C^{-})$, $e':=
end(C^{+})$ and $\alpha:= rank (e, \Omega (T))$. Then every
neighbourhood of $e$ contains an element $e_{\beta}$ of
$\Omega^{(\beta)}(T)$ for every $\beta< \alpha$. Let $U'$ be a
neighbourhood of $e'$ in $\Omega(T)$. Some iterate of $f$ will map $e_{
\beta}$ into some element $e'_{\beta}\in U'\cap
\Omega^{(\beta)}(T)$. This implies that $e'\in \Omega^{(\beta)}(T)$.
The inequality follows.


\begin{corollary} \label{for:the two cases} 
If a tree $T$ is scattered, then $rank(\Omega(T))$ is a successor
ordinal and $\Omega^{(last)}(T)$ has at most two elements provided
that some embedding $f$ does not fix a vertex or an edge.
\end{corollary}

\begin{proof}
 By Proposition \ref{prop:embeddingsoftrees}, such an $f$ satisfies the hypothesis of Proposition \ref{lem:infinite orbit}. 
\end{proof}
   
%

\begin{corollary}\label{lem:part of thm1.3} 
Let $T$ be a scattered tree. If $rank (\Omega(T))$ is a successor
ordinal, $\vert \Omega^{(last)}(T)\vert =1$ and $e\in \Omega^{(last)}(T)$,
then $e$ is preserved forward by every embedding of $T$.
\end{corollary}

\begin{proof}
Let $f$ be an embedding of $T$. Then $f$ preserves
$\Omega^{(last)}(T)$. Since $\Omega^{(last)}(T)= \{e\}$, $f$ preserves
$e$. Suppose that $e$ is not preserved forward by $f$. That is there
is some ray $C\in e$ such that no infinite subray $C'$ of $C$ is sent
into $C'$ by $f$.  In this case, some subray, say $C'$, extends to a
two-way infinite path $D$ which is preserved by $f$. Indeed, pick a
subray $C'$ of $C$ such that $C' \subset f(C')$ and $D:= \bigcup_{n\in
\N} f^{(n)} (C)$. We may write $D:= \{x_n:n\in \Z\}$ in such a way
that $C'= D^{-}$ and $f(x_{n})= x_{n+k}$ for some positive $k$. Let
$e':= end(D^{+})$. According to Proposition \ref{lem:infinite orbit}
we have $e'\in \Omega^{(last)}(T)$. Since $e\not =e'$, this is
impossible. \end{proof}

\subsection{Proof of Theorem \ref{main1-2}} \label{proof of main1-2} 
According to Corollary \ref{lem:part of thm1.3}, if $T$ contains
exactly one end of maximal rank, then this end is preserved
forward. Next, if $e$ contains some regular end $e$ preserved forward
by every embedding, then $e$ contains some ray preserved forward
provided that $T$ is scattered and $e$ almost rigid (Proposition
\ref{prop:one end almost rigid}) or $e$ is not almost rigid
(Proposition \ref{Fact:finitdist}) .

\subsection{Examples of ends preserved} \label{subsection:examples}

We present six examples. The first four examples are scattered trees
 containing exactly one end preserved by every embedding.

In the first three examples there are embeddings with no fixed
point. The third example is due to Polat and Sabidussi
\cite{polat-sabidussi}. In the fourth example the end is ``almost
fixed" by every embedding.  First a notation: let $\kappa$ be a
cardinal, we set $\bigvee_{\kappa} T$ for the tree sup of $\kappa$
copies of $T$, that is the tree $T$ is rooted and the copies are
disjoint, and we glue the copies together at their roots. For an example, if $T$ is a $1$-element tree, then $\bigvee_{\kappa} T= T$. 

\noindent{\bf Example 1.} \label{ex:examples}
Let $(T^{1}_n)_{n\in \omega}$ be the sequence of finite trees defined
by induction as follows: $T^{1}_0$ is the tree consisting of one
vertex rooted at $0$; $T^{1}_{n+1}:= 1+\bigvee_{2}T^{1}_{n}$ rooted at
$n+1$, (hence, $\vert T^{1}_n\vert =2^n$). Set
$T^{1}_{\omega}:= \bigoplus_{n<\omega} T^{1}_n$. 
This tree has an interesting reproducing property: if we delete the end vertices of this tree, we get an isomorphic copy.

\begin{center}
\begin{figure}[H]\label{fig:example1}
\includegraphics[width=8cm]{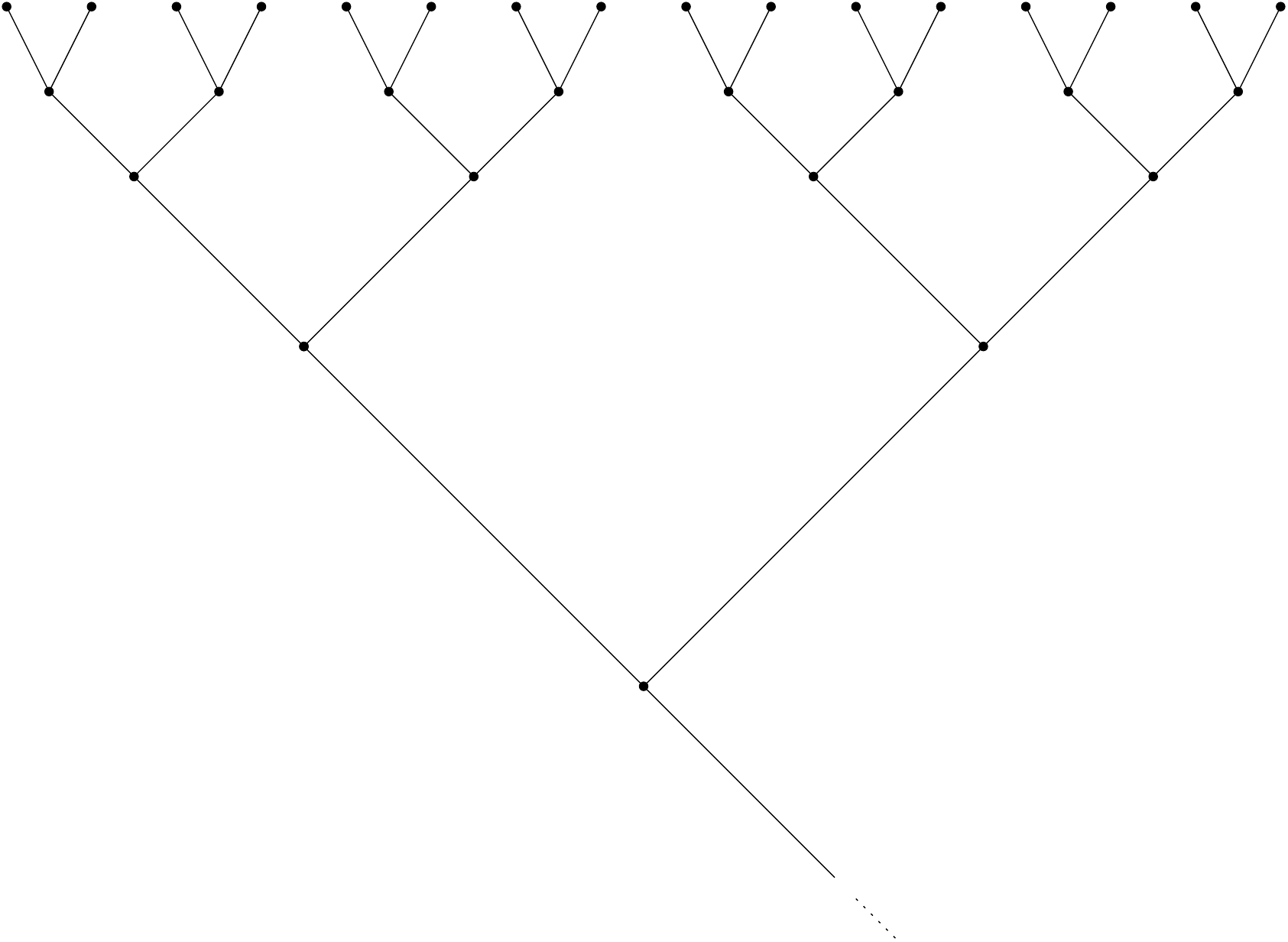}
\caption {Example 1: $T^1_{\omega}$}
\end{figure}
\end{center}

\noindent{\bf Example 2.} 
Let $(T^{2}_n)_{n\in \omega}$ be the sequence of trees defined by
induction as follows: $T^{2}_0$ is the tree consisting of one vertex
rooted at $0$; $T^{2}_{n+1}:=\bigvee_{\aleph_0}(1+ T^{2}_{n})$ rooted
at $n+1$, $T^{2}_{\omega}:= \bigoplus_{n<\omega} T^{2}_n$.

\begin{center}
\begin{figure}[H]\label{fig:example2.pdf}
\includegraphics[width=10cm]{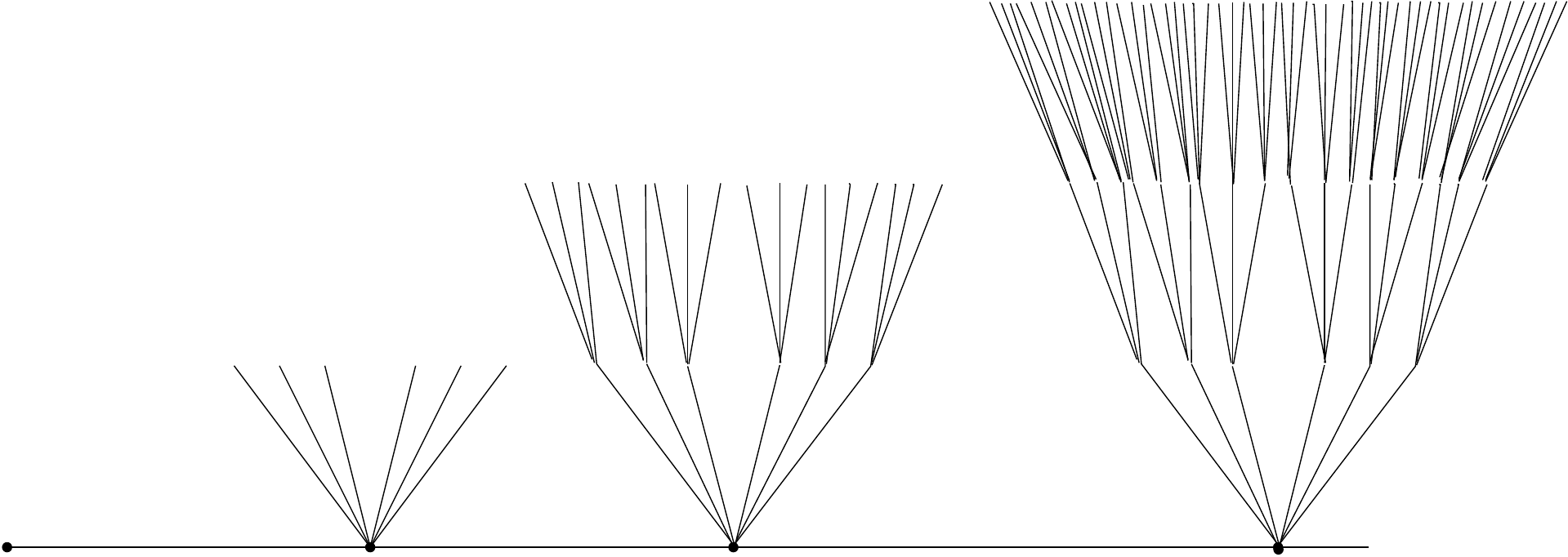}
\caption {Example 2: $T^2_{\omega}$}
\end{figure}
\end{center}

\noindent{\bf Example 3.} 
Let $(T^{3}_{n})_{n\in \omega}$ be the sequence of trees defined by
induction as follows: $T^{3}_n$ for $n=0, 1$ is the rooted tree consisting of an infinite 
path  rooted at $n$; for larger $n$, $T^{3}_{n+1}:=1+ \bigvee_{2}(1+ T^{3}_n)$
rooted at $n+1$, $T^{3}_{\omega}:= \bigoplus_{n<\omega} T^{3}_n$. (See
Fig 2 in Polat and Sabidussi \cite{polat-sabidussi}.) Equivalently,
$T^{3}_{\omega}$ is obtained from $T^{1}_{\omega}$ by replacing each
terminal vertex of $T^{1}_{\omega}$ by an infinite path rooted at this vertex.  We could 
as well replace each terminal vertex by two infinite paths rooted at this vertex.

\begin{center}
\begin{figure}[H]\label{fig:example3}
\includegraphics[width=8cm]{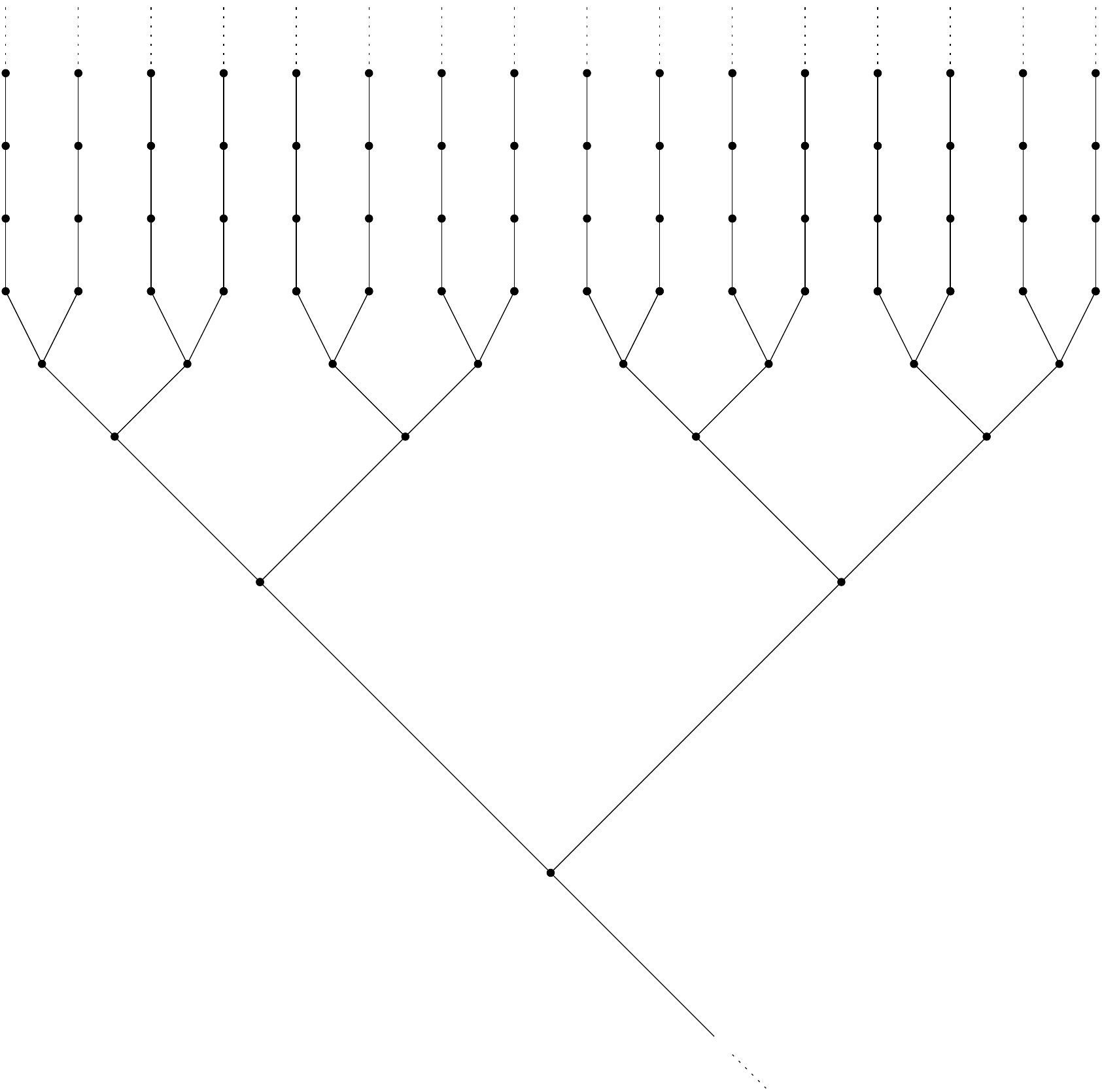}
\caption {Example 3: $T^3_{\omega}$}
\end{figure}
\end{center}

\noindent {\bf Example 4.} Let $T^{4}_{\omega}$ be the tree obtained by replacing each terminal vertex of 
$T^{1}_{\omega}$ by the rooted tree $\bigvee_{3}2$.

\begin{center}
\begin{figure}[H]\label{fig:example4}
\includegraphics[width=8cm]{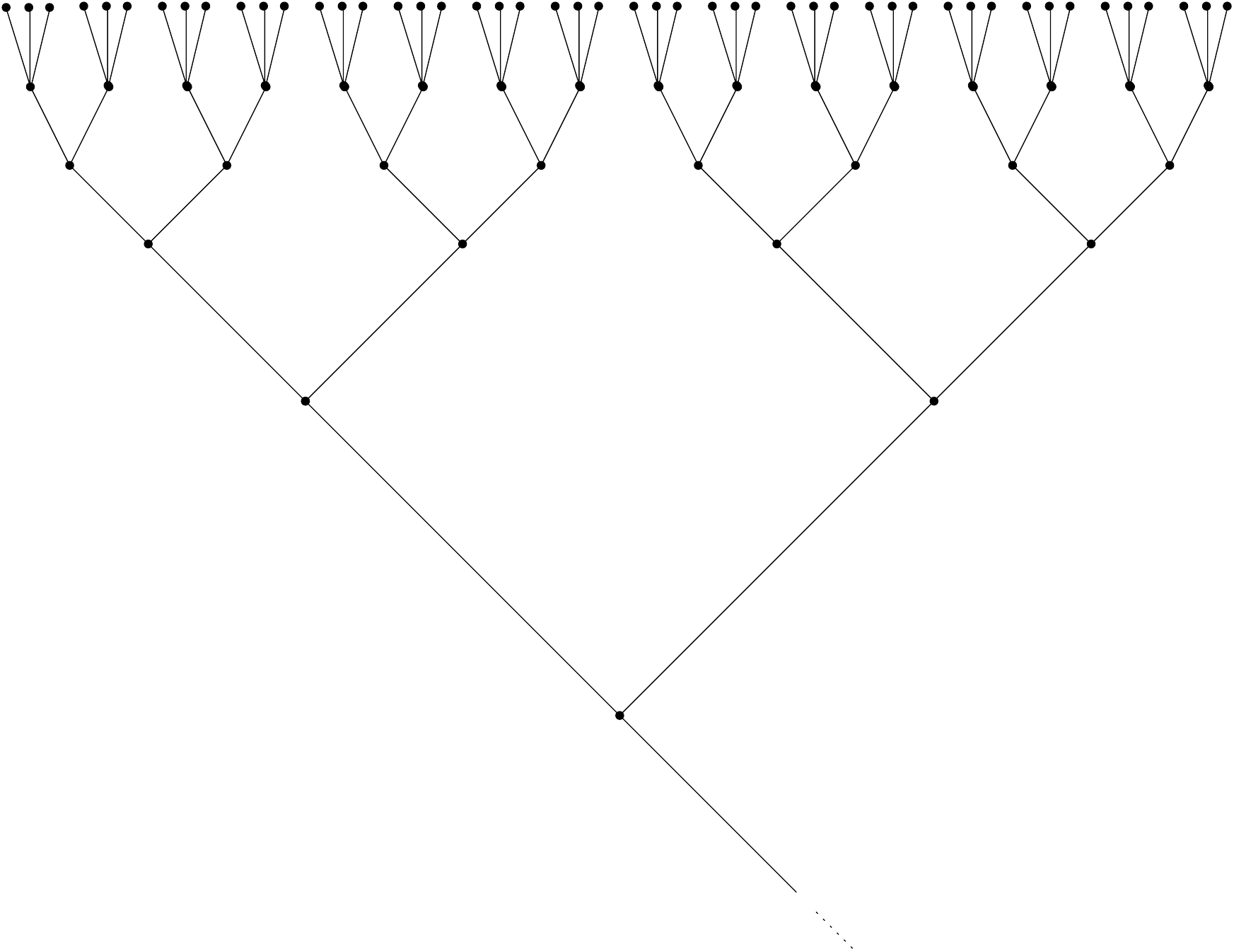}
\caption {Example 4: $T^4_{\omega}$}
\end{figure}
\end{center}

%


In all these examples, the non-negative integers form a ray. This ray
is not regular, hence its end, say $e$ is not regular. This end is
preserved by all embeddings. In the first two examples, this is the
only end, hence the rank of the space of ends is $1$.  This end is
preserved forward, but no ray containing it is preserved. In the third
example, the rank of the space of ends is $2$.  In the first three
examples there are embeddings which move the path forward. Hence
$\vert twin(T)\vert\geq 2^{\aleph_0}$ according to
Theorem~\ref{main2}.  This is not the case in the fourth example.
Indeed, there the end is almost rigid (any  embedding of
$T_{\omega}^{4}$ preserves the level function).  Because
$T_{\omega}^{4}$ is locally finite, every embedding is an automorphism.

The next two examples are trees containing exactly one end preserved
by every embedding, this end being preserved backward.

\noindent{\bf Example 5.}\label{example-Hamann} 
To the rooted binary tree add an infinite ray at the root $r$ (See Figure 3). Let
$T^{5}$ be the resulting tree. Label the vertices from the root by
integers as follows: The root has label $0$. The label of a vertex $x$
in the binary tree is the distance from $x$ to $r$. The label of a
vertex $x$ in the ray added is $-n$ if the distance from $r$ is $n$.
To each vertex labelled $0$ modulo $3$, graft in ten leaves and to each
vertex labelled $1$ modulo $3$ graft in five leaves. For degree
reasons, if $f$ is an embedding, then either the ray added is fixed or
it is moved backward.

\begin{center}
\begin{figure}[H]\label{fig:hamman}
\includegraphics[width=8cm]{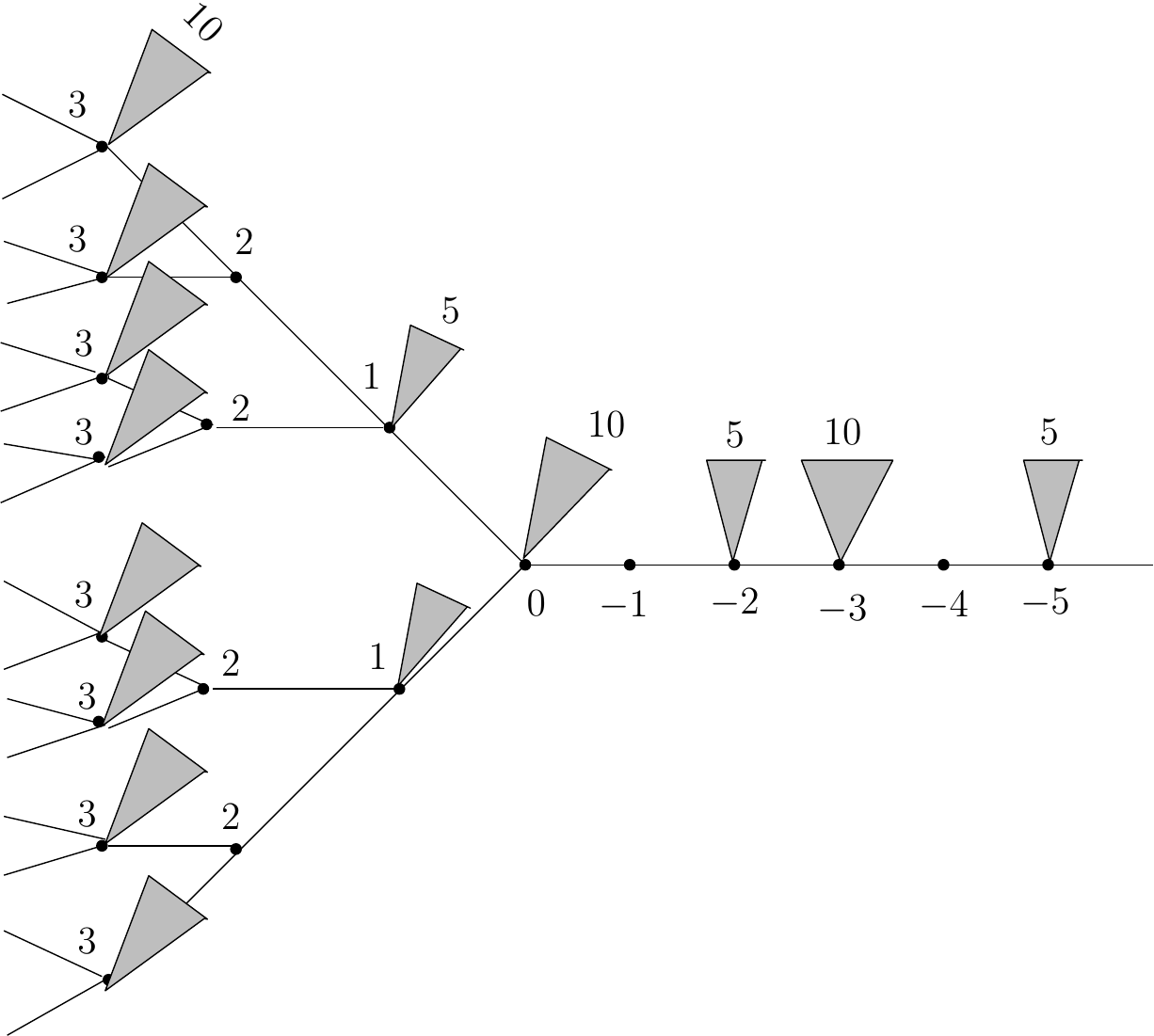}
\caption{Hamann's Example 5}
\end{figure}
\end{center}

\noindent{\bf Example 6.}\label{example-Lehner}
In the ternary tree, select an infinite ray, say $0, 1, 2 \dots, n,
\dots $ (See Figure 8).  Subdivide each edge on this ray by adding two
vertices. Subdivide all other edges $w:=\{u, v\}$ (outside that ray)
whose distance from $w$ to that ray is even by adding a vertex.  Hence
edges touching the ray are subdivided, edges a step further are not
subdivided, etc.  Let $T^{6}$ be the resulting tree. In this tree, the
infinite ray with the vertices added is either fixed or moved
backward.

\begin{center}
\begin{figure}[H]\label{fig:lehner}
\includegraphics[width=8cm]{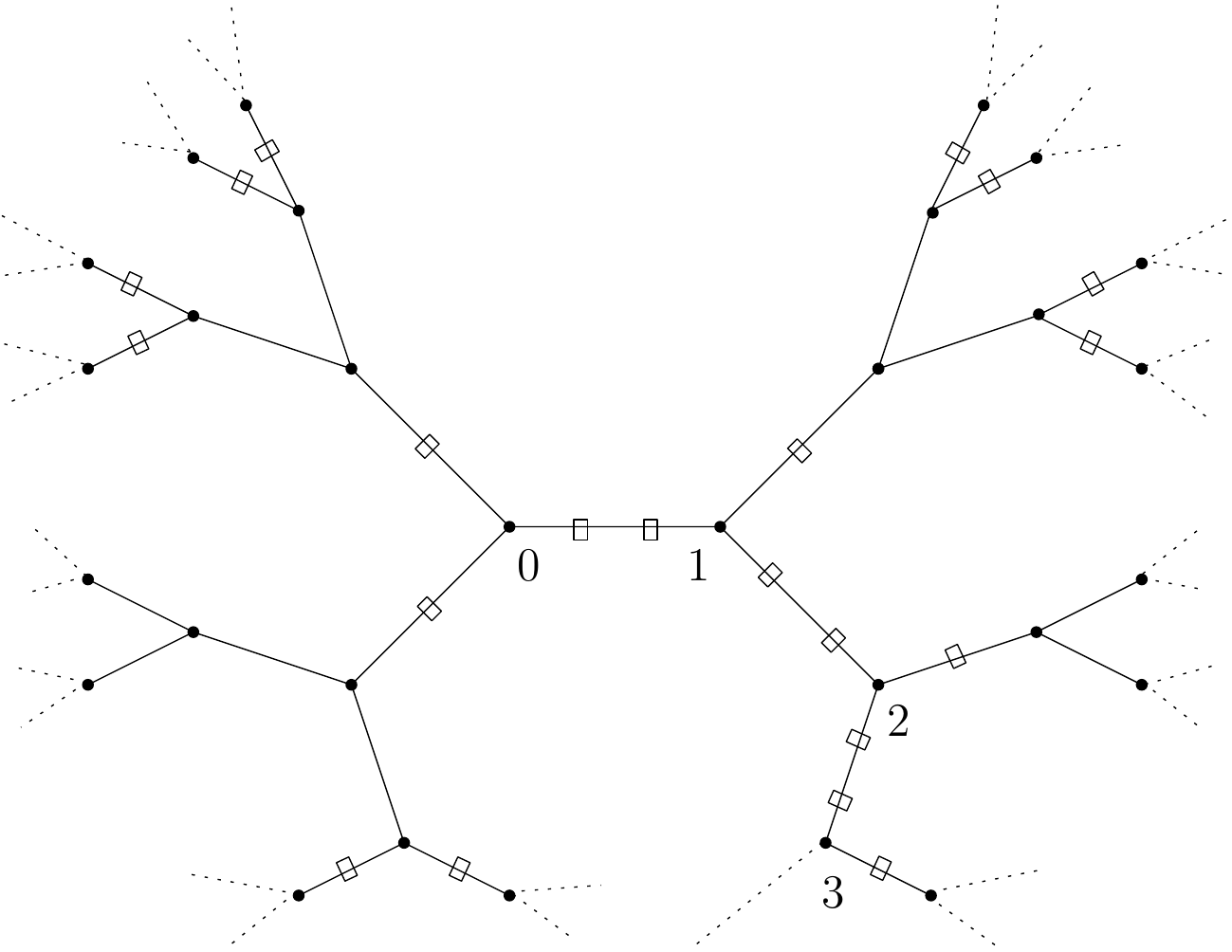}
\caption{Lehner's Example 6}
\end{figure}
\end{center}

%

\section{Scattered trees}\label{sec:subsectionnonscattered}
 
In this section we present some characterizations and constructions of
scattered trees, e.g. Lemma \ref {lem:scattered-binary} and Theorem
\ref{thm-scattered}, going back to Laver \cite{laver} and Jung
\cite{jung}.

\subsection{Scattered trees versus non-scattered trees}\label{subsec:subsectionnonscattered}

Let $T$ be a non scattered tree.  The \emph{non-scattered kernel} of
$T$, written $Ker(T)$, is the union of subsets of the form $V(T')$
where $T'$ is a subtree of $T$, which contains a subdivision of
$\mathrm{T}_2$.

\begin{notation}
Let $T$ be a tree, $x\in V(T)$ and $y\in V(T)\setminus \{x\}$, then
$T\setminus\{x\}(y)$ stands for the connected component of $y$ in the
graph $T\setminus\{x\}$.
\end{notation}

We have the following decomposition result, whose proof is immediate:

\begin{theorem} 
Let $T$ be a non scattered tree.  The non-scattered kernel of $T$ is a
non-scattered tree such that $Ker(Ker (T))=Ker(T)$ and $(T,
Ker(T))(x)$ is scattered for every $x\in Ker(T)$.
\end{theorem}

\begin{remark} 
$Ker(T)$ is the set of vertices $x\in V(T)$ having at least two
distinct neighbours $y$ and $y'$ such that the connected components $T
\setminus \{x\}(y)$ of $y$ and $T \setminus \{x\}(y')$ of $y'$ in
$T\setminus \{x\}$ are non-scattered. 
\end{remark} 

Indeed, if $x$ is such a vertex, then the tree that is the  union of a subdivision of
the binary tree taken in $T\setminus \{x\}(y)$ and a subdivision of
the binary tree taken in $T\setminus \{x\}(y')$, with a path
connecting these two subdivision, is a subdivision of\/ $\mathrm{T}_2$
containing $x$, hence $x\in Ker(T)$. The converse is immediate: If
$x\in Ker(T)$ and $T'$ is a subtree of $T $ which is a subdivision of
the binary tree containing $x$, then $x$ has two neighbours $y$ and $y'$ in $T'$ such
that the connected components of $y$ and $y'$ in $T'\setminus \{x\}$
are non-scattered. Hence these connected components in $T$ are
non-scattered too.

\begin{corollary}The non-scattered kernel of a non-scattered tree is preserved by every embedding. 
\end{corollary}

\begin{problem} If there is some embedding $f$ with $Ker(T)\not =f(Ker(T))$, does $T$ have infinitely many twins?

\end{problem}


\begin{proposition} \label{kappatwinsscattered}Let $T$ be a non scattered tree and $x\in Ker(T)$ and  $T_x$ be  the rooted tree $((T, Ker(T))(x),x)$. If $T_x$ has $\kappa$ twins, then $T$ has at least $\kappa$ twins. 
\end{proposition}
\begin{proof} Let $K:=Ker(T)$.  Let $(T', x)$ be a   twin of  $T_x$ and let $S:= \{z \in K : ((T, K)(z), z) \equiv ((T, K)(x), x)\}$.   Let $T(T')$ be the tree obtained by replacing each $((T, K)(z), z)$ for $z\in S$ by $(T', x)$.

Then $T(T')$ is equimorphic to $T$ and $Ker(T(T'))=K$.  We claim that if $(T',x)$ and $(T'',x)$ are two twins of $((T, Ker(T))(x), x)$ which are  not isomorphic (as rooted trees), then the trees  $T(T')$ and $T(T'')$  are not  isomorphic  and this will prove the proposition. Suppose by contradiction that there is some isomorphism $f$ from  $T(T')$ onto  $T(T'')$.  Clearly $f$ carries the kernel $K$ of $T(T')$ onto  the kernel $K$ of $T(T'')$ and we have   $((T(T''), K)(f(z)), f(z))\simeq ((T(T'),K)(z), z)$. Let  $z\in S$.  We have by construction $((T(T'),K)(z), z)\simeq (T',x)$. Hence $((T(T''), K)(f(z)), f(z))\simeq (T',x)$. If $f(z)\not \in S$, then $((T(T''), K)(f(z)), f(z))$ is not equimorphic to $(T'', x)$ which contradicts the fact that $(T',x)$ and $(T'', x)$ are equimorphic.  Hence $f(z)\in S$  and hence $((T(T''), K)(f(z)), f(z))\simeq (T'',x)$, from which follows that $(T',x)$ and $(T'',x)$ are isomorphic. 
\end{proof}

\subsection{Scattered trees as posets}
Let $T$ be a tree, $x\in V(T)$ and $P(T,x)$ be the poset associated with the rooted tree  $(T, x)$. For an example, if $\mathrm{T}_2$ is the binary tree  and the root is the unique vertex of degree $2$, we have the dual of the ordered binary tree. 

%

\begin{lemma}\label{lem:scattered-binary}Let $T$ be a tree and $x\in V(T)$. Then the following properties are equivalent:
\begin{enumerate} [{(i)}]

\item $T$ is scattered;
\item The dual of the ordered binary tree does not embed into $P(T,x)$ as a poset.
\end{enumerate}
\end{lemma}
\begin{proof} 
$\neg (ii)Ê\Rightarrow \neg (i)$. Suppose that there is an embedding
$f$ of the dual of the ordered binary tree into $P(T,x)$. We construct
an embedding $g$ of a subdivision of $\mathrm{T}_2$ into $T$ and hence
$T$ is not scattered. For that, we transform first $f$ to a map
$\check f$ from $2^{<\N}$ into $P(T,x)$.  We set $\check f(\Box):=
f(0)\vee f(1)$ and more generally $\check f(s):= f(s\concat 0)\vee
f(s\concat 1)$. This map satisfies $\check f(s\wedge t)=\check
f(s)\vee \check f(t)$. Now, if for each $s\in 2^{<\N}$ we add the
shortest path between $\check f(s)$ and $\check f(s\concat i)$, for
$i\in 2$, the resulting graph is a subdivision of $\mathrm{T}_2$.

$\neg (i) \Rightarrow \neg (ii)$. If $T$ is not scattered, let $K$ be a subtree of $T$ which is a subdivision of the binary tree. Let $x$ be a vertex of $K$ having degree $2$. Then  the dual of the ordered binary tree $\mathrm{T}_2$ embeds into the poset $P(T, x)$.   
\end{proof}
\subsection{Construction of scattered trees}\label{subs:constsc}

We describe the collection of rooted trees such that
$rank(\Omega_r(T))\leq\alpha$.  For that, we need the following
notation.  Let $\mathcal X$ be a collection of rooted trees. We set
$\bigoplus_{n\in \N}\mathcal X$ for the collection of rooted trees
$\bigoplus_{n\in \N} T_n$ such that $T_n\in \mathcal X$ for every
$n\in \N$, with the convention that this class is reduced to the empty
tree if $\mathcal X$ is empty. Note that if the empty tree $\Box \in
\mathcal{X}$, then $\bigoplus$-sums over finite paths are in $\mathcal
X$ and hence $\mathcal X\subseteq \bigoplus_{n\in \N}\mathcal X $.  We
set $\mathcal {X}^{suc, \sup}$ for the closure of $\mathcal X$ by means
of the operations \emph{successor} and \emph{sum}. Hence $\mathcal
{X}^{suc, \sup}$ is the smallest class $\mathcal Y$ of rooted trees
which contains $\mathcal X$ and such that $T\in \mathcal Y$ provided
that $T'\in \mathcal Y$ if $T=1+T'$ or all $T_i\in \mathcal Y$ if
$T=\bigvee_{i\in I} T_i$.  We obtain from Equation (\ref{eq:rank-sup})
that if $rank (\Omega(T))\leq \alpha$ for all $T\in\mathcal{X}$, then
$rank(\Omega(T))\leq \alpha$ for all $T\in \mathcal {X}^{suc, \sup}$.

\begin{lemma}\label{lem:root-neighbour}
Let $\mathcal X$ be  a collection of rooted trees. Then $\mathcal {X}^{suc, \sup}$ is the collection $\mathcal{R}(\mathcal X)$   of rooted trees  for which there exists a rayless rooted tree $R$ with members of $\mathcal{X}$ attached at the endpoints of $R$. For $R$  empty the tree in  $\mathcal{R}(\mathcal X)$ is a member of $\mathcal{X}$. 
\end{lemma}
\begin{proof} 
Clearly $\mathcal{X}\in \mathcal{R}(\mathcal X)$.   Applying the operation $+$ to a tree in $\mathcal{R}(\mathcal X)$ or the operation $\bigvee_{i\in I}$ to a set of trees $(T_i)_{i\in I}$ results in a tree in $\mathcal{R}(\mathcal X)$.  \end{proof}

%

We  define inductively a collection $\mathcal D^{(\alpha)}_{r}$ of rooted trees for each ordinal $\alpha$ setting first $\mathcal D^{(0)}_{r }:= \{\Box\}^{suc, \sup}$. We set $\mathcal D^{(\alpha)}_{r }:=(\bigcup_{\beta\in \alpha} \mathcal D^{(\beta)}_{r })^{suc, \sup}$ if  $\alpha$ is a non-zero limit ordinal. We set $\mathcal D ^{(\alpha)}_{r }=(\bigoplus_{n\in \N}\mathcal D^{(\beta)}_{r})^{suc, \sup}$ where  $\beta$ is such that $\alpha=\beta+1$. We denote by   $\mathcal D_{r}$ the union of all $\mathcal D^{(\alpha)}_{r}$ and by $\mathcal D$ the collection of trees obtained by forgetting the roots of members of $\mathcal D_r$.   Unless otherwise stated we assume tacitly that $r$ is the name of the root of the elements of $\mathcal{D}_r$.

\begin{lemma} \label{prop:scattered} For each ordinal $\alpha$  and $T\in  \mathcal D^{(\alpha)}_{r}$,  the space $\Omega_r(T)$ is  topologically scattered and $rank(\Omega_r(T))\leq \alpha$.  

If $T\in \mathcal D^{(\alpha)}_{r}$ and $y\in V(T)$ and $\leq_r$ is the order of $P(T,r)$ with $r$ as largest element, then the subtree of $T$ induced by the set $\{x\in V(T) : x \leq_r y \}$ of vertices rooted at $y$ is an element of $\mathcal D^{(\alpha)}_{r}$. 
\end{lemma}  

\begin{proof}
The proof goes by induction on $\alpha$. Suppose  $\alpha=0$. $\mathcal D^{(0)}_{r}:= (\Box)^ {suc, \sup}$ is just the collection of   rayless trees. These are the trees such that $\Omega(T)= \emptyset$, hence they are scattered and  $rank(\Omega(T))=0$. Suppose that the Lemma  holds for all $\beta <\alpha$. If $\alpha$ is a limit ordinal, then induction asserts that for each $T\in \bigcup_{\beta<\alpha} \mathcal D^{\beta}_{r }$,  $rank(\Omega(T))<\alpha$. Since $\mathcal D^{\alpha}_{r }=(\bigcup_{\beta<\alpha} \mathcal D^{\beta}_{r })^{suc, \sup}$, the  Equation  (\ref{eq:rank-sup}) of Lemma \ref{lem:values of the rank} ensures that $rank(\Omega(T))\leq \alpha$ for every $T\in \mathcal D^{\alpha}_{r }$. Suppose that $\alpha$ is a successor ordinal and let $\beta$ such that  $\alpha=\beta+1$.  We have $\mathcal D_r^{(\alpha)}= (\bigoplus_{n\in \N}\mathcal D^{(\beta)}_{r})^{suc, \sup}$. Induction asserts that $rank(\Omega_r(T))\leq \beta$ holds for each $T\in \mathcal D^{(\beta)}_{r})$. According to Equation~(\ref{eq:rank-sum1})   of Lemma \ref{lem:values of the rank}, $rank(\Omega_r(T))\leq \beta+1=\alpha$  for every $T$ of $(\bigoplus_{n\in \N}\mathcal D^{(\beta)}_{r})$. With Equation (\ref{eq:rank-sup}) we obtain that $rank(\Omega_r(T))\leq \alpha$ holds for each  $T\in \mathcal D^{(\beta)}_{r}$.  The induction process verifies via  Lemma \ref{lem:values of the rank} the claim that $\Omega(T)$ is topologically scattered for all trees $T\in \mathcal{D}$. The second statement  obviously for trees in $\mathcal D^{(0}_{r}$ and then is easily seen to hold throughout the induction process.  \end{proof}

\begin{theorem}\label{thm-scattered}
Let $T$ be a tree. Then the following properties are equivalent:
\begin{enumerate}[{(i)}]
\item $T$ is scattered;
\item $\Omega(T)$ is topologically scattered; 
\item $T\in \mathcal D$. 
\end{enumerate} 
\end{theorem}

\begin{proof}
We proceed by showing that $(i)\Rightarrow (iii)\Rightarrow (ii)\Rightarrow (i)$. 

$\neg iii\Rightarrow \neg i$. Let $T\not\in \mathcal{D}$. Choose some vertex in $V(T)$ and name it $r$.   The rooted tree $(T,r)$ is not in $\mathcal{D}_r'$  because $T\not\in \mathcal{D}$.  Let $P(T,r)$ be the corresponding poset with order denoted by $\leq_r$. Let $F$ be the set of $y\in V(T)$ for which the set $\{x\geq_r y: x\in V(T)\}$ induces a subtree of $T$ which is an element of $\mathcal{D}_r$ with root at $y$. We claim that there are two incomparable elements $y_0$ and $y_1$ in $V(T) \setminus F$. For  otherwise, $V(T)\setminus F$ reduces to a path and $T$ is a sum over this path of members of $\mathcal D_r$ hence belongs to   $\mathcal D_r$. Applying this property repeatedly we get an embedding of the dual of the ordered binary tree and hence,  from Lemma \ref{lem:scattered-binary},  a subdivision of the binary tree.

$(iii)\Rightarrow (ii)$.  According to Lemma \ref{prop:scattered}, for each ordinal $\alpha$, and each rooted tree $T\in \mathcal D^{(\alpha)}_{r}$,  $\Omega_r (T)$  is topologically scattered.   

$\neg i\Rightarrow \neg ii$. Let $A$ be a subtree of $T$ isomorphic to
a subdivision of\/ $\mathrm{T}_2$. Let $r$ be the root of $A$ (a
vertex of degree $2$ within $A$.) Then $\Omega(A,r)$ identifies to the
Cantor space, hence is not topologically scattered. (No one-way path
in $\Omega_r(A)$ for example is isolated in $\Omega_r(A)$ and hence
not isolated in $\Omega_r(T)$.) Or, it is easy to see that
$\Omega_r(A)$ is a closed subspace of $\Omega_r(T)$, hence this space
is not topologically scattered.
\end{proof}

\begin{proposition}
For each ordinal $\alpha$ a tree  $T$ is an element of $\mathcal D^{(\alpha)}_{r}$ if and only if $\Omega_r(T)$ is topologically scattered and $rank(\Omega_r(T))\leq \alpha$. 
\end{proposition}
\begin{proof} The direct implication is Lemma \ref{prop:scattered}.  Suppose that the reverse implication does not hold. Then there is a minimum ordinal   $\alpha$  such that for some rooted tree $T$, the set $\Omega_r(T)$ of one-way paths is topologically scattered and $rank(\Omega_r(T))\leq \alpha$ but $T\not \in  \mathcal D^{(\alpha)}_{r}$. In this case $rank(\Omega_r(T))=\alpha$. Let $P(T,r)$ be the  poset associated with  $T$, which has $r$ as  largest  element  and  the order denoted by $\leq_r$. Let $F$ be the set of vertices  $y\in V(T)$ such that the subtree of $T$ induced by the set $\{x\leq_r y: x\in V(T)\}:=\downarrow x$ of vertices and rooted at $y$ is in  $\mathcal D_r^{(\alpha)}$.  Since $T\not \in  \mathcal D^{(\alpha)}_{r}$ the vertex $r\not \in F$ and  hence  $F$ is  a proper  initial  segment of $P(T,r)$ according to the second claim of Lemma \ref{prop:scattered}.   Note that for $x\in V(T)$ the rank of the set of ends of the subtree induced by  the set of vertices $\downarrow x$ is less than or equal to $\alpha$.  

We claim that for every $x\in V(T) \setminus F$ there are two
incomparable elements $y_0$ and $y_1$ in $V(T) \setminus F$ with
$y_0,y_1\leq x$. For otherwise the initial segment $(\downarrow x)
\setminus F$ is a chain in the poset $P(T,r)$, hence there is a path
$C:=x=x_0 \sim \dots \sim x _n\sim \dots$. Let $Y_n$ be the set of
neighbours $y$ of $x_n$ which are distinct from $x_{n+1}$. Note that
$Y\subseteq F$.  Hence the connected component $T\setminus \{x_{n}\}
(y)$ of $T\setminus \{x_n\}$ containing $y$ is an element in $
\mathcal D_r^{(\alpha)}$ rooted at $y$. It follows that the tree
$T'_n=((T,C)(x_n), x_n)$ is an element of $\{T\setminus\{x_n\}(y):y\in
Y_n\}^{suc, \sup}$ and hence an element of
$\mathcal{D}_r^{(\alpha)}$. (Remember here that $(T,C)(x_n)$ is the
connected component of $x_n$ in $T\setminus(C\setminus \{x_n\})$. And
then $T'_n$ is this tree rooted at $x_n$.) Let $T'$ be the tree
induced by $\downarrow x$ rooted at $x$.  If $C$ is finite it is
easily seen that then the tree $T'$ is an element of $\{T'_n: x_n\in
V(C)\}^{suc,\sup}$ and hence an element of $\mathcal{D}_r^{(\alpha)}$,
contradicting the stipulation $x\in V(T) \setminus F$.

Hence $C$ is an infinite path and the tree $T'$ can be written as $T'=
\bigoplus_{\{x_n: n\in \N\}} T'_n$.  Let $\alpha_n$ be the rank of
$\Omega_{x_n}(T'_n)$.  Via the minimality of $\alpha$, we have
$rank(\Omega_x(T'))=\alpha$. Each $T'_{n}$ belongs to $\mathcal
D_r^{(\alpha)}$.  We apply Equation~(\ref{eq:rank-sum1}).  The set $M$
of integers $n$ such that $\alpha_n=\alpha$ must be finite for
otherwise $rank(\Omega_x(T'))=\alpha+1$.  Let $n_0=0$ if there is no
$n$ with $\alpha_n=\alpha$, and otherwise let $n_0=\max M$.  Let $T''$
be the tree induced by $\downarrow x_{n_0+1}$ rooted at
$x_{n_0+1}$. By construction $T''\not\in \mathcal{D}_r^{(\alpha)}$ but
it can be written as $T''=\bigoplus_{\{x_n: n_0< n\in \N\}} T'_n$. If
$\sup\{\alpha_n: n_0<n\in \N\}=\alpha$, then because $\alpha_n<\alpha$
for all $n_0<n\in \N$, the sequence $(T'_n)_{n_0<n\in \N}$ does not
have property $(\ast)$. This is not possible because then, according
to Equation (\ref{eq:rank-sum1}), we would have $rank
\Omega(T'')=\alpha +1$; implying $\sup\{\alpha_n: n_0<n\in
\N\}=\beta<\alpha$. Hence $T''\in \mathcal{D}_r^{(\beta+1)}\subseteq
\mathcal{D}_r^{(\alpha)}$. Again a contradiction.

Hence indeed, for every $x\in V(T) \setminus F$ there are two
incomparable elements $y_0$ and $y_1$ in $V(T) \setminus F$ with
$y_0,y_1\leq x$.  But then $T$ embeds a copy of the dual of the
ordered binary tree, and hence, from Lemma \ref{lem:scattered-binary},
a subdivision of the binary tree, contradicting the fact that
$\Omega_r(T)$ is topologically scattered.
\end{proof}
    
\begin{corollary}
 For each ordinal $\alpha$ there is a scattered tree whose space of ends has ordinal rank  $\alpha$. 
 \end{corollary}
\section{An  extension of a result of  Polat and Sabidussi}\label{section-polat-sabidussi}
We extend Theorem 3.1 of \cite{polat-sabidussi}.

\begin{lemma}\label{lem:limit scattered rank0}
Let $T$ be a scattered tree such that $\alpha:= rank(\Omega(T))$ is a non-zero limit ordinal and $x\in V(T)$. Then the following properties are equivalent:
\begin{align*} 
(i)\hskip 3pt  &\text{Either $x$ has at least two neighbours $y,y'$ such that}\\ 
&\text{$rank(\Omega (T\setminus\{x\}(y)))= rank(\Omega (T\setminus\{x\}(y')))= \alpha$ or }\\
&\sup \{rank(\Omega (T\setminus\{x\}(y))): rank(\Omega (T\setminus\{x\}(y)))<\alpha \; \text{and}\; y\sim x\}= \alpha. \\
(ii)\hskip 3pt &\text{There are two disjoint sets $X_1$ and $X_2$ of neighbours of $x$ such that $\alpha=$}\\
&\sup \{rank(\Omega (T\setminus\{x\}(y))): y\in X_1\}=\sup \{rank(\Omega (T\setminus\{x\}(y))): y\in X_2\}.
\end{align*}   
\end{lemma}

For $rank(\Omega(T))$  a non-zero limit ordinal let $Lim(T)$ be the set of vertices $x\in V(T)$ such that $(i)$ or $(ii)$ holds. 
We list some simple properties of $Lim(T)$.

\begin{lemma}\label{complement} 
For $\alpha=rank(\Omega(T))$  a non-zero limit ordinal  and $a\in V(T)$: 

\begin{enumerate}
\item $a\in V(T)\setminus Lim (T)$ iff there is a  vertex $y\in V(T)$ such that $y\sim a$ and $rank(\Omega (T\setminus\{a\}(y)))=\alpha$ and some ordinal $\beta<\alpha$ such that $rank((\Omega (T\setminus\{a\})(y'))\leq\beta$ for every other $y'\sim a$. 

\item  If $a\in Lim(T)$ and $a\sim y$ and $rank(\Omega (T\setminus\{a\}(y)))=\alpha$, then $y\in Lim(T)$.
\item  $Lim(T)= \{a\}$   iff $rank(\Omega (T\setminus\{a\}(y)))< \alpha$   for every $y\sim a$. 

\item  Let  $a\in Lim(T)$. Then $a$  has degree $1$ in  $Lim(T)$ iff
 there is a unique $y\sim a$ such that $rank(\Omega
 (T\setminus\{a\}(y)))= \alpha$ and $\sup\{rank(\Omega
 (T\setminus\{y\}(y'))): \text{$y'\sim y$ and $ y'\not =a$}\}=
 \alpha$.
\end{enumerate}

\end{lemma}
\begin{proof}
Item (1). Let $a\in V(T)$. For each neighbour $y$ of $a$, let $T'(y,x)$ be the rooted tree obtained from $T\setminus\{a\}(y)$ by adding $x$ as a new node and a root. Then  $T= \bigvee_{y \simeq a} T'(y,x)$, hence by Equation \ref {eq:rank-sup}:
\begin{align*}
&\alpha= rank(\Omega(T)) = \\
&sup \{rank(\Omega (T'(y,x)): y\sim a\}= \sup \{rank(\Omega (T\setminus\{a\}(y))): y\sim a\}.
\end{align*}  
Let $S(a):= \{y\sim a:  rank(\Omega (T\setminus\{a\}(y)))<\alpha$. If $a$ has either no neighbour outside $S(a)$ or at least two neighbours outside $S(a)$, then  $a\in Lim(T)$. If $a$ has just one neighbour outside $S(a)$, then $a\in Lim (T)$   iff $\sup\{rank(\Omega (T\setminus\{a\}(y))): y\sim a\}=\alpha$. 

\vskip 5pt\noindent
Item (2).  We apply Item (1). If $y\not \in Lim(T)$ there is a vertex $z\sim y$  such that $rank(\Omega (T\setminus\{y\}(z)))=\alpha$ and some ordinal $\beta<\alpha$ such that $rank(\Omega (T\setminus\{a\}(y')))\leq\beta$ for every other $y'\sim y$. If $z\not = a$,  then $rank(\Omega (T\setminus\{y\}(a)))<\alpha$; but then $a\not \in Lim(T)$, a contradiction. Hence $z=a$ and $rank(\Omega (T\setminus\{z\}(y)))\leq\beta<\alpha$. 

\vskip 5pt\noindent
Item (3) This is clear. 

\vskip 5pt\noindent
Item (4) Suppose that $a$ has degree $1$ in  $Lim(T)$. The  only neighbour $y$ of $a$ in $Lim(T)$ has the stated property. 
\end{proof}

\begin{lemma} \label{lem:limit scattered rank}Let $T$ be a scattered tree such that $\alpha:= rank(\Omega(T))$ is a non-zero limit ordinal. Then:

\noindent a) 
$Lim (T)$ is preserved by every embedding of $T$. 

\noindent b) 
$Lim (T)$ is a non-empty rayless subtree of $T$. 
\end{lemma} 

\begin{proof} 
a) Let $a\in Lim(T)$ and $f$ an embedding of $T$.  Then
$rank(\Omega_a(T\setminus \{a\}(x)))
rank(\Omega_{f(a)}(T\setminus\{f(a)\}(f(x)))$ for all $x\sim a$.
Which, with $(i)$ of Lemma \ref{lem:limit scattered rank0}, implies
that $f(a)\in Lim(T)$.

\vskip 5pt\noindent
b) We prove first that $Lim(T)$ is rayless.  Let  $C:=a_0, \dots, a_n, \dots$ be a one-way infinite path. We claim that $C$ is not contained in $Lim(T)$. Indeed, we have  $T=\bigoplus_{a_n\in C} (T,C)(a_n)$. Since $\alpha$ is a limit ordinal, Lemma \ref{lem:values of the rank} asserts that  there is some integer $n_0$ and some ordinal $\beta<\alpha$ such that $rank(\Omega(T({a_n})))\leq rank(\Omega(T(a_{n_0})))=\alpha$ for $n\leq n_0$ and 
$rank(\Omega(T({a_n})))\leq \beta$ for $n>n_0$. It follows that $rank(\Omega(T\setminus \{a_{n_0}\}(a_{n_0+1})))<\alpha$. This implies that  $a_{n_0+1}\not \in  Lim(T)$, proving our claim.

Next we prove that   $Lim(T)$ is non-empty. 

Let $a\in T\setminus Lim(T)$. We show that there is an integer $n\in
\N$ and a sequence $a_0, \dots, a_i, \dots a_n$ with $a_0:=a$ and
$rank (\Omega(T\setminus\{a_i\} (a_{i+1})))=\alpha$ for $i<n$ and with
$a_i\in T\setminus Lim(T)$ and $a_n\in Lim(T)$.  Indeed, let $a_1$ be
the unique vertex $y$ given by Lemma \ref{complement}. If $a_1\in
Lim(T)$, set $n:=1$. Otherwise, suppose $a_0, \dots, a_{i+1}$ be
defined, with $rank (\Omega(T\setminus\{a_i\} (a_{i+1})))=\alpha$. If
$a_{i+1}\in Lim(T)$ set $n:=i+1$. Otherwise, observe with Lemma
\ref{complement} that there is a unique neighbour $b$ of $a_{i +1}$
such that $rank (\Omega(T\setminus\{a_{i+1}\} (b)))=\alpha$.  Since
furthermore, $rank(\Omega (T\setminus\{x\}(y')))\leq\beta$ for every
other $y'\sim a_{i+1}$ and some ordinal $\beta<\alpha$, it follows that
$b\not =a_{i+1}$. If the process does not stop, then $T$ is the sum
$\bigoplus_{\{a_n:n\in \N\}} T_{a_n}$, where $T_{a_n}:=T\setminus
\{a_{n+1}\}(a_n)$ for each $n\in \N$. Since $rank
(\Omega(T\setminus\{a_n\} (a_{n+1})))=\alpha$ for each $n\in \N$,
Equation \ref{eq:rank-sum1} of Lemma \ref{lem:values of the rank}
asserts that $rank(\Omega(T))=\alpha+1$. This is impossible.

Finally, we prove that $Lim(T)$ is connected. 

Suppose not.  Let $n$ be the least integer such that there are two
vertices $a,b$ in two different connected components of $Lim(T)$ such
that $d_T(a, b)= n$.  Let $a:=a_0,a_1, \dots, a_n:=b$ be the unique
path connecting $a$ and $b$.  We have $n\geq 2$, and $a_i \not \in
Lim(T)$ for every $i$, $1\leq i\leq n-1$. Since $a_1\not \in Lim(T)$,
there is a unique vertex $c\in V(T)$ such that $c\sim a_1$ and
$rank(\Omega (T\setminus\{a_1\}(c)))=\alpha$, and some ordinal
$\beta<\alpha$ such that $rank(\Omega
(T\setminus\{a_1\}(y')))\leq\beta$ for every other $y'\sim
a_1$. Suppose $c\not =a_0$. Then $rank(\Omega
(T\setminus\{a_1\}(a_0)))\leq\beta$. Hence $rank(\Omega
(T\setminus\{a_0\}(y')))\leq \beta$ for every $y'\sim a_0$, $y'\not =
a_1$. Thus $a_0\not \in Lim(T)$. Impossible. Since $c=a_0$,
$rank(\Omega (T\setminus\{a_0\}(a_1)))<\alpha$, but then
$T\setminus\{a_0\}(a_1)$ does not contain a vertex in $Lim(T)$,
contradicting that $b\in Lim(T)$.  \end{proof}

\begin{lemma}\label{lem:raylesslimit} Let $T$ be a non-empty scattered tree and  $\alpha:= rank(\Omega(T))$. If $\alpha$  is either a limit ordinal or  a successor ordinal $\alpha'+1$ and $\vert \Omega^{(\alpha')}(T)\vert \geq 3$, then some non-empty rayless subtree $A$ of $T$ is preserved by every embedding of $T$. 
\end{lemma}
\begin{proof}
If $\alpha=0$, then $\Omega(T)=\emptyset$, that is $T$ is rayless and there is nothing to prove. 
Suppose $\alpha>0$. We consider two cases: 

\vskip 5pt\noindent
Case 1. $\alpha$ is a limit ordinal. The 
assertion  follows from Lemma \ref{lem:limit scattered rank}. 

\vskip 5pt\noindent
Case 2.  $\alpha$ is a successor ordinal $\alpha'+1$ and $\vert
\Omega^{(\alpha')}(T)\vert \geq 3$. We may reproduce verbatim the
proof of Theorem 3.1 \cite {polat-sabidussi}, in which the authors
prove that there is some rayless tree $A$ which is preserved by every
automorphism of $T$. For every end in $e\in \Omega^{(\alpha')}(T)$
choose a one-way infinite path $C_e$ so that $V(C_e)\cap
V(C_e')=\emptyset$ for $e\not=e'\in \Omega^{(\alpha')}(T)$. This is
possible because every end in $\Omega^{(\alpha')}(T)$ is isolated
within $\Omega^{(\alpha')}(T)$. Let $W=\{C_e: e\in
\Omega^{(\alpha')}(T)\}$. Choose a vertex in each of the sets $V(W)$
and let $A$ be the set of those vertices. Let $G$ be the subtree of
$T$ induced by the set of vertices which are on the path from a vertex
in $A$ to a vertex in $A$ together with the vertices in $\bigcup_{C\in
W}V(C)$.  The subtree $G$ contains a vertex $v$ having degree at least
three because $\vert \Omega^{(\alpha')}(T)\vert \geq 3$. For every
$a\in A$ let $a_v$ be the last vertex on the path from $v$ to $a$ of
degree larger than or equal to three. Let $a'$ be the first vertex on
the path from $a_v$ to $a$ not equal to $a_v$. Then $a'$ has valence
two in $G$. Let $A'=\{a': a\in A\}$ and let $C_a$ be the one way
infinite path with endpoint $a'$. Note that for every $a'\in A'$,
every one-way infinite path in the end containing $C_a$ has $C_a$ as a
subpath. This implies, because every embedding preserves $\vert
\Omega^{(\alpha')}(T)\vert \geq 3$, that our embedding  maps every path in $W$ into a
path in $W$ and hence preserves $G$.  Because every embedding maps
every vertex to a vertex of the same degree it has to preserve $A'$
and hence the subtree $G'$ of $G$ induced by the set of vertices
$V(G)\setminus \bigcup_{C\in W}V(C)$.

It remains to prove that $G'$ does not contain a one-way infinite path, also called a ray.  This is proven in  \cite {polat-sabidussi}. For completeness sake we provide their proof here: Suppose there is a ray $R\in G'$. Every vertex in $R$ is on a path from a vertex in $A'$ to a vertex in $A'$, implying that $R$ does not have a terminal end so that each vertex on it has valence two in $G'$. If $x\in V(R)$ and the valence of $x$ in $G'$ is at least three, then there is a path from $x$ to some $a'\in A'$. It follows that $R$ is not an isolated ray in $G$, that is $R\in \Omega^{(\alpha)}(T)=\emptyset$, a contradiction. 
\end{proof}  

Polat and Sabidussi \cite{polat-sabidussi}, (Theorem 2.5), proved that
every non-empty rayless tree $T$ has a vertex or an edge fixed by
every automorphism of $T$. One of their ideas is to successively
change $T$ to the tree $\Phi(T)$ generated by the vertices of infinite
degree. Clearly $\Phi(T)$ is preserved by every embedding, not just
automorphism of $T$. It can be easily checked that the arguments in
\cite{polat-sabidussi} can be extended from automorphisms to
embeddings. Another way to verify the extension of the Theorem of
Polat and Sabidussi to embeddings is to specialize Theorem 11.5 of
Halin, see \cite{halin2}, from graphs to trees. In any case we obtain:

\begin{lemma}\label{lem:Polat_Sabidussi_Halin}
Every non-empty  tree $T$ which does not contain a one-way infinite path,  has a vertex or an edge preserved by every embedding of $T$. 
\end{lemma}

\subsection{Proof of Theorem \ref{main1}. } \label{prooftheorem1.1}
Let $T$ be a scattered tree.  According to Theorem \ref{thm-scattered}, the space $\Omega(T)$ is scattered. Let $\alpha:=rank(\Omega(T))$ be its Cantor-Bendixson rank. There are two cases: 

\vskip 5pt
\noindent 
\textbf{Case 1:}  $\alpha$ is a limit ordinal or   
is a successor ordinal $\alpha'+1$ and $\vert
\Omega^{(\alpha')}(T)\vert \geq 3$. According to Lemma
\ref{lem:raylesslimit}, some non-empty rayless subtree $A$ of $T$ is
preserved by every embedding of $T$. According to Lemma
\ref{lem:Polat_Sabidussi_Halin} the subtree $A$ contains a vertex or
an edge preserved under every embedding of $A$, and  hence of $T$.

\vskip 5pt
\noindent    
\textbf{Case 2:}  $\alpha$ is a successor ordinal $\alpha'+1$ and $\vert \Omega^{(\alpha')}(T)\vert \leq 2$.  The set $\Omega^{(\alpha')}(T)$ is preserved by every embedding. We get the required conclusion. 
 \hfill $\Box$

\section{Embeddings of labelled trees}\label{section: labelled}

Let $Q$ be a quasi ordered set (qoset), that is, a set equipped with a
quasi order that we denote by $\leq$. We set $q\equiv q'$ if $q'\leq
q\leq q'\in Q$ and $equiv(q):=\{q'\in Q: q'\equiv q\}$. A
$Q$-\emph{labelled graph} is a pair $(G, \ell)$ where $\ell$ is a map
from $V(G)$ into $Q$. If $(G, \ell)$ and $(G',\ell')$ are two
$Q$-labelled graphs, a $Q$-\emph{isomorphism} of $(G, \ell)$ onto
$(G',\ell')$ is an isomorphism $f$ of $G$ onto $G'$ such that
$\ell(x)=\ell(f(x))$ for every $x\in V(G)$.  A $Q$-\emph{embedding} of
$(G, \ell)$ into $(G',\ell')$ is an embedding $f$ of $V(G)$ into
$V(G')$ such that $\ell(x)\leq \ell(f(x))$ for every $x\in V(G)$. We
say that $(G, \ell)$ and $(G',\ell')$ are \emph{isomorphic} and we set
$(G, \ell)\simeq(G',\ell')$ if there is a $Q$-isomorphism of $(G,
\ell)$ to $(G',\ell')$. We say that they are \emph{equimorphic} if
each of $(G, \ell)$ and $(G',\ell')$ are $Q$-embeddable in the other;
in which case we set $(G, \ell)\equiv (G',\ell')$.  If $(G, \ell)$ is
a labelled tree, then as before we set $twin(G,\ell) $ for the set of equimorphic 
$Q$-labelled trees up to isomorphism.

If $(G,\ell)$ is a one-way labelled path we will assume that $V(G)$ has been enumerated with the numbers in $\N$  in a natural way and then, in order to simplify notation,  pretend to have identified $V(G)$ with $\N$ along this enumeration. For a two-way infinite path we will use $\Z$ in the place of $\N$. Let $(G,\ell)$ be a one-way or two-way infinite path labelled by a poset $Q$. The labelling $\ell$ is {\em $r$-periodic} for $1\leq \vert r\vert\in \N$ if $\ell(n)\leq \ell(n+r)$ for all $n\in \N$. The number $r$ for which $\ell$ is $r$-periodic with $|r|$ minimal is the {\em period} of $\ell$. If $\ell$ is $r$-periodic  with $\vert r\vert$ minimal and $\ell$ is also $-r$ periodic then the period of $\ell$ is $\vert r\vert$.   If $\ell$ has period $p$, then $\mathfrak{d}_\ell:=\vert \{n\in V(G): \ell(n)<\ell(n+p)\}\vert$. 

We may note that the tree alternative property fails badly for  labelled trees, even with the assumption that the qoset is a poset. In fact: 

Let $C$ be the path on the set $\N$ of non-negative integers and
$\ell$ be a labeling by a two element antichain $Q:=\{a, b\}$. Then
$\vert twin(C,\ell) \vert=1$ if $\ell$ is constant or $\ell$ is not
periodic and $\vert twin(C, \ell) \vert =p$ if $p$ is the period of $\ell$.
However:

\begin{theorem}\label{lemma:paths}
Let $(G,\ell)$ be a one-way infinite path or a two-way infinite path
labelled by a poset $Q$. If $Q$ has a minimum, say 0, then the tree
alternative property holds. If $\ell$ is not periodic, then $\vert
twin(G,\ell) \vert =1$ and if $\ell$ is periodic and
$\mathfrak{d}_\ell$ is infinite then $\vert twin(G,\ell) \vert \geq
2^{\aleph_0}$. 
countable, then $twin(G,\ell)$ is countable.

If $\ell$ is periodic and $G$ a one-way infinite path, then
$twin(G,\ell)$ is infinite unless $\ell$ is the constant map with
image $\{0\}$, in which case $\vert twin(G,\ell) \vert=1$. If $\ell$ is periodic
and $G$ a two-way infinite path, then $twin(G,\ell)$ is infinite
unless $\mathfrak{d}_\ell=0$ in which case $\vert twin(G,\ell) \vert =1$.

\end{theorem}

\begin{question}
Does the tree alternative property hold for labelled trees labelled with a poset having a least element?
\end{question}

The proof of Theorem \ref{lemma:paths} relies on the following lemmas and Theorem \ref{thm:continuum} below. 
\begin{lemma}\label{lem:colormodif} Let $(G, \ell )$ and  $(G', \ell' )$ be two $Q$-labelled connected graphs and $g$ be an embedding of $(G, \ell )$ into  $(G', \ell' )$. Let $\ell'':V(G')\rightarrow Q$ such that:
\begin{equation}\label{eq:ineg}
\ell(g^{-1}(y))\leq \ell''(y)\leq \ell'(y)
\end{equation}
for every $y\in V(G')$, with the stipulation that the first inequality holds if $y$ does not belong to the range of $g$. Then:

\begin{equation}
\text{$(G,\ell)$ embeds into  $(G',\ell'')$ embeds into  $(G', \ell')$.} 
\end{equation}

As a consequence, if $(G, \ell )$ and  $(G', \ell' )$ are equimorphic, then  $(G, \ell )$ and  $(G', \ell'')$ are equimorphic. 
\end{lemma}
\begin{proof} From the first inequality in (\ref{eq:ineg}), the map $g$ is an embedding of $(G, \ell )$ into   $(G', \ell'' )$, whereas from the second inequality, the identity map is an embedding of $(G', \ell'' )$ into   $(G', \ell' )$. 
\end{proof}

We need the following result (see Lemma 4, p.39, \cite{laflamme-pouzet-sauer-zaguia}). For reader convenience, we give a proof.
\begin{lemma}\label{pz2} Let $S$ be an infinite subset of $\N$. There is a family $\mathcal A$ of $2^{\aleph_0}$ subsets of $S$ such that for every pair $A, A'$ of distinct subsets of $\mathcal A$,  no translate  of an infinite subset of $A$  is  almost included into $A'$.
\end{lemma}
\begin{proof} 
Start with $X:=\{x_n\in S, n \in \omega\}$ where $x_0=\min(S)$ and
$x_{n+1}\geq x_n + n$ (one just needs increasing gaps).  Now, let
$\mathcal A$ be an almost disjoint family of $2^{\aleph_0}$ infinite
subsets of $X$.  For any $A\in \mathcal A$, and $n>0$, $A+n$ is almost
disjoint from $X$, and thus almost disjoint from any other $A'$.
\end{proof}

\begin{lemma}\label{lem:manylabelled} 
Let $Q$ be a poset with a least element $0$. Let $(G, \ell )$ and
$(G', \ell' )$ be two equimorphic $Q$-labelled locally finite
graphs. Let $g$ be an embedding of $(G, \ell )$ into $(G', \ell' )$
and let $D:= \{y\in range (g): \ell(g^{-1}(y))<\ell'(y)\}\cup \{y\in
V(G')\setminus range(g): \ell'(y)\not =0\}$.  If $D$ is infinite, then
$\vert twin (G, \ell) \vert \geq 2^{\aleph_0}$.
\end{lemma}

\begin{proof} 
Our aim is to find a family $\mathcal B$ of $2^{\aleph_0}$ subsets $B$
of $D$ such that the $Q$-labelled graphs $(G', \ell '_{B})$ defined by
$\ell'_{B}(y):= \ell(g^{-1}(y))$ are pairwise isomorphic if $y \in
B\cap range (g)$ and $\ell'_{B}(y)=0$.  Indeed, according to Lemma
\ref{lem:colormodif}, these $Q$-labelled graphs are equimorphic to
$(G, \ell)$. For this purpose, we will use the following notation. 

Let $x\in V(G')$ and $n\in \N$, set $B_{G'}(x, n):= \{y\in V(G'):
d_{G'}(x,y)=n\}$. Let $X\subseteq V(G')$, set $Spec(X,x):=\{n\in \N:
X\cap B_{G'}(x,n)\not = \emptyset\}$. Now, fix a vertex of $G'$, say
$r$.  Let $S:= Spec (D, r)$. Since $G'$ is locally finite, $S$ is an
infinite subset of $\N$. Let $\mathcal A$ be a family of
$2^{\aleph_0}$ subsets $A$ of $S$ satisfying the properties of
Lemma~\ref{pz2}. For each $A\in \mathcal A$, set $B:=
\{x\in D: d_{G'}(r, x)\in A\}$, hence $Spec(B,r) =A$. Let $\mathcal
B':= \{ B: A\in S\}$. This family yields a collection of
$2^{\aleph_0}$ $Q$-labelled graphs.  Divide this collection into
isomorphism classes. 

We claim that each class is countable. Picking a
representative in each class, we will get a family $\mathcal B$ as
described above. To prove our claim, suppose by contradiction that some
isomorphy class, say $\mathcal C$, is uncountable. For each $(G', \ell
'_{B})\in \mathcal C$ pick $x_{B}\in V(G')$. Since $G'$ is locally
finite, $V(G')$ is countable, hence we may find an uncountable
subfamily $C'$ such that all labelled graphs $(G', \ell '_{B},
x_{B})\in \mathcal C'$ are pairwise isomorphic and, furthermore all
$x_{B}$ are equal to the same element $x\in V(G')$ (indeed, fix a
labelled graph $(G', \ell '_{B})\in C$ and observe that if $(G',
\ell'_{B'})$ and $(G', \ell'_{B''})$ are isomorphic to $(G', \ell
'_{B})$ and the images of $x_{B'}$ and $x_{B''}$ by some isomorphisms
are the same, then the labelled graphs $(G', \ell '_{B'}, x_{B'})$ and
$(G', \ell'_{B''}, x_{B''})$ are isomorphic).

\begin{claim} \label {claimSpec}
If $(G', \ell '_{B'}, x)\simeq (G', \ell '_{B''}, x)$, then $Spec(B', x)=Spec (B'',x)$. 
\end{claim}

\noindent{\bf{Proof of Claim \ref{claimSpec}.}} 
Assume that $(G', \ell '_{B'}, x)\simeq (G', \ell '_{B''}, x)$. Let
$n\in \N$. We prove that if $n\not \in Spec(B', x)$, then $n\not \in
Spec(B'', x)$; that is, $B''\cap B_{G'}(x, n)=\emptyset$ implies
$B''\cap B_{G'}(x, n)=\emptyset$. Supposing $B''\cap B_{G'}(x,
n)=\emptyset$ we have $\ell_{B''}(y)=\ell(y)$ for every $y\in
B_{G'}(x, n)$. Let $h$ be an isomorphism of $(G', \ell '_{B''}, x)$
onto $(G', \ell '_{B''}, x)$. For all $y\in B_{G'}(x, n)$ we have:
\begin{equation} \ell (y)=\ell_{B''}(y)=\ell_{B'}(h(y)\leq \ell (h(y)). 
\end{equation}

Since $h$ is a self map on $V(G')$ we may iterate it. The iterates
satisfies the inequalities above. Since $G'$ is locally finite,
$B_{G'}(x, n)$ is finite too. Hence, some iterate of $h$ is the
identity, and therefore $\ell(y)= \ell(h(y))$ for all $y\in B_{G'}(x, n)$. It
follows that $\ell_{B'}(y)=\ell (y)$ for all $y\in B_{G'}(x,
n)$. Hence $B'\cap B_{G'}(x, n)=\emptyset$, proving our claim.  \hfill
$\Box$

\begin{claim}\label {claimtranslate} 
If $Spec(B', x)=Spec (B'',x)$, then one of the  sets $A':=Spec(B', r)$, $A'':=Spec (B'',r)$ contains an infinite translate of the other.
\end{claim}

\noindent{\bf {Proof of Claim \ref{claimtranslate}.}} 
Let $D':= Spec(B', x)$. For each $n\in D'$ pick $x'_n\in B',x''\in
B''$ such that $d_{G'}(x,x'_n)=d_{G'}(x,x''_n)=n$. Since $\vert
d_{G'}(x,x'_n)-d_{G'}(r,x'_n)\vert \leq d_{G'}(x, r)$ and $\vert
d_{G'}(x,x''_n)-d_{G'}(r,x"_n)\vert \leq d_{G'}(x, r)$ we may find a
increasing sequence of integers $(\varphi(n))_{n\in \N} $ such that $
d_{G'}(x,x'_{\varphi(n)})-d_{G'}(r,x'_{\varphi(n)}) =c'$ and $
d_{G'}(x,x''_{\varphi(n)})-d_{G'}(r,x''_{\varphi(n)}) =c''$ for some
integer constants $c'$ and $c''$. It follows that
$d_{G'}(r,x'_{\varphi(n)})-d_{G'}(r,x''_{\varphi(n)})$ is an integer
constant. The set $Y':= \{d_{G'}(r,x'_{\varphi(n)}): n\in \N \}$ is an
infinite subset of $A':=Spec(B', x)$ and the set $Y''$ similarly
defined is an infinite subset of $A'':=Spec(B'', x)$. One of the sets
$Y'$, $Y''$ is a translate of the other, as claimed.  \hfill $\Box$
 
To complete the proof of Lemma \ref{lem:manylabelled}, observe that
this situation is not possible, because $A', A''\in \mathcal A$.
\end{proof}

 We can rephrase Lemma \ref{lem:manylabelled} and state it as Theorem \ref{thm:continuum}: 
 
\begin{theorem}\label{thm:continuum} 
Let $Q$ be a poset with a least element $0$. Let $(G, \ell )$ be a
locally finite graph such that $\vert twin(G,\ell) \vert <2^{\aleph_0}$. Then, for
every embedding $g$ of $(G, \ell)$ into a twin $(G', \ell' )$, the set
of $x\in V(G)$ such that $\ell(x)\not =\ell'(f(x))$, and the set of
$y\in V(G')\setminus range (g)$ such that $\ell'(y)\not =0$, are
finite. 
\end{theorem}

\vskip 8pt 
\noindent {\bf Proof of Theorem \ref{lemma:paths}} \label{subsection:paths} 
Let $(G,\ell)$ be a one way infinite path. The function $\ell$ is
$r$-{\em constant} for $1\leq r\in \N$ if $\ell(n+r)=\ell(n)$ for all
$n\in \N$. If there exists such a number $r$, then $\ell$ is {\em
stepwise constant}. Note that the smallest such number $r$ is the
period of $\ell$. If $\ell$ is stepwise constant when restricted to
some terminal interval $[n,\infty)$ of $\N$, then $\ell$ is {\em
eventually stepwise constant}. Let $\ell_n$ be the restriction of the
labelling $\ell$ to the interval $[n,\infty)$.  If $\ell$ is
eventually stepwise constant there exists a smallest number $n$, the
{\em character of $\ell$}, for which $\mathfrak{d}_{\ell_n}=0$. The
character of $\ell$ is an isomorphism invariant. Note that if
$\mathfrak{d}_{\ell_n}=0$, then $\ell_n$ is stepwise constant and if
$Q$ is countable, then there are countably many eventually stepwise
constant $Q$-labelled one way paths.

If $\ell$ is not periodic, then there is no embedding $f$ of
$(G,\ell)$ with $0<f(0)$ and hence $\vert twin(G,\ell) \vert=1$. Let $\ell$ be
periodic with period $p$. If $\mathfrak{d}_l$ is infinite, then
because the function $n\to n+p$ is an embedding, it follows from
Theorem \ref {thm:continuum} that $\vert twin(G,\ell) \vert \geq 2^{\aleph_0}$.

Let $(G,\ell)$ be a two-way infinite path. The function $\ell$ is
$r$-{\em constant} for $1\leq r\in \N$ if $\ell(n+r)=\ell(n)$ for all
$n\in \Z$. If there exists such a number $r$, then $\ell$ is {\em
stepwise constant}.  Note that the smallest such number $r$ is the
period of $\ell$.  The function $\ell$ is {\em eventually stepwise
constant} if there are numbers $m$ and $n$ in $\Z$ so that $\ell$ is
stepwise constant when restricted to the one-way infinite path
$(-\infty,m)$ as well as to the one-way infinite path $[n,\infty)$.
In this case the {\em eventual period} of $\ell$ is the least common
multiple of the period of $\ell$ on $(-\infty,m]$ and the period of
$\ell$ on $[n,\infty)$.

If $\ell$ is not periodic,  there is no embedding $f$ of
$(G,\ell)$ which is a translation of $\Z$ with $f(0)\not= 0$. This
implies that every embedding is an automorphism and hence that
$\vert twin(G,\ell) \vert =1$. If $\ell$ is periodic, say with period $p$ and
$\mathfrak{d}_\ell=0$, then $\ell$ is stepwise constant, in which case
$\vert twin(G,\ell) \vert =1$. If $\mathfrak{d}_\ell$ is infinite, then because the
function $n\to n+p$ is an embedding, it follows from Theorem \ref
{thm:continuum} that $\vert twin(G,\ell) \vert \geq 2^{\aleph_0}$.  \hfill $\Box$


We look at  graphs labelled  by a qoset. 

\begin{lemma}\label{label=1} 
Let $(G, \ell )$ be a $Q$-labelled connected graph.  Let $Q_\ell:= \{
equiv(\ell (x)): x\in V(G)\}$, and $\kappa$ be the cardinality of the
set $\Phi$ of maps $\varphi$ from $Q_\ell$ into $Q$ such that
$\varphi(C)\in C$ for each $C\in Q_\ell$. Then
\begin{equation}
\kappa \leq \vert twin (G, \ell)\vert .
\end{equation}

In particular 
\begin{equation}
\vert equiv(\ell(x))\vert  \leq \vert twin (G, \ell)\vert.
\end{equation} 
for every $x\in V(G)$. 
 \end{lemma}

\begin{proof} Let  $\varphi\in \Phi$. Set $V:=V(G)$ and define  $\ell_\varphi: V \rightarrow Q$ by setting $\ell_\varphi(y)=\varphi(equiv(\ell (y)))$. The $Q$-labelled graph $G<\varphi>:= (G, \ell_\varphi)$ is equimorphic to $(G,\ell)$. Indeed, the identity map on $V$ is an embedding of $(G, \ell)$ into $G<\varphi>$ and  of $G<\varphi>$ into $(G, \ell)$. If  $\varphi \not =\varphi'$, the $Q$-labelled graphs $G<\varphi>$ and $G<\varphi'>$ cannot be isomorphic. Otherwise, let $x\in V$ such that $q:=\varphi(equiv(l(x)))\not =\varphi'(equiv(l(x)))=:q'$. We  would have  $q\not=\ell_{\varphi'}(f(x))=\ell_\varphi(x)=q$ for every  isomorphism $f$ of $G<\varphi>$ onto $G<\varphi'>$. The 
 inequality follows. 
\end{proof}

\section{Bonato-Tardif and Tyomkyn conjectures and a proof of Theorem \ref{main2}}\label{section: Bonato-tardif-level function}

\begin{lemma}\label{twin-twinrooted} 
Let $T$ a tree, $A$ be a subtree of $T$, $x\in A$ and
$T_x:=((T,A)(x), x)$. If $f(A)=A$ for every embedding $f$ of $T$, then
$\vert twin(T_x) \vert \leq \vert twin (T) \vert $.
\end{lemma}

\begin{proof} 
Let $Q$ be the set of rooted trees (considered up to isomorphy) which
are equimorphic to some $T(x)$ for $x\in A$. Let $(A, \ell)$ where
$\ell(x):=T(x)$. We claim that $\vert twin (A, \ell)\vert \leq \vert
twin(T)\vert $. Indeed, notice first that if $(A', \ell')$ is a
$Q$-labelled tree, and $(A, \ell')\equiv (A,\ell)$, then $
\bigoplus_{x\in A} \ell'(x)\equiv \bigoplus_{x\in A}\ell(x)$. Next,
let $(A',\ell')$, $(A'', \ell'')$ be two $Q$-labelled trees
equimorphic to $(A,\ell)$.  Suppose that $\bigoplus_{x\in A'}
\ell'(x)\simeq \bigoplus_{x\in A''}\ell''(x)\equiv T$. Let $f$ be an
isomorphism of $\bigoplus_{x\in A'} \ell'(x)$ onto $ \bigoplus_{x\in
A''}\ell''(x)$. The map $f$ induces an embedding from $T$ into $T$
hence, by our hypothesis, $f(A)=A$. This implies that $(A,
\ell')\simeq(A'',\ell)$, proving our claim. According to Lemma
\ref{label=1}, we have $\vert twin((T,A)(x), x) \vert \leq \vert twin
((T, \ell)) \vert $.
\end{proof} 

From Lemma \ref{twin-twinrooted} and Tyomkyn's Theorem \cite{tyomkyn} that the rooted tree alternative conjecture holds we obtain: 
\begin{corollary}
Let $T$ be a tree and $A$ be a subtree such that $f(A)=A$ for every
embedding of $T$. If $\vert twin (T) \vert <\omega$, then $\vert twin((T,A)(x), x)\vert=1$ for
every $x\in A$.
\end{corollary}

Let $T$ be a tree and $A$ be a finite subtree. If $T'$ is an other
tree and $A'$ is a finite subtree of $T'$, we will say that $(T, A)$
is embeddable into $(T',A')$ if there is an embedding of $T$ into $T'$
which maps $A$ into $A'$. If $n:= \vert A\vert$ and $(a_i)_{i<n}$ is
an enumeration of $A$ and $A'$ is an $n$-element subtree of a tree
$T'$ with an enumeration $(a'_i)_{i<n}$, we will say that $((T,
(a_i)_{i<n})$ is embeddable into $(T',(a'_i)_{i<n})$ if there is an
embedding of $T$ into $T'$ which sends each $a_i$ onto $a'_i$. This
allows to define $twin((T,A))$ and $twin( (T, (a_i)_{i<n}))$. We
define $twin((T,A))$ as the set of twins of $(T, A)$ considered up to
isomorphy, and similarly $twin((T, (a_i)_{i<n}))$.

\begin{lemma}
\begin{equation}\label {eq:labeling}
\frac {1} {n!}\cdot \vert twin( (T,   (a_i)_{i<n})) \vert  \leq \vert twin((T,A)) \vert \leq \vert twin( (T,  (a_i)_{i<n}))\vert. 
\end{equation}
\end{lemma}

\begin{proof} 
Note that the map $\varphi$ which transforms $(T, (a_i)_{i<n})$ to $(T,
A)$ preserves isomorphy, namely if $(T', (a'_i)_{i<n})\simeq (T'',
(a''_i)_{i<n})$, then $(T', A')\simeq (T'', A'')$. Hence $\varphi$
sends isomorphic types onto types. Also $(T', (a'_i)_{i<n})\equiv (T'',
(a''_i)_{i<n})$, then $(T', A')\equiv (T'', A'')$, hence $\varphi$
sends $twin((T,(a_i)_{i<n}))$ into $twin((T,A))$. This map is
surjective, proving that the second inequality holds. Indeed, let
$(T', A')\equiv (T, A)$. Let $f,g$ witness this fact, that is $f$
is an embedding of $(T',A')$ into $(T, A)$ and $g$ is an embedding of
$(T, A)$ into $(T', A')$.  Since $A$ is finite, there is a nonnegative
integer $m$ such that $(g\circ f)^{(m)}_{\restriction A}$ is the
identity on $A'$. Let $f':=(g\circ f)^{(m)}\circ g$. Let
$(a'_i)_{i<n}$ such that $a_i=f(a'_i)$. Then $(T', (a'_i)_{i<n})\equiv
(T, (a_i)_{i<n})$, proving that the map $\varphi$ is surjective. Let
$(T',A')\in twin((T,A))$. The number of non-isomorphic types in
$twin((T, (a_i)_{i<n}))$ which can be sent onto $(T', A')$ is at most
$n!$. The first inequality follows.
\end{proof}

From this type of trivial argument we cannot expect to prove that
$\vert twin((T,A)) \vert \leq \vert twin( (T, (a_i)_{i<n}))\vert$. For an example, there are
two graphs $G$ and $G'$ and subsets $A:= \{a_0, a_1\}, A':= \{a'_0,
a'_1\}$ such that $(G, (a_0,a_1) )\equiv (G', (a'_0,a'_1))$, $(G,
(a_0,a_1) )\not \simeq (G', (a'_0,a'_1))$. The graph $G$ is a path on
$\N$ with two extra vertices $a_0$ and $a_1$, with $a_0$ linked by an
edge to every even integer, $a_1$ linked by an edge to every odd
integer. The graph $G'$ is $G_{-0}$, $\alpha'_i=\alpha_i$.

\subsection{A vertex or an edge preserved}
\begin{proposition}\label{prop:edge-vertex}
Let $T$ be a tree. If there is a vertex or an edge preserved by every
embedding of $T$, then $\vert twin(T)\vert $ is $1$ or $\infty$, and if $T$ is also
locally finite, then every embedding of $T$ is an automorphism of $T$.
\end{proposition}

\begin{proof}
Case 1: There is a vertex $x$ preserved by every embedding. In this
case $twin(T)=twin ((T, r))$. According to Tyomkyn's Theorem, $\vert
twin ((T,r))\vert$ is $1$ or $\infty$. Every embedding of a locally
finite rooted tree is an automorphism, see \cite{tyomkyn} Lemma 4.

\vskip 2pt
\noindent 
Case 2: There is an edge $u:=\{x,y\}$ preserved by every embedding but there is no vertex fixed by every embedding. 
According to the inequalities in (\ref{eq:labeling}) we have:  
\begin{equation}\label {eq:labeling2}
\frac {1} {2}\cdot \vert twin( (T,   (x,y))) \vert \leq twin((T, u)) \vert \leq \vert twin( (T,  (x,y))) \vert. 
\end{equation}

According to Tyomkyn's Theorem, $\vert twin ((T,x))\vert$ is
$1$ or $\infty$. Since in our case $twin ((T,x))=twin ((T,(x,y)))$,
$\vert twin ((T,\{x,y\}))\vert$ is $1$ or $\infty$. Since $twin(T)=
twin((T,\{x,y\}))$, the result follows.  Let $T$ be locally finite and
$f$ an embedding of $T$. If $f(x)=x$, then $f(y)=y$ and hence $f$ is
an automorphism on the tree $(T,x)$ and on the tree $(T,y)$. If
$f(x)=y$ then $f(y)=x$ and hence $f^2$ is the identity,  and therefore $f$ is an
automorphism.

\end{proof}

\subsection{A two-way infinite path  preserved}

\begin{proposition} If a tree contains a two-way infinite path preserved by every embedding, then the tree alternative property holds. 
 \end{proposition}
 
\begin{proof} 
Let $T$ be a tree, $D$ be a two-way infinite path preserved by every
embedding. Suppose that $\vert twin (T) \vert<\omega$. According to
Lemma \ref {twin-twinrooted} and Tyomkyn's Theorem, $\vert
twin((T,D)(x), x)) \vert =1$ for every $x\in D$.  Write $T=
\bigoplus_{x\in D} ((T,D)(x), x)$. Let $Q$ be the set of rooted trees
which embed in some $((T,D)(x), x)$ for some $x\in D$; we quasi order
$Q$ by embeddability of rooted trees.  Let $(D, \ell)$ be the labelled
two-way infinite path where $\ell: D \rightarrow Q$ is defined by
$\ell(x):= ((T,D)(x), x)$. Each embedding of a labelled two-way
infinite path yields an embedding of $T$ and conversely. Hence
$twin(T)= twin (D, \ell)$. The conclusion follows from Theorem
\ref{lemma:paths}.

%
%
\end{proof}

\subsection{One end preserved}  

Let $\mathfrak{E}$ be the set of  trees having an end  preserved forward by every embedding.  We will use the notions introduced in Subsection \ref {subsection:valuation}. 
In particular, each $T\in \mathfrak{E}$  has a level function, say $lev_T$.


\vskip 5pt
\noindent
{\bf Case \textbf{1}}: Every embedding of $T$ preserves the level function. This amounts to the fact that the end $e$ is almost rigid. Let $\mathfrak{G}$ be the set of trees $T$ in $\mathfrak{E}$ with the property that every embedding of $T$ preserves $lev_T$. 

\begin{proposition}\label{claim:finalindexpr}
If $T\in \mathfrak{G}$, then the set $twin(T)$ is infinite or a
singleton.If in addition $T$ is locally finite, then every embedding
of $T$ is an automorphism of $T$, hence $\vert twin(T)\vert=1$.
\end{proposition} 

Note that if $T$ is locally finite, every embedding which has a fixed
point is surjective. Since $T$ has a level function preserved by every
embedding of $T$, every embedding has a fixed point. Note also that if
$T$ is scattered and the ray is regular, then by Theorem
\ref{main1-2}, $T$ has a vertex fixed by every embedding. Hence in
this case the value of $\vert twin(T)\vert$ is given by Proposition
\ref{prop:edge-vertex}. The general case needs more work.

For $n$ a number in $\Z$ let $\mathcal{L}_n(T):=\{x\in V(T):
lev_T(x)=n\}$ and let $\mathcal{L}_{\geq n}(T):=\{x\in V(T):
lev_T(x)\geq n\}$. Two trees $T$ and $T'$, with the same origin $o\in
V(T)\cap V(T')$ for the level functions, are \emph{$n$-equimorphic} if
$\mathcal{L}_{\geq n}(T)$ is equal to $\mathcal{L}_{\geq n}(T')$ and
if for every $x\in \mathcal{L}_n(T)$ the rooted trees $T(\to x)$ and
$T'(\to x)$ are equimorphic.  An embedding of a tree $T$ into a tree
$T'$ is \emph{level preserving} if $lev_T(x)=lev_{T'}(f(x))$ for all
$x\in V(T)$.

\begin{claim}\label{claim:presindup}
Let $T$ and $T'$ be two trees which  are $n$-equimorphic, with common origin $o$ for the level function.  If $T\in \mathfrak{G}$, then  $T'\in \mathfrak{G}$ and  every embedding of $T'$ into $T$ is level preserving and every embedding of $T$ into $T'$ is level preserving.   
\end{claim}
\vskip-1pt

\noindent{\bf Proof of Claim \ref{claim:presindup}.}  Let $f$ be an embedding of $T$ into $T'$ with $f(x)=x$ for all $x\in \mathcal{L}_{\geq n}$ and $f$ restricted to $T(\to x)$ an embedding of $T(\to x)$ into $T'(\to x)$ for all $x\in \mathcal{L}_n$.   Let $g$ be an embedding of $T'$ into $T$ with $g(x)=x$ for all $x\in \mathcal{L}_{\geq n}$ and $g$ restricted to $T'(\to x)$ an embedding of $T'(\to x)$ into $T(\to x)$ for all $x\in \mathcal{L}_n$. Then $lev_T(x)=lev_{T'}(x)$ for all $x\in \mathcal{L}_{\geq n}(T)$.   Because $f$ is then an embedding of the rooted tree $T(\to x)$ into the rooted tree $T'(\to x)$ for every $x\in \mathcal{L}_n(T)$ and  $lev_T(x)=lev_{T'}(x)$, the embedding $f$ is level preserving on $T(\to x)$. It follows that $f$ is level preserving. Similarly, the embedding $g$ is level preserving.  

Let $h$ be an embedding of $T'$. Then $g\circ h\circ f$ is an embedding of $T$ and hence level preserving. The embeddings $f$ and $g$ and $g\circ h\circ f$ are level preserving and hence $h$ is level preserving.  Let $h$ be an embedding of $T'$ into $T$. Then $h\circ f$ is an embedding of $T$ and hence level preserving. Because $f$ is level preserving, the embedding $h$ is level preserving.     
\hfill $\Box$

\begin{lemma}\label{claim:rootedtoT}
Let $T\in \mathfrak{G}$.  Then $\vert twin(T(\to x))\vert \leq \vert twin(T)\vert$ for all $x\in V(T)$ and if the set $twin(T)$  is finite, then $\vert twin(T(\to x))\vert =1$ for every $x\in V(T)$.  
\end{lemma}
\vskip-1pt

\begin{proof} 
Let $x\in V(T)$ and $n:=lev_{T}(x)$ and $\kappa:=\vert twin(T(\to
x))\vert$ and let $(R^{\langle \alpha\rangle};\alpha\in\kappa)$ be an
enumeration of $twin(T(\to x))$. Let $Y:=\{y\in \mathcal{L}_n: T(\to
y)\equiv T(\to x)\}$.  For $\alpha\in\kappa$ let $T^{\langle
\alpha\rangle}$ be the tree obtained from $T$ by replacing for every
$y\in Y$ the tree $T(\to y)$ with the tree $R^{\langle
\alpha\rangle}$.  Then $T^{\langle \alpha\rangle}$ is equimorphic to
$T^{\langle \beta\rangle}$ for all $\alpha$ and $\beta$ in $\kappa$.
It follows from Claim \ref{claim:presindup} that every embedding of
$T^{\langle \alpha\rangle}$ is level preserving for every $\alpha\in
\kappa$.  Let $\alpha$ and $\beta$ in $\kappa$ and let $h$ be an
isomorphism of $T^{\langle \beta\rangle}$ to $T^{\langle
\alpha\rangle}$. It follows from Claim \ref{claim:presindup} that $h$
is level preserving.  Hence $h$ maps the set $Y$ onto the set $Y$ and
then induces an isomorphism of $T^{\langle \beta\rangle}(\to y)$ onto
$T^{\langle \alpha\rangle}(\to y)$ for every $y\in Y$. This implies that
$\alpha=\beta$.

If $\vert twin(T(\to x))\vert >1$ for some $x\in V(T)$, then $\vert
twin(T(\to x))\vert$ is infinite, according to Tyomkyn's Theorem
that the rooted tree alternative conjecture holds, and
hence $twin(T)$ is infinite.  
\end{proof}

Two twins $T$ and $T'$ are said to be {\em in position} if there are embeddings $f$
of $T$ to $T'$ and $g$ of $T'$ to $T$ for which $S_{g\circ f}=V(T)\cap
V(T')$ and $f(x)=x$ and $g(x)=x$ for all $x\in S_{g\circ f}$.  The
embeddings $f$ and $g$ are the {\em position embeddings}. Note that if
$T$ and $T'$ are in position with position embeddings $f$ and $g$,
then $T(\to x)\equiv T'(\to x)$ for all $x\in S_{g\circ f}$.

\begin{claim}\label{claim:normrep}
Let $T\in \mathfrak{G}$ and let  $T'$ be   a tree equimorphic to $T$. Then there exists a tree $T''$ isomorphic to $T'$ so that the trees $T$ and $T''$ are in position. 
\end{claim}
\vskip-1pt
\noindent
\noindent{\bf Proof of Claim \ref{claim:normrep}.}  
Let $\bar{T}$ be an isomorphic copy of $T'$ with $V(T)\cap V(\bar{T})=\emptyset$. There are then embeddings $\bar{f}$ of $T$ into $\bar{T}$ and $\bar{g}$ of $\bar{T}$ into $T$.  The embedding $\bar{g}\circ \bar{f}$ of $T$ is level preserving and hence has a set of fixed points $S_{\bar{g}\circ \bar{f}}$. Then $\bar{f}$ is an isomorphism of $S_{\bar{g}\circ \bar{f}}$ into $\bar{T}$ with $\bar{g}$ restricted to  $f[S_{\bar{g}\circ \bar{f}}]$ the inverse of $\bar{f}$. 

For every point $x\in S_{g\circ f}$ let $T_x=((T,
S_{\bar{g}\circ\bar{f}})(x), x)$. That is, the subtree of $T$ rooted
at $x$ which is the connected component after removing the set
$S_{\bar{g}\circ \bar{f}}\setminus \{x\}$ of vertices from
$T$. Similarly let $\bar{T}_{\bar{f}(x)}=((\bar{T},
\bar{f}[S_{\bar{g}\circ \bar{f}}])(\bar{f}(x)), \bar{f}(x))$. Let
\[
T''=\bigoplus_{x\in S_{\bar{g}\circ \bar{f}}}(\bar{T}_{\bar{f}(x)},\bar{f}(x)). 
\]
Then $T''$ is isomorphic to $\bar{T}$ and hence isomorphic to $T'$. 

Let $f$ be the embedding of $T$ into $T''$ for which $f(x)=x$ for all $x\in S_{\bar{g}\circ\bar{f}}$ and $f(z)=\bar{f}(z)$ for all $z\in V(T)\setminus S_{\bar{g}\circ\bar{f}}$. Let $g$ be the embedding of $T''$ into $T$ for which $g(x)=x$ for all $x\in S_{\bar{g}\circ\bar{f}}$ and $g(z)=\bar{g}(z)$ for all $z\in V(T'')\setminus S_{\bar{g}\circ\bar{f}}$. \hfill $\Box$

Let $\mathfrak{F}$ be the subclass of $\mathfrak{G}$ made of trees $T$
for which the set $twin(T)$ is finite. We have to prove that $\vert
twin(T)\vert=1$ for every tree in $\mathfrak{F}$.  It follows from
Claim~\ref{claim:normrep} that it suffices to prove that if $T\in
\mathfrak{F}$ and the pair $T$ and $T'$ are twins in position, then
$T'$ is isomorphic to $T$. If $T$ and $T'$ are in position with
position embeddings $f$ of $T$ into $T'$ and $g$ of $T'$ into $T$ and
if $x\in S_{g\circ f}$, then $T(\to x)\equiv T'(\to x)$ and hence
$T(\to x)\simeq T'(\to x)$, because $\vert twin(T(\to x))\vert=1$.

Let $T\in \mathfrak{F}$ and $T'$ be twins in position with position
embeddings $f$ of $T$ into $T'$ and $g$ of $T'$ into $T$.  Let
$Y\subseteq S_{g\circ f}$ so that $x\not \in C_y$ for all $x$ and $y$
in $Y$ with $x\not=y$.  Let the tree {\em obtained from $T'$ by
replacing $T'(\to Y)$ with $T(\to Y)$} be the tree:
\[
repl(T',T(\to Y)): =\bigoplus_{y\in Y, T'\setminus T'(\to -Y)}T(\to y).  
\]

\begin{claim}\label{claim:refatz}
Let $T\in \mathfrak{F}$ and $T'$ be twins in position with position embeddings $f$ of $T$ into $T'$ and $g$ of $T'$ into $T$.  Let $Y\subseteq S_{g\circ f}$ so that $x\not \in C_y$ for all $x$ and $y$ in $Y$ with $x\not=y$. 

Then $T$ and $T''=repl(T',T(\to Y))$ are twins in position with position embeddings $\bar{f}$ and $\bar{g}$ for which $\bar{f}(z)=f(z)$ for all $z\in T\setminus T(\to -Y)$ and $\bar{f}(z)=z$ for all $z\in T(\to Y)$ and for which $\bar{g}(z)=g(z)$ for all $z\in T'\setminus T'(\to -Y)$ and $\bar{g}(z)=z$ for all $z\in T'(\to Y)$. There exists, for any choice of isomorphism $\iota_y$ of $T'(\to y)$ to $T(\to y)$ for $y\in Y$, an isomorphism $\beta$ of $T'$ to $T''$ with $\beta(z)=z$ for all $z$ in $T'\setminus T'(\to -Y)$ and $\beta(z)=\iota(z)$ for $z\in T'(\to y)$ and $y\in Y$. 
\end{claim}
\noindent{\bf Proof of Claim \ref{claim:refatz}.}  
We have $V(T)\cap V(repl(T',T(\to X))=V(T)\cap V(T')\cup T(\to X)$ and hence both $\bar{f}$ and  $\bar{g}$ fix every vertex in $V(T)\cap V(repl(T',T(\to X))$. Also $\bar{f}$ and $\bar{g}$ are embeddings. 
\hfill $\Box$

\begin{definition}\label{defin:compl}

Let $T\in \mathfrak{E}$ and let $x\in V(T)$ and $X\subseteq V(T)$.  

Let $T(\to -x)=T(\to x)\setminus \{x\}$ and  $T(\to X)=\bigcup_{x\in X}T(\to x)$ and  $T(\to -X)=\\ T(\to X)\setminus X$ and $N_T(x)=\{y\in (T): y\to x)\}$. For  $y^\ast\in N_T(x)$ let:  
\[
\text{$T\lfloor y^\ast,x\rfloor=\negthickspace\negthickspace\negthickspace\negthickspace\bigcup_{y\in (N_T(x)\setminus \{y^\ast\}}\negthickspace\negthickspace\negthickspace\negthickspace\negthickspace\negthickspace V(T(\to y))$\thickspace\thickspace and \thickspace\thickspace $T\lceil y^\ast,x\rceil= T\setminus T\lfloor y^\ast,x\rfloor$}.      
\]

\end{definition}

\begin{claim}\label{claim:refatzb}

Let $T\in \mathfrak{F}$ and $T'$ be twins in position with position embeddings $f$ of $T$ into $T'$ and $g$ of $T'$ into $T$.  Let   $x\in S_{g\circ f}$ and  $y^\ast \in N_T(x)\cap S_{g\circ f}$.  If $T(\to y)\subseteq S_{g\circ f}$ for all $y\in N_T(x)\cap S_{g\circ f}$ then: 
 
There exists a pair of twins $R$ and $R'$ in position with $\{y^\ast,x\}\subseteq V(R)\cap V(R')$ and with: 
\[
\text{$T\lceil y^\ast,x\rceil=R\lceil y^\ast,x\rceil$  and  $T'\lceil y^\ast,x\rceil=R'\lceil y^\ast,x\rceil$ and $R(\to x)=R'(\to x)$.}
\]
There is an isomorphism $\delta$ of $T$ to $R$ with $\delta(z)=z$  for all vertices $z$ in $T\lceil y^\ast,x\rceil$. There is an isomorphism $\delta'$ of $T'$ to $R'$ with  $\delta'(z)=z$  for all vertices $z$ in $T'\lceil y^\ast,x\rceil$. The isomorphism $\delta$ restricted to $T(\to x)$ is an isomorphism $\epsilon$ of $T(\to x)$ to $R(\to x)$ with $\epsilon(z)=z$ for all vertices $z$ in $T(\to y^\ast)$.  The isomorphism $\delta'$ restricted to $T'(\to x)$ is an isomorphism $\epsilon'$ of $T'(\to x)$ to $R'(\to x)$ with $\epsilon'(z)=z$ for all vertices $z$ in $T'(\to y^\ast)$. 

\end{claim}
\vskip-1pt
\noindent{\bf Proof of Claim \ref{claim:refatzb}.}  If $y$ and $z$ are vertices in $N_T(x)$ let $y\cong z$ if the rooted tree $T(\to y)$ is isomorphic to the rooted tree $T(\to z)$.   If $y$ and $z$ are vertices in $N_{T'}(x)$ let $y\cong z$ if the rooted tree $T'(\to y)$ is isomorphic to the rooted tree $T'(\to z)$. If $y\in N_T(x)$ and $z\in N_{T'}(x)$ let $y\cong z$ if the rooted tree $T(\to y)$ is isomorphic to the rooted tree $T'(\to z)$. Because $T(\to x)$ and $T'(\to x)$ are isomorphic there exist partitions $\mathcal{P} =(P_{k};k\in K )$ of $N_T(x)$ and $\mathcal{P}' =(P'_{k};k\in K )$ of $N_{T'}(x)$  into $\cong$-equivalence classes so that $y\cong y'$  for every $k\in K $ and every $y\in P_{k}$ and $y'\in P'_{k}$ and so that $\vert P_{k}\vert =\vert P'_{k}\vert$ for every $k\in K $.  For every $k\in K $ let $\bar{P}_{k}=P_{k}\setminus(P_{k}\cap S_{g\circ f})$ and let  $\bar{P'}_{k}=P'_{k}\setminus(P_{k}\cap S_{g\circ f})$.  Note  that $P_{k}\cap S_{g\circ f}=P'_{k}\cap S_{g\circ f}$ for every $k\in K $.

Let  $k\in K $. It follows from $\vert P_k\vert =\vert P'_k\vert$ and $P\cap S_{g\circ f}= P'_k\cap S_{g\circ f}$ that  if  $\vert \bar{P}_{k}\vert \not=\vert \bar{P'}_{k}\vert$, then $\vert \bar{P}_{k}\vert \leq \vert P_{k}\cap S\vert$ and $\vert \bar{P'}_{k}\vert \leq \vert P'_{k}\cap S\vert$ and $P_{k}\cap S_{g\circ f}=P'_{k}\cap S_{g\circ f}$ is infinite. Let $I=\{k\in K : \vert \bar{P}_{k}\vert \not=\vert \bar{P'}_{k}\vert\}$ and $J=\{k\in K : \vert \bar{P}_{k}\vert =\vert \bar{P'}_{k}\vert\}$. Let $W=T(\to \bigcup_{k\in J}\bar{P}_k)$ and $W'=T'(\to \bigcup_{k\in J}\bar{P}'_k)$.  Note that there is an isomorphism, say $\iota$,  of $W'$ to  $W$. Let $(W\cup \{x\},x)$ be the tree rooted at $x$ for which $(W\cup x,x)\setminus x=W$.   Let $\gamma $ be an embedding of $T$ so that for every $k\in I$:   
\begin{enumerate}
\item For every $k\in I$,  the restriction of $\gamma $ to $P_k$ is an injection of $P_k$ onto $(P_k\cap S)\setminus \{y^\ast\}$. 
\item For every $k\in I$,  the restriction of $\gamma $ to $T(\to y)$ for $y\in P_k$ is an isomorphism of $T(\to y)$ to $T(\to \gamma (y))$.  
\item If  $z\in  W$ or $z\in T(\to y^\ast)$, then $\gamma (z)=z$.  
\item If $z$ is a vertex in $T\lceil y^\ast,x\rceil$, then $\gamma(z)=z$.  
\end{enumerate}
Let $\gamma' $ be an embedding of $T'$ so that for every $k\in I$:   
\begin{enumerate}
\item For every $k\in I$, the restriction of $\gamma $ to $P'_k$ is an injection of $P'_k$ onto $(P'_k\cap S)\setminus \{y^\ast\}$. 
\item For every $k\in I$, the restriction of $\gamma $ to $T'(\to y)$ for $y\in P'_k$ is an isomorphism of $T'(\to y)$ to $T'(\to \gamma (y))$.  
\item If  $z\in  W'$ or $z\in T'(\to y^\ast)$, then $\gamma'(z)=z$.  
\item If $z$ is a vertex in $T'\lceil y^\ast,x\rceil$, then $\gamma'(z)=z$.   
\end{enumerate}
For $R:=\gamma [T]$ and $R'':=\gamma' [T']$ and $U=\bigcup_{k\in
I}\bar{P}_k$ and $U'=\bigcup_{k\in I}\bar{P'}_k$ we have $R=T\setminus
T(\to U)$ and $R''=T'\setminus T'(\to U')$. Then {$T\lceil
y^\ast,x\rceil=R\lceil y^\ast,x\rceil$ and $T'\lceil
y^\ast,x\rceil=R''\lceil y^\ast,x\rceil$.  Let $\bar{f}$ be the
embedding of $R$ into $R''$ given by: $\bar{f}(z)=f(z)$ for all $z$ in
$R$ with $z\in T\setminus W $.  For $z\in W $ let
$\bar{f}(z)=\iota^{-1}(z)$.  Let $\bar{g}$ be the embedding of $R''$
into $R$ given by: $\bar{g}(z)=g(z)$ for all $z$ in $R''$ with $z\in
R''\setminus W'$.  For all $z\in W'$ let $\bar{g}(z)=\iota(z)$.  It
follows that $S_{\bar{g}\circ \bar{f}}=S_{g\circ f}$ and that $R$ and
$R''$ are in position with position embeddings $\bar{f}$ and
$\bar{g}$.

Next we identify the vertices $z$  in $R'' $ with the vertices $\iota(z)$. Let $R'=(R''\setminus W' )\oplus_x  (W\cup \{x\},x) = repl(R'', T(\to \bigcup_{k\in J}\bar{P}_k))$ and use Claim \ref{claim:refatz}.   Then there exists an isomorphism $\nu$ of $R''$ to $R'$ given by $\nu(z)=z$ for all vertices $z$ in $R''\setminus W' $ and $\nu(z)=\iota(z)$ for $z$ in $W'$. Then $T'\lceil y^\ast,x\rceil=R''\lceil y^\ast,x\rceil=R'\lceil y^\ast,x\rceil$ and $R(\to x)=R'(\to x)$.

Let $\delta $ be  the isomorphism $\gamma $ of $T$ to $R$ and let $\delta' $ be the isomorphism $\nu\circ \gamma $ of $T'$ to $R'$. Then $\delta(z)=z$ for all $z\in T(\to y^\ast)$ and because $\delta(x)=x$ and $\delta$ is an isomorphism the restriction of $\delta$ to $T(\to x)$ is an isomorphism of $T(\to x)$ to $R(\to x)$. Similar of $\delta'$.   Let $\check{f} (z)=\bar{f}(z)$ for those $z$ in $R$ with  $z$ in $R\setminus W$ and $\check{f}(z)=z$ for $z$ in $W$.   Let $\check{g} (z)=\bar{g}(z)$ for those $z$ in $R'$ with $z$ in $R'\setminus W$ and $\check{g} (z)=z$ for $z$ in $W$.   Then   $R$ and $R'$ are in position with position embeddings $\check{f}$ and $\check{g}$.     \hfill $\Box$

\begin{claim}\label{claim:refatzup}
Let $T\in \mathfrak{F}$ and $T'$ be twins in position. Then there exists a tree $R$ for which both $T$ and $T'$ are isomorphic to $R$. 
\end{claim}
\noindent{\bf Proof of Claim \ref{claim:refatzup}.}  
Let $f$ of $T$ to $T'$ and $g$ of $T'$ to $T$ be the position embeddings of $T$ and $T'$. Let the origin $o$ of the level function of $T$ and $T'$ be in $S_{g\circ f}$ and let $x_n:=o\boxplus n$ for every $n\in \N$.  We will construct an $\omega$-sequence $\mathfrak{R}$ of pairs of twins in position with: 
$\mathfrak{R}=(T\simeq R^{\langle 0\rangle},T'\simeq R'^{\langle 0\rangle}), (R^{\langle 1\rangle},R'^{\langle 1\rangle}), (R^{\langle 2\rangle},R'^{\langle 2\rangle}), \dots$  and  position embeddings $f_n$ of $R^{\langle n\rangle}$ into $R'^{\langle n\rangle}$ and $g_n$ of $R'^{\langle n\rangle}$ into $R^{\langle n\rangle}$ and with isomorphism  $\gamma_{n+1}$ of $R^{\langle n+1\rangle}$ to $R^{\langle n\rangle}$ and $\gamma'_{n+1}$ of $R'^{\langle n+1\rangle}$ to $R'^{\langle n\rangle}$ so that for all $n\in \N$: 
\begin{enumerate} 
\item The path $C_o\subseteq R^{\langle n\rangle}$ and $C_o\subseteq R'^{\langle n\rangle}$.  
\item $R^{\langle n\rangle}(\to x_n)=R'^{\langle n\rangle}(\to x_n)\subseteq S_{g_n\circ f_n}$.
\item $R^{\langle n\rangle}\setminus (R^{\langle n\rangle}(\to -x_n))=T\setminus (T(\to -x_n))$. 
\item $\gamma_{n+1}(z)=z$ for all $z$ in $R^{\langle n\rangle}\lceil x_n,x_{n+1}\rceil$. 
\item $\gamma'_{n+1}(z)=z$ for all $z$ in $R'^{\langle n\rangle}\lceil x_n,x_{n+1}\rceil$. 
\end{enumerate} 

We proceed by recursion  on $n$. Let $R^{\langle 0\rangle}=T$ and $P^{\langle 0\rangle}=T(\to x_0)=Q^{\langle 0\rangle}$ and let $\alpha_0$ be the identity map on $T(\to x_0)$ and $\alpha'_0$ an isomorphism  of $T'(\to x_0):=Q'^{\langle 0\rangle}$ to $T(\to x_0)$. Let $R'^{\langle 0\rangle }=repl(T', T(\to \{x_0\}))$  and $\gamma_0$ the identity map on $T$ and $\gamma'_0$ the isomorphism $\beta$ given by Claim \ref{claim:refatz} with $\iota$ for $\alpha_0'$.  

Let $Y_{n}=(N_{R^{\langle n\rangle}}(x_{n+1})\cap S_{g_n\circ f_n})\setminus \{x_n\}$ and $\iota_y$ an isomorphism of $R'^{\langle n\rangle}(\to y)$ to $R^{\langle n\rangle}(\to y)$ for every $y\in Y_n$.  Let  $R''^{\langle n\rangle}=repl(R'^{\langle n\rangle}, R^{\langle n\rangle}(\to Y_{n}))$ and $\beta$ the isomorphism of $R'^{\langle n\rangle}$ to $R''^{\langle n\rangle}$ given by Claim \ref{claim:refatz}. Note that then $R''^{\langle n\rangle}(\to y)=R^{\langle n\rangle}(\to y)$  and hence $R^{\langle n\rangle}(\to y)\subseteq S_{\bar{g}\circ \bar{f}}$ for all $y\in N_{R^{\langle n\rangle}}(x_{n+1})\cap S_{\bar{g}\circ \bar{f}}$. Note also that $\beta$ restricted to $R'^{\langle n\rangle}(\to x_{n+1})$  is an isomorphism of $R''^{\langle n\rangle}(\to x_{n+1)}$ which fixes  $R'^{\langle n\rangle}(\to x_{n})$.

We apply Claim \ref{claim:refatzb}  with $R^{\langle n\rangle}$ for $T$ and $R''^{\langle n\rangle}$ for $T'$ and  $x_{n+1}$ for $x$ and $x_n$ for $y^\ast$ to obtain the pair $R^{\langle n+1\rangle}$ and $R'^{\langle n+1\rangle}$ of twins in position for which $R^{\langle n+\rangle}(\to x_{n+1})=R'^{\langle n+1\rangle}(\to x_{n+1})$ and: 
\[
\text{$R^{\langle n\rangle}\lceil x_n,x_{n+1}\rceil =R^{\langle n+1\rangle}\lceil x_n,x_{n+1}\rceil $ and   $R'^{\langle n\rangle}\lceil x_n,x_{n+1}\rceil =R'^{\langle n\rangle}\lceil x_n,x_{n+1}\rceil $.}  
\]
Let 
\[
P^{\langle n+1\rangle}=R^{\langle n+1\rangle}(\to x_{n+1})\setminus R^{\langle n+1\rangle}(\to -x_n)
\]
and note that  
\[
\text{$R^{\langle n+1\rangle}=R^{\langle n\rangle}\lceil x_n,x_{n+1}\rceil\oplus_{x_{n+1}}P^{\langle n+1\rangle}$ and $R'^{\langle n+1\rangle}= R'^{\langle n\rangle}\lceil x_n,x_{n+1}\rceil\oplus_{x_{n+1}}P^{\langle n+1\rangle}$}
\] 
and that
\begin{align*}
&T(\to x_{n+1})\setminus T(\to -x_n)=  R^{\langle n\rangle}(\to x_{n+1})\setminus R^{\langle n\rangle}(\to x_{n}):= Q^{\langle n+1\rangle},  \\
& T'(\to x_{n+1})\setminus T'(\to -x_n)=  R'^{\langle n\rangle}(\to x_{n+1})\setminus R'^{\langle n\rangle}(\to x_{n}):= Q'^{\langle n+1\rangle}. 
\end{align*}
The isomorphism $\epsilon$ mapping $R^{\langle n\rangle}(\to x_{n+1})$ to $R^{\langle n+1\rangle}(\to x_{n+1})$ given by Claim~\ref{claim:refatzb} fixes the rooted tree $R^{\langle n\rangle}(\to x_{n})$ and hence induces an isomorphism $\alpha_{n+1}$ of $Q^{\langle n+1\rangle}$ to $P^{\langle n+1\rangle}$. Similarly $\epsilon'\circ \beta$ induces  an isomorphism of $R'^{\langle n\rangle}(\to x_{n+1})$ to $R'^{\langle n+1\rangle}(\to x_{n+1})$, which fixes the rooted tree $R'^{\langle n\rangle}(\to x_{n})$ and hence induces an isomorphism $\alpha'_{n+1}$ of $Q'^{\langle n+1\rangle}$ to $P^{\langle n+1\rangle}$.

The tree $T$ can be written as a sum $T=\bigoplus_{C_o}Q^{\langle n\rangle}$. That is, the tree $T$ consists of the one-way infinite path $C_o$ with a rooted tree $Q^{\langle n\rangle}$ attached at $x_n$ for every $n\in \N$. Similarly $T'=\bigoplus_{C_o}Q'^{\langle n\rangle}$. Let $R=\bigoplus_{C_o}P^{\langle n\rangle}$. Then $\bigcup_{n\in \N}\alpha_n$ is an isomorphism of $T$ to $R$ and  $\bigcup_{n\in \N}\alpha'_n$ is an isomorphism of $T'$ to $R$. \hfill $\Box$


\subsubsection{\noindent{\bf Proof of Proposition \ref{claim:finalindexpr}.}}  
Assume that the set $twin(T)$ is not infinite. It follows from Lemma
\ref{claim:rootedtoT} that then $\vert twin(T(\to x))\vert =1$ for all
$x\in V(T)$. Because of Claim \ref{claim:normrep} it suffices to show
that if $T'$ is a twin of $T$ and $T$ and $T'$ are in position, then
$T$ and $T'$ are isomorphic, which indeed is the case due to Claim
\ref{claim:refatzup}. \hfill $\Box$
%
%

\vskip 20pt
\noindent
{\bf Case \textbf{2}}: Let $T\in \mathfrak{E}$. There exists an embedding which does not preserve $lev_T$, that is the end is not semi-rigid. Such an embedding has a positive period.    
We have to deal with two subcases. 

\vskip 10pt
\noindent
Subcase \textbf{2.1}: The end is not regular.

\begin{lemma}\label{lem:almost contained} 
Let $T\in \mathfrak{E}$ and let $f$ be an embedding of $T$ with period $k>0$, then $\vert twin(T)\vert\geq 2^{\aleph_0}$.   
\end{lemma}
\begin{proof} 
For $u \in S_f$ and $x\in e(u)$, set $T(f, x):= T<e(u), x>$
Since the end $e$ of $T$ is no regular, for  every $u\in S_f$, the ray  $e(u)$ contains  an infinite subset $L$ such  that $T(f,x)$ is not equimorphic to $T(f,f(x))$ for all $x\in L$.

There exists a set $\mathcal{A}$ of  $2^{\aleph_0}$ infinite pairwise almost disjoint subsets of $L$.    If $S_f$ induces a two-way infinite path enumerate $S_f$ naturally as $\{x_i: i\in \Z\}$. If $S_f$  induces a one-way infinite path enumerate $S_f$ naturally as $\{x_i: i\in \N\}$. Then $f(x_n)=x_{n+k}$ for every $x_n\in S_f$.   

Let $A\in \mathcal{A}$.  For each $x_n\in A$, let $T_A$ be the tree obtained from $T$ by replacing each of the rooted trees $T(f,x_{n+k})$  by  the rooted tree  $T(f,x_n)$ and let $T(f,x_{n+k})$ unchanged if $n\not\in A$.   According to Lemma \ref{lem:sumoverpath},  the map $f$ induces an embedding of $T$ into $T_A$  because $f$ maps the rooted tree $T(f, x_n)$ into the rooted tree $T(f, x_{n+k})$, whereas the identity of $T_A$ is   an embedding of $T_A$ into $T$.  Hence $T_A$ and $T$ are equimorphic. Say that two sets $A$ and $A'$ are equivalent and set $A\simeq A'$ if $T_A$ is isomorphic to $T_{A'}$. If every isomorphism class of sets in $\mathcal{A}$ is countable  we obtain $2^{\aleph_0}$ pairwise non isomorphic trees. Hence it suffices to prove that every isomorphism class of sets in $\mathcal   {A}$ is countable.  

Suppose that there is an uncountable family of subsets
$(A_{\alpha})_{\alpha<\omega_{1}}$ of $\mathcal{A}$ which are all in
the same isomorphism class.  For $\alpha<\omega_1$ let $h_{\alpha}$ be
an isomorphism of $T_{A_{0}}$ onto $T_{A_{\alpha}}$. The map
$h_{\alpha}\circ f$ is an embedding of $T$. For every
$\alpha<\omega_1$ choose a vertex $x_{n_\alpha}\in S_f\cap
S_{h_\alpha\circ f}$ which is not an endpoint of $S_f$ nor of
$S_{h_\alpha\circ f}$. There is an $x_{\overline{n}}\in S_f$ so that
$x_{\overline{n}}=x_{n_\alpha}$ for uncountably many $\alpha$. Let
$C:=e({x_{\overline{n}}})$.  Each of those embeddings $h_\alpha\circ
f$ has an period which is a number in $\N$ and hence there is an
uncountable set $\mathcal{B}$ of $\alpha\in \omega_1$ and a number
$l\in \N$ so that $x_{\overline{n}}\in S_{h_\alpha\circ f}$ and the
period of $h_\alpha\circ f$ is the same number $l$ for all $\alpha\in
\mathcal{B}$. Let $\alpha$ and $\beta$ in $\mathcal{B}$. Then
$h_\alpha\circ f$ and $h_\beta\circ f$ have equal restrictions to $C$.

Let $n>\overline{n}+l-k$ with $n\in A_\alpha$ but $n\not\in
A_\beta$. Then the rooted tree $T(C,x_n)$ is isomorphic to the tree
$T_{A_\alpha}(C,x_{n+k})$ and is embedded by $f$ into the tree
$T(C,x_{n+k})$ which in turn is isomorphic to
$T_{A_\beta}(C,x_{n+k})$. But the tree $T(C_{x_{n+k}})$ can not be
embedded into $T(C,x_n)$, for otherwise the trees $T(C_{x_{n+k}})$
and $T(C,x_n)$ would be equimorphic. Hence $T_{A_\alpha}(C,x_{n+k})$
can not be embedded into $T_{A_\beta}(C,x_{n+k})$ and are therefore
not isomorphic. The embedding $f$ maps the tree $T(C,x_{n-l+k})$ into
the tree $T_{A_0}(C, x_{n-l+2k})$. Because $h_\beta$ is an isomorphism
of $T_{A_0}$ to $T_{A_\beta}$, the trees $T_{A_0}(C, x_{n-l+2k})$ and
$T_{A_\beta}(C,x_{n+k})$ are isomorphic. Because $h_\alpha$ is an
isomorphism of $T_{A_0}$ to $T_{A_\alpha}$ the trees $T_{A_0}(C,
x_{n-l+2k})$ and $T_{A_\alpha}(C,x_{n+k})$ are isomorphic. This implies 
the contradiction that $T_{A_\beta}(C,x_{n+k})$ and
$T_{A_\alpha}(C,x_{n+k})$ are isomorphic.  \end{proof}

\vskip 10pt
\noindent
Subcase \textbf{2.2}: The end is regular.  We also assume in this Subcase that $T$ is not the one-way infinite path.

Let $v\in \check{E}$ and $f$ be a proper embedding with $v$ the
endpoint of $\check{S}_f$.  We are going to construct infinitely many
subtrees $Twin_n$ of $T$, each of them a twin of $T$. Let $k$ be the
period of $f$ and note that $k$ is a multiple of $\mathbf{d}$. For
$v\in V(T)$ we will denote by $C_v$, instead of $e(v)$, the ray
originating at $v$.

In the case that $S_f$ is a one-way infinite path, let $u$ be the
endpoint of $S_f$ and $v=u\boxplus l$.  Because $T$ is not the one-way
infinite path there is a number $0\leq i<k$ for which the tree $T(C_v,
v\boxplus i)=T(C_u,u\boxplus(l+i)$ contains at least one vertex
different from its root $v\boxplus i$. Let $X_n=\{u\boxplus j: 0\leq
j\leq l+i+1+3nk\}$ with $n\in \N$. The subtree $Twin_n$ for $n\in \N$
of $T$ is obtained from $T$ by removing all trees $T(\to x)$ from $T$
for which $y\to x\in X_n$ and $x\not\in C_u$. Then $Twin_n(C_u, x)$ is
the singleton $x$ for every vertex $x\in X$.  The identity map on
$Twin_n$ embeds the tree $Twin_n$ into $T$ and $f^m$ maps $T$ into
$Twin_n$ for $m>l+i+1+3n$.  Hence $Twin_n$ is a twin of $T$ for every
$1\leq n\in \N$. The embedding $f$ restricted to $Twin_n$ is an
embedding of $Twin_n$ and the vertex $v_n:=u\oplus(l+i+1+3nk)\in
\check{E}(Twin_n)$, is the endpoint of the path
$\check{S}_f(Twin_n)$.  It follows from Proposition
\ref{Fact:finitdist} that the origin $o_n$ of the tree $Twin_n$ is of
the form $v_n\oplus j_n$ with $0\leq j_n\leq 2\mathbf{d}\leq 2k$. It
follows that the level of $o_n$ as a vertex of $T$ is strictly smaller
than the level of $o_m$ for $n<m$ and hence that $o_m\in C_{o_n}$ for
$n<m$. Hence if there is an isomorphism, say $h$, of $Twin_n$ to
$Twin_m$ with $n<m$, then $h$ translates the path $C_{o_n}$ forward
onto the path $C_{o_m}$ mapping $o_n$ to $o_m$ according to
Proposition \ref{Fact:finitdist}. Hence $lev(h(x))>lev(x)$ for all
vertices in $Twin_n$. Implying that there is no vertex $x\in
V(Twin_n)$ with $h(x)=u\in V(Twin_m)$.

In the case that $S_f$ is a two-way infinite path let $u$ be the vertex in $S_f$ with $u\boxplus k=v$ and denote by $C$ the two-way infinite path $S_f$. Let $X_n=\{v\oplus j: \mathbf{d}\leq j\leq 3nk\}$. For $x=v\oplus j\in X_n$ let $R^{\langle x\rangle}$ be the tree $T(C,u\oplus i)$  with $0\leq i<k$ and $i$ congruent to $j$ modulo $k$. We obtain the tree $Twin_n$ from the tree $T$ by replacing for each $x\in X_n$ the rooted tree $T(C,x)$ by the rooted tree $R^{\langle x\rangle}$. That is if $x\in X_n$, then $Twin_n(C,x)=R^{\langle x\rangle}$ and if $x\not\in X$ the $Twin_n(C,x)=T(C,x)$. (Note that there is no embedding of $Twin_n$ which moves the path from $v$ to $v\boxplus k-1$ into a path in $X$.) It follows that the embedding $f^{3nk}:=g_n$ of $T$ is also an embedding of $Twin_n$ and that $C=S_{g_n}(Twin_n)$. The identity maps $Twin_n$ into $T$ and $g_n$ maps $T$ into $Twin_n$ and hence $Twin_n$ is a twin of $T$ for every $n\in \N$. The vertex $v_n=v\boxplus 3nk$ is the endpoint of the path $\check{S}_{g_n}(Twin_n)$. Implying that $lev(o_n)+k\leq lev(o_m)$  for the origin $o_n$ of $Twin_n$ and the origin $o_m$ of $Twin_m$   if $n<m$.

Assume that there is an isomorphism $h$ of $Twin_n$ to $Twin_m$ with
$1\leq n<m$. Then $h$ maps $C_{o_n}$ onto $C_{o_m}$ and hence
$lev(h(x))- lev(x)=:l\geq k$ for every $x\in V(Twin_n)$. It is not
possible that $h$ maps $v$ into $C$ and hence $h$ maps $v$ into one of
the rooted trees $Twin_m(C,x)$ with $x\in X$. Let $r\in C$ be the
vertex of smallest level with $h(r)\in C$. Then $r$ is a vertex on the
oriented path from $v\boxplus 1$ to $o_n$. Let $x\in C$ with $x\to r$.
Then $h(r)=r\boxplus l$ and $h(x)$ is a vertex in $Twin_m(C, h(x))$
and hence every vertex in $Twin_n(\to x)$ is mapped by $h$ into the
tree $Twin_m(C, h(x))$. Let $y$ be the vertex in $Twin_n$ with
$h(y)=x$. Then $y$ is not a vertex in $Twin_n(\to x)$ and hence the
vertex, say $z$, of smallest level in $C_y$ and $C$ has to be a vertex
in $C_r(Twin_n)$ and hence $y\in Twin_n(C,z)$. But then $h(y)\in
Twin_m(C,z\boxplus l)$. Implying that $h(y)\not=x$ and hence that $h$
could not have been an isomorphism. This completes Case 2.

We are now ready to complete the proof of Theorem \ref{main2}.

\subsection{\bf Proof of Theorem \ref{main2}.}
 Recall that stable trees are characterized by four types of subgraphs that are preserved by every embedding.


$(i)$ In this case, $T$ may contain a two way infinite path or a non-regular end preserved by every embedding. In the first case, the conclusion follows from Theorem \ref{lemma:paths}. In the second case, $T$ contains a path preserved by every embedding (Theorem \ref{main1-2}). The conclusion follows by Subcase 2.2. \\
$(ii)$  This case follows from the fact that embeddings of a locally finite trees are automorphism provided they fix a vertex. \\
$(iii)$ This case follows from  Lemma \ref {claim:rootedtoT}. \\
$(iv)$ This case follows from  Lemma \ref{lem:almost contained}.

\end{document}